\documentclass[a4paper,11pt,reqno]{amsart}
\usepackage{amssymb,mathtools}
\usepackage[margin=2.6cm]{geometry}
\usepackage[T1]{fontenc}
\usepackage{lmodern}
\usepackage{accents}
\usepackage{tikz-cd}
\allowdisplaybreaks[1]

\numberwithin{equation}{section}

\newtheorem{thm}{Theorem}[section]
\newtheorem{prop}[thm]{Proposition}
\newtheorem{lem}[thm]{Lemma}

\theoremstyle{definition}
\newtheorem{defn}[thm]{Definition}
\newtheorem{assump}[thm]{Assumption}
\theoremstyle{remark}
\newtheorem{rem}[thm]{Remark}

\makeatletter
\newcommand*{\BigCat}[2]{%
  \expandafter\ifx\expandafter\relax\detokenize{#1}\relax\expandafter\@firstoftwo\else\expandafter\@secondoftwo\fi
    {\mathcal{C}}%
    {\presubscript{#1}{\mathcal{C}}_{#2}}%
}
\newcommand*{\Sbimod}[2]{%
  \expandafter\ifx\expandafter\relax\detokenize{#1}\relax\expandafter\@firstoftwo\else\expandafter\@secondoftwo\fi
    {\mathcal{S}}%
    {\presubscript{#1}{\mathcal{S}}}%
  \expandafter\ifx\expandafter\relax\detokenize{#2}\relax\expandafter\@gobble\else\expandafter\@firstofone\fi
    {_{#2}}%
}
\makeatother

\DeclareMathOperator{\Hom}{Hom}
\DeclareMathOperator{\End}{End}
\DeclareMathOperator{\supp}{supp}

\DeclareMathOperator{\grk}{grk}

\DeclareMathOperator{\Aut}{Aut}

\DeclareMathOperator{\Parity}{Parity}
\DeclareMathOperator{\For}{For}
\DeclareMathOperator{\Average}{Av}
\DeclareMathOperator{\Infl}{Infl}
\DeclareMathOperator{\diag}{diag}
\DeclareMathOperator{\sgn}{sgn}
\newcommand*{\Z}{\mathbb{Z}}
\newcommand*{\Q}{\mathbb{Q}}
\newcommand*{\C}{\mathbb{C}}
\newcommand*{\presubscript}[2]{{}_{#1}#2}
\newcommand*{\presupscript}[2]{{}^{#1}#2}
\newcommand*{\id}{\mathrm{id}}

\newcommand*{\ch}{\mathrm{ch}}

\newcommand*{\triv}{\mathrm{triv}}

\title{Singular Soergel bimodules for realizations}
\author{Noriyuki Abe}
\address{Graduate School of Mathematical Sciences, the University of Tokyo, 3-8-1 Komaba, Meguro-ku, Tokyo 153-8914, Japan.}
\email{abenori@ms.u-tokyo.ac.jp}
\begin{document}
\begin{abstract}
Williamson defined the category of singular Soergel bimodules attached to a reflection faithful representation of a Coxeter group.
We generalize this construction to more general realizations of Coxeter groups.
\end{abstract}
\maketitle

\section{Introduction}
A Hecke category, by which we mean a good categorification of the Hecke algebra, is an important object in representation theory.
The first Hecke category that appeared in representation theory is the category of Borel subgroup equivariant semisimple complexes on the flag variety and it plays an important role in the proof of the Kazhdan-Lusztig conjecture.
Soergel~\cite{MR2329762} introduced another Hecke category.
His category is not geometric, it is algebraic.
Soergel also proved that his category is equivalent to the category of semisimple complexes when the field is of characteristic zero.
This equivalence is used to prove the Koszul duality of category $\mathcal{O}$ for reductive Lie algebras~\cite{MR1322847}.
There are also singular versions.
Instead of considering Borel equivariant semisimple complexes on the flag variety, we consider semisimple complexes on a generalized flag variety which is equivariant under the action of a parabolic subgroup.
A singular version of Soergel's category is introduced by Williamson~\cite{MR2844932}.

To prove that these categories are Hecke categories, we need some assumptions.
For the category of equivariant semisimple complexes, the characteristic of the coefficient field has to be zero.
For Soergel's category, Soergel needed a representation of the Weyl group and he assumed that this is reflection faithful.
After those theories appeared, some new constructions which are the generalizations of these categories appeared.
As a generalization of the category of semisimple complexes, Juteau-Mauter-Williamson~\cite{MR3230821} introduced the category of parity complexes and this gives the Hecke category for the coefficient field of any characteristic.
Note that this can be used to define the singular version.
A generalization of Soergel's category was introduced in \cite{MR4321542} for the regular case and with \cite{arXiv:2012.09414_accepted} it was proved that this category gives a Hecke category if it is attached to the root system.
There is also another Hecke category for the regular case defined by Elias-Williamson~\cite{MR3555156} and this also works well if it is attached to the root system~\cite{MR4668478,arXiv:2302.14476}.

So the natural question is to construct a category that generalizes the category of singular Soergel bimodules introduced by Williamson.
A partial answer was given in \cite{MR4620135}.
In \cite{MR4620135} we defined a category that generalizes the category of Soergel bimodules which is singular only from the left.
The aim of this paper is to give a category that generalizes Williamson's theory, namely the category of Soergel bimodules which are singular from both sides.

To give the theorem more precisely, let $(W,S)$ be a Coxeter system and $\mathcal{H}$ the Hecke algebra.
This is a $\Z[v,v^{-1}]$-algebra with a basis $\{H_{w}\mid w\in W\}$ where $v$ is an indeterminate.
Fix subsets $S_{1},S_{2}\subset S$ such that the group $W_{S_{i}}$ generated by $S_{i}$ is finite for $i = 1,2$.
Let $\mathcal{H}_{S_{i}}$ be the Hecke algebra of $(W_{S_{i}},S_{i})$.
Then this is a subalgebra of $\mathcal{H}$.
Let $\triv_{S_{i}}\colon \mathcal{H}_{S_{i}}\to \Z[v,v^{-1}]$ be the trivial representation and put $\presubscript{S_{1}}{\mathcal{H}}_{S_{2}} = \Hom_{\mathcal{H}}(\triv_{S_{2}}\otimes_{\mathcal{H}_{S_{2}}}\mathcal{H},\triv_{S_{1}}\otimes_{\mathcal{H}_{S_{1}}}\mathcal{H})$.
If $S_{3}\subset S$ is another subset such that $W_{S_{3}}$ is finite, then we have a $\Z[v,v^{-1}]$-linear bilinear form $*_{S_{2}}\colon \presubscript{S_{1}}{\mathcal{H}}_{S_{2}}\times \presubscript{S_{2}}{\mathcal{H}}_{S_{3}}\to \presubscript{S_{1}}{\mathcal{H}}_{S_{3}}$ defined by the composition.
Williamson defined a category which categorifies modules $\presubscript{S_{1}}{\mathcal{H}}_{S_{2}}$ and the map $*_{S_{2}}$.

Soergel and Williamson needed a certain representation of $W$.
For simplicity, let $\mathbb{K}$ be a field (in the main body we will work with a complete local Noetherian ring) and $V$ a finite dimensional $\mathbb{K}$-representation of $W$.
Let $V^{*}$ be the dual space of $V$.
Following Elias-Williamson~\cite{MR3555156}, we assume that for each $s\in S$ non-zero elements $\alpha_{s}\in V$ and $\alpha_{s}^{\vee}\in V^{*}$ which satisfy $s(v) = v - \langle v,\alpha_{s}^{\vee}\rangle\alpha_{s}$ for any $v\in V$ and $s(\alpha_{s}) = -\alpha_{s}$ are given.
Soergel and Williamson assumed that $V$ is reflection faithful and under this assumption, Williamson constructed categories $\Sbimod{S_{1}}{S_{2}}$ and bifunctors $\Sbimod{S_{1}}{S_{2}}\times \Sbimod{S_{2}}{S_{3}}\to \Sbimod{S_{1}}{S_{3}}$ and proved that these categories and functors form a categorification of the above data.
(To be more precise, he assumed that the torsion primes for $(W_{S_{i}}, S_{i})$ are non-zero in $\mathbb{K}$. We also need this assumption for our theorem.)

\subsection{The category of singular Soergel bimodules}
Now we state our main theorem.
We do not assume that $V$ is reflection faithful.
However, we assume that $V$ satisfies \cite[Assumption~1.1]{arXiv:2012.09414_accepted} which enables us to use a theory in \cite{MR4321542}, in particular, we have a Hecke category for the regular case.
Let $R$ be the symmetric algebra of $V$ and set $Q = R[w(\alpha_{s})^{-1}\mid w\in W,s\in S]$.
Let us recall the construction in \cite{MR4321542} briefly.
We define the category $\mathcal{C}$ as follows: an object in $\mathcal{C}$ is $M = (M,(M_{x}^{Q})_{x\in W})$ where $M$ is a graded $R$-bimodule and $M_{x}^{Q}$ is a graded $Q$-bimodule such that $mf = x(f)m$ for $m\in M_{Q}^{x}$, $f\in Q$ and $M\otimes_{R}Q = \bigoplus_{x\in W}M_{Q}^{x}$.
We often write only $M$ for the object $(M,(M_{x}^{Q})_{x\in W})$.
The morphism $M\to N$ is an $R$-module homomorphism $\varphi\colon M\to N$ of degree zero such that $\varphi(M_{x}^{Q})\subset N_{x}^{Q}$ for any $x\in W$.
It is easy to see that this category is additive.
Moreover, it has a structure of a monoidal category as follows.
Let $M,N\in \mathcal{C}$.
Then $M \otimes N\in \mathcal{C}$ is defined as $(M\otimes_{R}N,(M \otimes N)_{x}^{Q})$ where $(M \otimes N)_{x}^{Q} = \bigoplus_{y,z\in W,yz = x}M_{y}^{Q}\otimes_{Q}N_{z}^{Q}$.
The regular Hecke category (the category of Soergel bimodules) $\mathcal{S}$ is defined as the full-subcategory of $\mathcal{C}$ generated by Bott-Samelson bimodules.
(Here we omit the precise definition, because it is not important for explaining the category for singular cases.)

To define the categories $\Sbimod{S_{1}}{S_{2}}$, we need the following assumptions on $S_{i}$.
\begin{enumerate}
\item The representation $V$ is faithful as a $W_{S_{i}}$-representation.
\item Any torsion prime of $(W_{S_{i}},S_{i})$ in $V$ is non-zero.
\end{enumerate}
For the precise meaning of the second assumption, see Assumption~\ref{assump:assumption}.

Fix such $S_{1},S_{2},S_{3}$.
Let $R^{S_{i}}$ (resp.\ $Q^{S_{i}}$) be the ring of $W_{S_{i}}$-fixed points in $R$ (resp.\ $Q$).
Let $x\in W_{S_{1}}\backslash W/W_{S_{2}}$ and we regard this as a subset of $W$.
We consider the algebra $\prod_{a\in x}Q$ and an action of $W_{S_{1}}\times W_{S_{2}}$ defined by $w_{1}(f_{a})w_{2} = (w_{1}(f_{w_{1}^{-1}aw_{2}^{-1}}))_{a\in x}$.
Then we can consider the algebra $\presupscript{S_{1}}{Q}^{S_{2}}_{x} = (\prod_{a\in x}Q)^{W_{S_{1}}\times W_{S_{2}}}$.
We define a structure map $Q^{S_{1}}\otimes_{\mathbb{K}}Q^{S_{2}}\to \presupscript{S_{1}}{Q}^{S_{2}}_{x}$ by $f\otimes g\mapsto (fa(g))_{a\in x}$.
Then $\presupscript{S_{1}}{Q}^{S_{2}}_{x}$ is a $Q^{S_{1}}\otimes_{\mathbb{K}}Q^{S_{2}}$-algebra.
Using this algebra, we define the category $\BigCat{S_{1}}{S_{2}}$.
The category $\BigCat{S_{1}}{S_{2}}$ is the category of $(M,(M_{Q}^{x})_{x\in W_{S_{1}}\backslash W/W_{S_{2}}})$ which satisfy the following conditions.
\begin{itemize}
\item $M$ is a graded $(R^{S_{1}},R^{S_{2}})$-bimodule.
\item $M_{x}^{Q}$ is an $\presupscript{S_{1}}{Q}^{S_{2}}_{x}$-module.
\item We have $Q^{S_{1}}\otimes_{R^{S_{1}}}M\otimes_{R^{S_{2}}}Q^{S_{2}}\simeq \bigoplus_{x\in W_{S_{1}}\backslash W/W_{S_{2}}}M_{x}^{Q}$ as $(Q^{S_{1}},Q^{S_{2}})$-bimodules where $Q^{S_{1}}\otimes_{\mathbb{K}}Q^{S_{2}}$ acts on the right through the structure map.
\item Some other conditions (see Definition~\ref{defn:big category}).
\end{itemize}
We only write $M$ for $(M,(M_{Q}^{x}))$.
For $M\in \BigCat{S_{1}}{S_{2}}$ we have a grading shift $M(k)\in \BigCat{S_{1}}{S_{2}}$ for $k\in \Z$.
Therefore $\BigCat{S_{1}}{S_{2}}$ is a $\mathbb{K}$-linear graded category.

Let $M_{1}\in \BigCat{S_{1}}{S_{2}}$ and $M_{2}\in \BigCat{S_{2}}{S_{3}}$, we define $M_{1}\otimes_{S_{2}}M_{2}\in \BigCat{S_{1}}{S_{3}}$.
As an $(R^{S_{1}},R^{S_{3}})$-bimodule, we set $M_{1}\otimes_{S_{2}}M_{2} = M_{1}\otimes_{R^{S_{2}}}M_{2}$.
To define $(M_{1}\otimes_{S_{2}}M_{2})_{Q}^{x}$ for $x\in W_{S_{1}}\backslash W/W_{S_{3}}$, we need to analyze an $\presupscript{S_{1}}{Q}^{S_{2}}_{y}\otimes_{Q^{S_{2}}} \presupscript{S_{2}}{Q}^{S_{3}}_{z}$-module $(M_{1})_{Q}^{y}\otimes_{Q^{S_{2}}}(M_{2})_{Q}^{z}$.
Since the definition needs some preparation, we skip the definition in the introduction.
See subsection~\ref{subsec:Convolution} for the definition.

The category $\Sbimod{S_{1}}{S_{2}}$ is defined as a full-subcategory of $\BigCat{S_{1}}{S_{2}}$.
Let $M\in \mathcal{C}$.
Then by restriction, $M$ is an $(R^{S_{1}},R^{S_{2}})$-bimodule.
For $x\in W_{S_{1}}\backslash W/W_{S_{2}}$, we consider $\bigoplus_{a\in x}M_{Q}^{x}$.
Then this is $\prod_{a\in x}Q$-module, hence $\presupscript{S_{1}}{Q}^{S_{2}}_{x}$-module.
This data gives an object in $\BigCat{S_{1}}{S_{2}}$ and define a functor $\presubscript{S_{1},\emptyset}{\pi}_{S_{2},\emptyset,*}\colon \BigCat{}{}\to \BigCat{S_{1}}{S_{2}}$.
Now for $M\in \BigCat{S_{1}}{S_{2}}$, we declare $M\in \Sbimod{S_{1}}{S_{2}}$ if $M$ is a direct summand of $\presubscript{S_{1},\emptyset}{\pi}_{S_{2},\emptyset,*}(N)$ for some $N\in \Sbimod{}{}$.
We can also prove that if $M_{1}\in \Sbimod{S_{1}}{S_{2}}$ and $M_{2}\in\Sbimod{S_{2}}{S_{3}}$ then $M_{1}\otimes_{S_{2}}M_{2}\in \Sbimod{S_{1}}{S_{3}}$.

Let $[\Sbimod{S_{1}}{S_{2}}]$ be the split Grothendieck group of $\Sbimod{S_{1}}{S_{2}}$.
Since $\Sbimod{S_{1}}{S_{2}}$ is graded, this is a $\Z[v,v^{-1}]$-module and the functor $\otimes_{S_{2}}$ gives a $\Z[v,v^{-1}]$-bilinear map $[\Sbimod{S_{1}}{S_{2}}]\times [\Sbimod{S_{2}}{S_{3}}]\to [\Sbimod{S_{1}}{S_{3}}]$.
We also note that we have a natural $\Z[v,v^{-1}]$-basis $\{\presupscript{S_{1}}{H}_{x}^{S_{2}}\mid x\in W_{S_{1}}\backslash W/W_{S_{2}}\}$ of $\presubscript{S_{1}}{\mathcal{H}}_{S_{2}}$.
The following is the main theorem of this paper.
\begin{thm}
We have the following.
\begin{enumerate}
\item There exists a $\Z[v,v^{-1}]$-linear isomorphism $\presubscript{S_{1}}{\ch}_{S_{2}}\colon [\Sbimod{S_{1}}{S_{2}}]\xrightarrow{\sim}\presubscript{S_{1}}{\mathcal{H}}_{S_{2}}$ such that the diagram
\[
\begin{tikzcd}\relax
[\Sbimod{S_{1}}{S_{2}}]\times [\Sbimod{S_{2}}{S_{3}}]\arrow[r,"\otimes_{S_{2}}"]\arrow[d,"\presubscript{S_{1}}{\ch}_{S_{2}}\times \presubscript{S_{2}}{\ch}_{S_{3}}"] &\relax [\Sbimod{S_{1}}{S_{3}}]\arrow[d,"\presubscript{S_{1}}{\ch}_{S_{3}}"]\\
\presubscript{S_{1}}{\mathcal{H}}_{S_{2}}\times\presubscript{S_{2}}{\mathcal{H}}_{S_{3}}\arrow[r,"*_{S_{2}}"] & \presubscript{S_{1}}{\mathcal{H}}_{S_{3}}
\end{tikzcd}
\]
is commutative.
\item For each $x\in W_{S_{1}}\backslash W/W_{S_{2}}$ there exists an indecomposable object $\presubscript{S_{1}}{B}_{S_{2}}(x)\in \Sbimod{S_{1}}{S_{2}}$ such that $\presubscript{S_{1}}{\ch}_{S_{2}}(\presubscript{S_{1}}{B}_{S_{2}}(x))\in \presupscript{S_{1}}{H}^{S_{2}}_{x} + \sum_{y < x}\Z[v,v^{-1}]\presupscript{S_{1}}{H}^{S_{2}}_{y}$.
Moreover, this is unique up to isomorphism.
\item For any indecomposable object $B\in \Sbimod{S_{1}}{S_{2}}$ there exists $x\in W_{S_{1}}\backslash W/W_{S_{2}}$ and $n\in \Z$ such that $B\simeq \presubscript{S_{1}}{B}_{S_{2}}(x)(n)$.
\end{enumerate}
\end{thm}

The proof of this theorem is different from that of Williamson.
We will construct the left adjoint functor of $\presubscript{S_{1}}{\pi}_{S_{2},*}$ and using this functor we reduce problems to those for $\mathcal{S}$.
Finally, we use results in \cite{MR4321542}.

\subsection{Relation with parity complexes}
Let $G$ be a Kac-Moody group, $W$ the Weyl group of $G$ and $S$ the set of simple reflections in $S$.
For $S_{1},S_{2}\subset S$, let $\presubscript{S_{1}}{X}$ be the generalized flag variety corresponding to $S_{1}$ (here $G$ acts from the right) and $P_{S_{2}}$ the parabolic subgroup corresponding to $S_{2}$.
We consider the $W$-representation $V = X^{*}(T)\otimes_{\Z}\mathbb{K}$ where $T\subset G$ is the maximal torus and $X^{*}(T)$ the group of characters of $T$.
Assume that $S_{1}$ and $S_{2}$ satisfy the above assumptions and let $\Parity_{P_{S_{2}}}(\presubscript{S_{1}}{X})$ be the category of $P_{S_{2}}$-equivariant parity complexes on $\presubscript{S_{1}}{X}$.
Let $S_{3}$ be another subset which satisfies the above assumptions.
Then we have the convolution functor $*_{S_{2}}\colon \Parity_{P_{S_{2}}}(\presubscript{S_{1}}{X})\times \Parity_{P_{S_{3}}}(\presubscript{S_{2}}{X})\to \Parity_{P_{S_{3}}}(\presubscript{S_{1}}{X})$.
(The functor is first defined as a functor between derived categories and by a result in \cite{MR3230821} this induces a functor between parity complexes.)
Let $\mathcal{F}\in \Parity_{P_{S_{2}}}(\presubscript{S_{1}}{X})$ and put $\presubscript{S_{1}}{\mathbb{H}}_{S_{2}}(\mathcal{F}) = \bigoplus_{k\in \Z}H^{k}_{P_{S_{2}}}(\presubscript{S_{1}}{X},\mathcal{F})$.
This is an $(R^{S_{1}},R^{S_{2}})$-bimodule and we can endow it with a structure of an object of $\BigCat{S_{1}}{S_{2}}$.
\begin{thm}
For $\mathcal{F}\in \Parity_{P_{S_{2}}}(\presubscript{S_{1}}{X})$ we have $\presubscript{S_{1}}{\mathbb{H}}_{S_{2}}(\mathcal{F}) \in \Sbimod{S_{1}}{S_{2}}$ and this gives an equivalence of categories $\presubscript{S_{1}}{\mathbb{H}}_{S_{2}}\colon \Parity_{P_{S_{2}}}(\presubscript{S_{1}}{X})\to \Sbimod{S_{1}}{S_{2}}$.
Moreover we have $\presubscript{S_{1}}{\mathbb{H}}_{S_{3}}(\mathcal{F}*_{S_{2}}\mathcal{G})\simeq \presubscript{S_{1}}{\mathbb{H}}_{S_{2}}(\mathcal{F})\otimes_{S_{2}}\presubscript{S_{2}}{\mathbb{H}}_{S_{3}}(\mathcal{G})$ for $\mathcal{F}\in \Parity_{P_{S_{2}}}(\presubscript{S_{1}}{X})$ and $\mathcal{G}\in \Parity_{P_{S_{3}}}(\presubscript{S_{2}}{X})$.
\end{thm}

When $G$ is a finite dimensional reductive group, $S_{1} = S_{2} = \emptyset$ and if the characteristic of $\mathbb{K}$ is zero, then this is a theorem of Soergel~\cite{MR1322847}.
The author thinks this theorem is known to experts at least when the characteristic of $\mathbb{K}$ is zero, but it does not exist in the literature as far as the author looked for.

\subsection*{Acknowledgment}
The author thanks Geordie Williamson for answering his question on \cite{MR2844932}.
He was supported by JSPS KAKENHI Grant Number 23H01065.
He also thank the referees for their helpful comments.

\section{Singular Soergel bimodules}
\subsection{Notation}
In this paper, we use the following notation.
Let $(W,S)$ be a Coxeter system, $\mathbb{K}$ a complete local noetherian integral domain.
The unit element of $W$ is denoted by $1$ and the length function is denoted by $\ell$.
We fix a realization $(V,\{(\alpha_s,\alpha_s^\vee)\}_{s\in S})$ over $\mathbb{K}$, namely $V$ is a free $\mathbb{K}$-module of finite rank with an action of $W$, $\alpha_{s}\in V,\alpha_{s}^{\vee}\in\Hom_{\mathbb{K}}(V,\mathbb{K})$ such that $s(v) = v - \langle\alpha_{s}^{\vee},v\rangle\alpha_{s}$ for $s\in S,v\in V$ and $\langle\alpha_{s}^{\vee},\alpha_{s}\rangle = 2$~\cite[Definition~3.1]{MR3555156}.
We assume that $\alpha_s\neq 0$ and $\alpha_s^\vee\colon V\to \mathbb{K}$ is surjective.
We also assume the vanishing of the two-colored quantum binomial coefficients~\cite[Assumption~1.1]{arXiv:2012.09414_accepted}.
Since the last assumption is technical and is not important for the arguments in this paper, we do not recall the precise statement.
We remark that this assumption enable us to use results in \cite{MR4321542} by the main theorem of \cite{arXiv:2012.09414_accepted}.
We also remark that \cite[Assumption~1.1]{arXiv:2012.09414_accepted} is automatically satisfied when $(W,S,V)$ comes from a Kac-Moody root datum~\cite[Proposition~3.7]{arXiv:2012.09414_accepted}.

We call $t\in W$ a reflection if it is conjugate to an element in $S$.
For each reflection $t = wsw^{-1}$ with $s\in S$ and $w\in W$, we put $\alpha_t = w(\alpha_s)$.
This depends on a choice of the pair $(w,s)$ but $\mathbb{K}^{\times} \alpha_t$ does not depend on $(w,s)$~\cite[Lemma~2.1]{MR4321542}.
We fix such $(w,s)$ for each $t$ to define $\alpha_t$.
Let $R$ be the symmetric algebra of $V$ over $\mathbb{K}$ and let $Q = R[\alpha_{t}^{-1}\mid \text{$t$ is a reflection}]$.
For a subset $S_{1}\subset S$, let $W_{S_{1}}$ be the group generated by $S_{1}$.
We say that $S_{1}$ is finitary if $W_{S_{1}}$ is finite.
If $S_{1}$ is finitary then let $w_{S_{1}}$ be the longest element in $W_{S_{1}}$.
Let $S_{1},S_{2}\subset S$ be finitary subsets.
Then each $w\in W_{S_{1}}\backslash W/W_{S_{2}}$ has the minimal (resp.\ the maximal) length representative.
Let $w_{-}$ (resp.\ $w_{+}$) be this representative.

The Bruhat order on $W$ is denoted by $\le$.
We also define the order $\le$ on $W_{S_{1}}\backslash W/W_{S_{2}}$ by $x\le y$ if and only if $x_{-}\le y_{-}$.
We define a topology on $W_{S_{1}}\backslash W/W_{S_{2}}$ as follows: a subset $I\subset W_{S_{1}}\backslash W/W_{S_{2}}$ is open if $y\in I$, $x\in W_{S_{1}}\backslash W/W_{S_{2}}$, $x\ge y$ implies $x\in I$.

The algebra $R$ is a graded algebra with $\deg(V) = 2$.
Here, by graded we always mean $\Z$-graded.
In general, let $A$ be a commutative graded algebra.
For a graded $A$-module $M = \bigoplus_{i\in\Z}M^i$, we define a graded $A$-module $M(k)$ by $M(k) = \bigoplus_{i\in\Z}M(k)^i$, $M(k)^i = M^{i + k}$.
A graded $A$-module $M$ is called graded free if it is isomorphic to $\bigoplus_{i = 1}^r A(n_i)$ for some $n_1,\dots,n_r\in \Z$.
Note that in this paper graded free means graded free of finite rank.
If $M\simeq \bigoplus_{i = 1}^rA(n_i)$ is graded free, the graded rank $\grk_{A}(M)\in \Z[v,v^{-1}]$ of $M$ is defined by $\grk_{A}(M) = \sum_{i = 1}^rv^{n_i}$ where $v$ is an indeterminate.

\subsection{A category}
In general, for a subset $S_{1}\subset S$, let $Q^{S_{1}} = Q^{W_{S_{1}}}$ be the subring of $W_{S_{1}}$-fixed elements in $Q$ and put $R^{S_{1}} = R^{W_{S_{1}}} = R\cap Q^{W_{S_{1}}}$.
Let $S_{1}\subset S$ be a subset such that $\#W_{S_{1}} < \infty$.

\begin{lem}\label{lem:invariants and fractions}
Let $\Sigma$ be a multiplicative subset of $R$ generated by $\{\alpha_{t}\mid \text{$t$ is a reflection}\}$ and $\Sigma^{S_{1}}$ the subset of $\Sigma$ consisting of $W_{S_{1}}$-fixed elements.
Then $Q^{S_{1}} = (\Sigma^{W_{S_{1}}})^{-1}R^{S_{1}}$.
\end{lem}
\begin{proof}
We have $Q = \Sigma^{-1}R$.
Obviously we have $(\Sigma^{S_{1}})^{-1}R^{S_{1}} \subset Q^{S_{1}}$.
Let $f_{0}\in Q^{S_{1}}$ and take $s\in \Sigma$ and $f\in R$ such that $f_{0} = s^{-1}f$.
Set $s' = \prod_{w\in W_{S_{1}}}w(s)$ and $f' = f\prod_{w\in W_{S_{1}}\setminus\{1\}}w(s)$.
Then $s'\in \Sigma^{S_{1}}$ and $f_{0} = (s')^{-1}f'$.
Since $f_{0}$ and $s'$ are fixed by $W_{S_{1}}$, $f'= f_{0}s'$ is also fixed by $W_{S_{1}}$.
Hence $f_{0}\in (\Sigma^{S_{1}})^{-1}R^{S_{1}}$.
\end{proof}

Let $x\in W_{S_{1}}\backslash W/W_{S_{2}}$ and we define a left $W_{S_{1}}$-action and a right $W_{S_{2}}$-action on $\prod_{a\in x}Q$ by $w_{1}(f_{a})w_{2} = (w_{1}(f_{w_{1}^{-1}aw_{2}^{-1}}))$.
We define
\[
\presupscript{S_{1}}{Q}^{S_{2}}_{x} = \left(\prod_{a\in x}Q\right)^{W_{S_{1}}\times W_{S_{2}}}.
\]
Note that the map $(f_{a})\mapsto f_{x_{-}}$ gives an isomorphism
\[
\presupscript{S_{1}}{Q}^{S_{2}}_{x}\simeq Q^{W_{S_{1}}\cap x_{-}W_{S_{2}}x_{-}^{-1}}
\]
and we often identify these algebras by this isomorphism.
We regard $\presupscript{S_{1}}{Q}^{S_{2}}_{x}$ as an $Q^{S_{1}}\otimes_{\mathbb{K}} Q^{S_{2}}$-algebra via a structure map $f\otimes g\mapsto (fa(g))_{a\in x}$.
In terms of $\presupscript{S_{1}}{Q}^{S_{2}}_{x}\simeq Q^{W_{S_{1}}\cap x_{-}W_{S_{2}}x_{-}^{-1}}$, this structure map is given by $f\otimes g\mapsto fx_{-}(g)$.
We also define an $R^{S_{1}}\otimes_{\mathbb{K}}R^{S_{2}}$-algebra $\presupscript{S_{1}}{R}^{S_{2}}_{x}$ in the same way.
Then we have $\presupscript{S_{1}}{R}^{S_{2}}_{x}\subset \presupscript{S_{1}}{Q}^{S_{2}}_{x}$.
An argument in the proof of Lemma~\ref{lem:invariants and fractions} gives $Q^{S_{1}}\otimes_{R^{S_{1}}}\presupscript{S_{1}}{R}^{S_{2}}_{x}\simeq \presupscript{S_{1}}{Q}^{S_{2}}_{x}\simeq \presupscript{S_{1}}{R}^{S_{2}}_{x}\otimes_{R^{S_{2}}}Q^{S_{2}}$.

\begin{defn}\label{defn:big category}
We define the category $\BigCat{S_{1}}{S_{2}}$ as follows.
An object of $\BigCat{S_{1}}{S_{2}}$ is a triple $M = (M,(M_{Q}^{x})_{x\in W_{S_{1}}\backslash W /W_{S_{2}}},\xi_{M})$ where
\begin{itemize}
\item $M$ is a graded $(R^{S_{1}},R^{S_{2}})$-bimodule.
\item $M_{Q}^{x}$ is a graded $\presupscript{S_{1}}{Q}^{S_{2}}_{x}$-module which is zero for all but finitely many $x$.
\item We have $Q^{S_{1}}\otimes_{R^{S_{1}}}M\xrightarrow{\sim}Q^{S_{1}}\otimes_{R^{S_{1}}}M\otimes_{R^{S_{2}}}Q^{S_{2}}\xleftarrow{\sim}M\otimes_{R^{S_{2}}}Q^{S_{2}}$.
\item $\xi_{M}\colon Q^{S_{1}}\otimes_{R^{S_{1}}}M\otimes_{R^{S_{2}}}Q^{S_{2}}\xrightarrow{\sim}\bigoplus_{x\in W_{S_{1}}\backslash W/W_{S_{2}}}M_{Q}^{x}$ is an isomorphism of $(Q^{S_{1}},Q^{S_{2}})$-bimodules of degree zero where the $(Q^{S_{1}},Q^{S_{2}})$-bimodule structure on the right hand side is defined by the structure morphisms $Q^{S_{1}}\otimes_{\mathbb{K}} Q^{S_{2}}\to \presupscript{S_{1}}{Q}^{S_{2}}_{x}$.
\end{itemize}

A morphism $\varphi\colon (M,(M_{Q}^{x})_{x\in W_{S_{1}}\backslash W /W_{S_{2}}},\xi_{M})\to (N,(N_{Q}^{x})_{x\in W_{S_{1}}\backslash W /W_{S_{2}}},\xi_{N})$ in $\BigCat{S_{1}}{S_{2}}$ is an $(R^{S_{1}},R^{S_{2}})$-bimodule homomorphism $\varphi\colon M\to N$ of degree zero such that $\varphi_{Q}(\xi_{M}^{-1}(M_{Q}^{x}))\subset \xi_{N}^{-1}(N_{Q}^{x})$ for any $x\in W_{S_{1}}\backslash W/W_{S_{2}}$ where $\varphi_{Q}\colon Q^{S_{1}}\otimes_{R^{S_{1}}}M\otimes_{R^{S_{2}}}Q^{S_{2}}\to Q^{S_{1}}\otimes_{R^{S_{1}}}N\otimes_{R^{S_{2}}}Q^{S_{2}}$ is the induced homomorphism.

We also define some related objects as follows:
\begin{itemize}
\item We put $M_{Q} = Q^{S_{1}}\otimes_{R^{S_{1}}}M\otimes_{R^{S_{2}}}Q^{S_{2}}$.
\item For $n\in\Z$ and $M\in \BigCat{S_{1}}{S_{2}}$, a graded $(R^{S_{1}},R^{S_{2}})$-bimodule $M(n)$ is again an object of $\BigCat{S_{1}}{S_{2}}$.
\item We put $\Hom^{\bullet}_{\BigCat{S_{1}}{S_{2}}}(M,N) = \bigoplus_{n\in\Z}\Hom_{\BigCat{S_{1}}{S_{2}}}(M,N(n))$.

\item We put $\supp_{W}(M) = \{x\in W_{S_{1}}\backslash W / W_{S_{2}}\mid M_{Q}^{x}\ne 0\}$.
We call this set the support of $M$.
\item Set $\BigCat{}{} = \BigCat{\emptyset}{\emptyset}$.
\end{itemize}
\end{defn}

We usually write only $M$ for $(M,(M_{Q}^{x})_{x\in W_{S_{1}}\backslash W /W_{S_{2}}},\xi_{M})$.

\begin{lem}\label{lem:morphism is Q_x-equivariant}
Let $M,N\in \BigCat{S_{1}}{S_{2}}$ and let $\varphi\colon M\to N$ be a morphism in $\BigCat{S_{1}}{S_{2}}$.
Assume that $N_{Q}$ is torsion-free as a left $Q^{S_{1}}$-module and a right $Q^{S_{2}}$-module.
Then the morphism $M_{Q}^{x}\to N_{Q}^{x}$ induced by $\varphi$ (and $\xi_{M},\xi_{N}$) is an $\presupscript{S_{1}}{Q}^{S_{2}}_{x}$-module homomorphism.
\end{lem}
\begin{proof}
Let $F^{S_{i}}$ be the field of fractions of $R^{S_{i}}$ for $i = 1,2$ and put $M_{F} = F^{S_{1}}\otimes_{R^{S_{1}}}M\otimes_{R^{S_{2}}}F^{S_{2}}$, $N_{F} = F^{S_{1}}\otimes_{R^{S_{1}}}N\otimes_{R^{S_{2}}}F^{S_{2}}$.
We also let $F$ be the field of fractions of $R$.
By the same argument as in the proof of Lemma~\ref{lem:invariants and fractions}, $F^{S_{i}}$ is the set of $W_{S_{i}}$-invariants in $F$.
If we put $M_{F}^{x} = F^{S_{1}}\otimes_{Q^{S_{1}}} M_{Q}^{x}\otimes_{Q^{S_{2}}} F^{S_{2}}$ and $N_{F}^{x} = F^{S_{1}}\otimes_{Q^{S_{1}}} N_{Q}^{x}\otimes_{Q^{S_{2}}} F^{S_{2}}$, then $\varphi$ induces $\varphi_{F}\colon M_{F}^{x}\to N_{F}^{x}$.
Set $\presupscript{S_{1}}{F}^{S_{2}}_{x} = F^{W_{S_{1}}\cap x_{-}W_{S_{2}}x_{-}^{-1}}$ and regard this as an $F^{S_{1}}\otimes_{\mathbb{K}}F^{S_{2}}$-algebra.
Since $N_{Q}^{x}\subset N_{F}^{x}$ and $\presupscript{S_{1}}{Q}^{S_{2}}_{x}\subset \presupscript{S_{1}}{F}^{S_{2}}_{x}$, it is sufficient to prove that $\varphi_{F}$ is an $\presupscript{S_{1}}{F}^{S_{2}}_{x}$-module homomorphism.
The map $\varphi_{F}$ is an $(F^{S_{1}},F^{S_{2}})$-bimodule homomorphism.
Hence it is sufficient to prove that the structure map $F^{S_{1}}\otimes_{\mathbb{K}}F^{S_{2}}\to \presupscript{S_{1}}{F}^{S_{2}}_{x}$ is surjective.
This follows from the following general lemma from field theory.
\end{proof}
\begin{lem}
Let $F$ be a field and $\Aut(F)$ the group of automorphisms of $F$.
\begin{enumerate}
\item If $X\subset \Aut(F)$ is a finite subset then $F\otimes_{\Z}F\to \prod_{X}F$ defined by $f\otimes g\mapsto (fx(g))_{x\in X}$ is surjective.
\item If $A,B\subset \Aut(F)$ are finite subgroups, then $F^{A\cap B} = F^{A}F^{B}$.
\end{enumerate}
\end{lem}
\begin{proof}
(1)
Let $\Phi\colon \prod_{X}F\to F$ be an $F$-linear map which is zero on the image of $F\otimes_{\Z}F\to \prod_{X}F$ and we prove $\Phi = 0$.
There exists $c_{x}\in F$ for each $x\in X$ such that $\Phi(f_{x}) = \sum_{x\in X}c_{x}f_{x}$.
From the assumptions we have $\Phi((x(g))) = 0$ for any $g\in F$, namely $\sum_{x\in X}c_{x}x(g) = 0$.
Hence $c_{x} = 0$ for any $x\in X$ by Dedekind's theorem.

(2)
We have an embedding $F^{A}F^{B}\hookrightarrow F^{A\cap B}$ and we prove that this is surjective.
It is sufficient to prove that $p\colon F\otimes_{F^{A}}F^{A}F^{B}\otimes_{F^{B}}F\to F\otimes_{F^{A}}F^{A\cap B}\otimes_{F^{B}}F$ is surjective.
Define $\varphi\colon F\otimes_{\Z}F\to F\otimes_{F^{A}}F^{A}F^{B}\otimes_{F^{B}}F$ by $\varphi(f\otimes g) = f\otimes 1\otimes g$ and $q\colon F\otimes_{F^{A}}F^{A\cap B}\otimes_{F^{B}}F\to \prod_{AB}F$ by $q(f\otimes g\otimes h) = (fa(g)ab(h))_{ab}$ where $a\in A$ and $b\in B$.
It is easy to see that $ab$-part of $q(f\otimes g\otimes h)$ does not depend on the choice of $a,b$.
Hence $q$ is well-defined.
\[
\begin{tikzcd}
F\otimes_{\Z}F\arrow[r,"\varphi"] & F\otimes_{F^{A}}F^{A}F^{B}\otimes_{F^{B}}F\arrow[r,"p"] & F\otimes_{F^{A}}F^{A\cap B}\otimes_{F^{B}}F\arrow[r,"q"] & \prod_{AB}F.
\end{tikzcd}
\]
By (1) the composition $q\circ p\circ \varphi$ is surjective.
Hence $q$ is surjective.
We have $\dim_{F}(F\otimes_{F^{A}}F^{A\cap B}\otimes_{F^{B}}F) = \dim_{F^{A}}(F^{A\cap B}\otimes_{F^{B}}F) = [F^{A\cap B}:F^{A}]\dim_{F^{A\cap B}}(F^{A\cap B}\otimes_{F^{B}}F)$.
We have $[F^{A\cap B}:F^{A}] = [F:F^{A}][F:F^{A\cap B}]^{-1} = \#A/\#(A\cap B)$.
We also have $\dim_{F^{A\cap B}}(F^{A\cap B}\otimes_{F^{B}}F) = \dim_{F^{B}}(F) = \#B$.
Hence $\dim_{F}(F\otimes_{F^{A}}F^{A\cap B}\otimes_{F^{B}}F) = \#A\#B/\#(A\cap B) = \#(AB) = \dim_{F}\prod_{AB}F$.
Therefore $q$ is an isomorphism.
As $q\circ p\circ \varphi$ is surjective, $p\circ\varphi$ is surjective.
Hence $p$ is surjective.
\end{proof}

If $I\subset W_{S_{1}}\backslash W/W_{S_{2}}$, then let $M^{I}$ be the image of $M\to M_{Q} \simeq \bigoplus_{x\in W_{S_{1}}\backslash W/W_{S_{2}}}M_{Q}^{x}\twoheadrightarrow \bigoplus_{x\in I}M_{Q}^{x}$ and $M_{I}$ be the inverse image of $\bigoplus_{x\in I}M_{Q}^{x}$ in $M$.
The following properties follow from the definitions.

\begin{lem}\label{lem:adjointness of truncation}
Let $I\subset W_{S_{1}}\backslash W/W_{S_{2}}$.
We have the following.
\begin{enumerate}
\item Let $I'$ be the complement of $I$. Then $M^{I}\simeq M/M_{I'}$.
\item Assume that $N\in\BigCat{S_{1}}{S_{2}}$ satisfies $\supp_{W}(N)\subset I$.
Then we have $\Hom_{\BigCat{S_{1}}{S_{2}}}(N,M)\simeq \Hom_{\BigCat{S_{1}}{S_{2}}}(N,M_{I})$.
Moreover if $N\to N_{Q}$ is injective then we have $\Hom_{\BigCat{S_{1}}{S_{2}}}(M,N)\simeq \Hom_{\BigCat{S_{1}}{S_{2}}}(M^{I},N)$.
\end{enumerate}
\end{lem}

\begin{defn}
We make $\presupscript{S_{1}}{R}^{S_{2}}_{x}$ an object of $\BigCat{S_{1}}{S_{2}}$ by
\[
(\presupscript{S_{1}}{R}^{S_{2}}_{x})_{Q}^{y}
=
\begin{cases}
\presupscript{S_{1}}{Q}^{S_{2}}_{x} & (y = x),\\
0 & (y\ne x).
\end{cases}
\]
\end{defn}
When $x$ is the coset containing the unit element, we write $\presupscript{S_{1}}{R}^{S_{2}}_{1}$.

\subsection{Assumption on a finitary subset and some Schubert calculus}
For $s\in S$, define $\partial_{s}\colon R\to R$ by $\partial_{s}(f) = (f - s(f))/\alpha_{s}$.
\begin{lem}\label{lem:braid relation of Demazure operators}
If $w = s_{1}\cdots s_{l} = s'_{1}\cdots s'_{l}$ are two reduced expressions of $w\in W$, then $\partial_{s_{1}}\cdots \partial_{s_{l}}\in\mathbb{K}^{\times}\partial_{s'_{1}}\cdots \partial_{s'_{l}}$.
\end{lem}
\begin{proof}
In general for $\underline{x} = (s_{1},\ldots,s_{l})\in S^{l}$ we put $\partial_{\underline{x}} = \partial_{s_{1}}\cdots \partial_{s_{l}}$.
This is an element in $\End_{\mathbb{K}}(Q)$.
By definition, as an element in $\End_{\mathbb{K}}(Q)$, we have
\begin{align*}
\partial_{\underline{x}} & = \alpha_{s_{1}}^{-1}(1 - s_{1})\alpha_{s_{2}}^{-1}(1 - s_{2})\cdots \alpha_{s_{l}}^{-1}(1 - s_{l})\\
& = \sum_{e\in \{0,1\}^{l}}(-1)^{\#\{i\mid e_{i} = 1\}}\left(\prod_{i = 1}^{l}s_{1}^{e_{1}}\cdots s_{i - 1}^{e_{i - 1}}\left(\frac{1}{\alpha_{s_{i}}}\right)\right)s_{1}^{e_{1}}\cdots s_{l}^{e_{l}}\\
& = \sum_{w\in W}\sgn(w)\left(\sum_{e\in \{0,1\}^{l},s_{1}^{e_{1}}\cdots s_{l}^{e_{l}} = w}\prod_{i = 1}^{l}s_{1}^{e_{1}}\cdots s_{i - 1}^{e_{i - 1}}\left(\frac{1}{\alpha_{s_{i}}}\right)\right)w.
\end{align*}
We put 
\[
\pi_{\underline{x}} = \prod_{i = 1}^{l}s_{1}\cdots s_{i - 1}(\alpha_{s_{i}}).
\]

Let $s,t\in S$ be distinct elements such that the order $m$ of $st$ is finite.
Let $\langle s,t\rangle$ be the subgroup of $W$ generated by $s,t$.
Put $\underline{x} = (s,t,\ldots)\in S^{m}$.
Then this is a reduced expression of the longest element in $\langle s,t\rangle$.
Note that by \cite[Assumption~1.1]{arXiv:2012.09414_accepted} and \cite[Proposition~3.4]{arXiv:2012.09414_accepted}, our realization is even-balanced.
Therefore $\xi$ in \cite[Section~3]{arXiv:2012.09414_accepted} is $1$ by the definition of $\xi$.
By \cite[Theorem~2.11, Lemma~3.2, Proposition~3.4]{arXiv:2012.09414_accepted} and the description of $\partial_{\underline{x}}$ above, we get
\[
\partial_{\underline{x}}
=
\sum_{w\in \langle s,t\rangle}\sgn(w)\frac{1}{\pi_{\underline{x}}}w.
\]
Let $\underline{y} = (t,s,\ldots)\in S^{m}$.
Then we also have the same formula for $\partial_{\underline{y}}$.
By \cite[(7.9), (7.10)]{MR4668478}, we have $\pi_{\underline{x}}\in \mathbb{K}^{\times}\pi_{\underline{y}}$.
Therefore we get $\partial_{\underline{x}}\in \mathbb{K}^{\times}\partial_{\underline{y}}$.
This implies the lemma.
\end{proof}
For each $w\in W$ we fix a reduced expression $w = s_{1}\cdots s_{l}$ and put $\partial_{w} = \partial_{s_{1}}\cdots \partial_{s_{l}}$.
We consider the following assumption on $S_{1}\subset S$.
\begin{assump}\label{assump:assumption}
\begin{enumerate}
\item The subset $S_{1}$ is finitary.
\item The map $W_{S_{1}}\to \Aut(V)$ is injective.
\item There exists $p\in R$ of degree $2\ell(w_{S_{1}})$ such that $\partial_{w_{S_{1}}}(p)\in \mathbb{K}^{\times}$.
\end{enumerate}
\end{assump}

\begin{rem}
If Assumption~\ref{assump:assumption} is true for $S_{1}$, then it is true for any subset $S'_{1}\subset S_{1}$.
It is obvious that $S'_{1}$ satisfies (1) and (2).
For (3), it follows from $\partial_{w_{S_{1}}} \in \mathbb{K}^{\times}\partial_{w_{S'_{1}}}\circ\partial_{w_{S'_{1}}^{-1}w_{S_{1}}}$.
\end{rem}

\begin{rem}
When $(W_{S_{1}},S_{1})$ is a Weyl group of a root datum, $\mathbb{K} = \Z$ and $V$ is a weight lattice of the root datum, then Demazure~\cite{MR0342522} called a prime which divides $\partial_{w_{S_{1}}}(p)$ for any $p\in R$ of degree $2\ell(w_{S_{1}})$ a torsion prime.
Therefore (2) in the assumption means that any torsion prime is invertible in $\mathbb{K}$.
\end{rem}

Standard arguments in Schubert calculus which can be found in ~\cite{MR0429933,MR0342522,MR1173115,MR2844932} give the following consequences.
We give proofs for the sake of completeness.
We fix $p$ in Assumption~\ref{assump:assumption}.

\begin{lem}\label{lem:composition of Demazure is zero}
Let $w_{1},w_{2}\in W_{S_{1}}$ such that $\ell(w_{1}) + \ell(w_{2})\ge \ell(w_{S_{1}})$ and $w_{1}w_{2}\ne w_{S_{1}}$.
Then $\partial_{w_{1}}\partial_{w_{2}} = 0$.
\end{lem}
\begin{proof}
The lemma follows from Lemma~\ref{lem:braid relation of Demazure operators} and $\partial_{s}^{2} = 0$.
\end{proof}

\begin{prop}\label{prop:Demazure basis}
We have a basis $\{\partial_{w}(p)\mid w\in W_{S_{1}}\}$ of $R^{S_{1}}$-module $R$.
In particular the $R^{S_{1}}$-module $R$ is graded free of the graded rank $\sum_{w\in W_{S_{1}}}v^{-2\ell(w)}$.
\end{prop}
\begin{proof}
We prove that $\{\partial_{w}(p)\mid w\in W_{S_{1}}\}$ is linearly independent.
Take $f_{w}\in R^{S_{1}}$ such that $\sum_{w} f_{w}\partial_{w}(p) = 0$ and we prove $f_{x} = 0$ by induction on $\ell(x)$ for $x\in W_{S_{1}}$.
Let $x\in W_{S_{1}}$.
We have $\sum_{w} f_{w}\partial_{w_{S_{1}}x^{-1}}\partial_{w}(p) = 0$.
If $\ell(w_{S_{1}}x^{-1}) + \ell(w) < \ell(w_{S_{1}})$, then $\ell(w) < \ell(x)$, hence $f_{w} = 0$.
If $\ell(w_{S_{1}}x^{-1}) + \ell(w) \ge \ell(w_{S_{1}})$ and $w\ne x$, then $\partial_{w_{S_{1}}x^{-1}}\partial_{w}(p) = 0$ by the previous lemma.
Hence $f_{w}\partial_{w_{S_{1}}x^{-1}}\partial_{w}(p) = 0$ if $w\ne x$.
Therefore $f_{x}\partial_{w_{S_{1}}x^{-1}}\partial_{x}(p) = 0$.
The left hand side is in $\mathbb{K}^{\times}f_{x}$ by Lemma~\ref{lem:braid relation of Demazure operators} and Assumption~\ref{assump:assumption}.
Hence $f_{x} = 0$.

We prove that for any $f\in R$ and $k\in \Z_{\ge 0}$ there exists $f_{0}\in R$ such that $f - f_{0}\in \sum_{\ell(w) \le \ell(w_{S_{1}}) - k}R^{S_{1}}\partial_{w}(p)$ and $\partial_{x}(f_{0}) = 0$ for any $x$ such that $\ell(x) = k$.
Assume that the claim is true for $k + 1$ and we prove the claim for $k$.
Set $a_{x} = \partial_{x}\partial_{x^{-1}w_{S_{1}}}(p)$.
It is in $\mathbb{K}^{\times}$ by Lemma~\ref{lem:braid relation of Demazure operators} and Assumption~\ref{assump:assumption}.
By inductive hypothesis there exists $f_{1}\in R$ such that $f - f_{1}\in \sum_{\ell(w) < \ell(w_{S_{1}}) - k}R^{S_{1}}\partial_{w}(p)$ and $\partial_{y}(f_{1}) = 0$ if $\ell(y) = k + 1$.

Let $w\in W_{S_{1}}$ such that $\ell(w) = k$.
Then $\partial_{s}\partial_{w}(f_{1}) = 0$ for any $s\in S_{1}$.
Indeed if $sw < w$ then it follows from $\partial_{s}^{2} = 0$ and Lemma~\ref{lem:braid relation of Demazure operators}.
If $sw > w$ then $\partial_{s}\partial_{w}(f_{1}) \in \mathbb{K}^{\times}\partial_{sw}(f_{1})$ and $\ell(sw)  = k + 1$.
Hence it follows from the condition on $f_{1}$.
Therefore $\partial_{w}(f_{1})\in R^{S_{1}}$.
Set $f_{0} = f_{1} - \sum_{\ell(w) = k}a_{w}^{-1}\partial_{w}(f_{1})\partial_{w^{-1}w_{S_{1}}}(p)$.
If $\ell(x) = \ell(w) = k$ and $x\ne w$, then $\partial_{x}(\partial_{w^{-1}w_{S_{1}}}(p)) = 0$ by the previous lemma.
Hence we have $\partial_{x}(f_{0}) = \partial_{x}(f_{1}) - \partial_{x}(f_{1})a_{x}^{-1}\partial_{x}\partial_{x^{-1}w_{S_{1}}}(p) = 0$.
We get the claim.

Put $k = 0$.
Then $f_{0}$ satisfies $\partial_{1}(f_{0}) = 0$.
Hence $f_{0} = 0$.
Therefore $\{\partial_{w}(p)\mid w\in W_{S_{1}}\}$ spans $R$.
\end{proof}

\begin{lem}
If $\ell(w_{1}^{-1}) + \ell(w_{2})\ge \ell(w_{S_{1}})$ and $w_{1}^{-1}w_{2}\ne w_{S_{1}}$ then $\partial_{w_{S_{1}}}(\partial_{w_{1}}(p)\partial_{w_{2}}(p)) = 0$.
If $w_{1}^{-1}w_{2} = w_{S_{1}}$ then $\partial_{w_{S_{1}}}(\partial_{w_{1}}(p)\partial_{w_{2}}(p))\in \mathbb{K}^{\times}$.
\end{lem}
\begin{proof}
In general we have $\partial_{s}(\partial_{s}(f)g) = \partial_{s}(f\partial_{s}(g))$.
This implies $\partial_{w_{S_{1}}}(\partial_{s}(f)g) = \partial_{w_{S_{1}}}(f\partial_{s}(g))$.
Therefore we have $\partial_{w_{S_{1}}}(\partial_{w_{1}}(p)\partial_{w_{2}}(p)) = \partial_{w_{S_{1}}}(p\partial_{w_{1}^{-1}}\partial_{w_{2}}(p))$.
Hence the lemma follows from Lemma~\ref{lem:composition of Demazure is zero} and Assumption~\ref{assump:assumption}.
\end{proof}

\begin{prop}\label{lem:Frobenius extension}
The map $R\to \Hom_{R^{W}}(R,R(-2\ell(w_{S_{1}})))$ defined by $f\mapsto (g\mapsto \partial_{w_{S_{1}}}(fg))$ is an isomorphism.
\end{prop}
\begin{proof}
From the previous lemma, a standard triangular argument gives existence of $\{q_{w}\mid w\in W_{S_{1}}\}\subset R$ such that $\partial_{w_{S_{1}}}(\partial_{x}(p)q_{y}) = \delta_{xy}$ (Kronecker's delta).
This implies the following proposition.
\end{proof}

Define $\varphi_{w}\colon R\otimes_{R^{W}}R\to R$ (resp.\ $D_{w}\colon R\otimes_{R^{W}}R\to R$) by $\varphi_{w}(f\otimes g) = f w(g)$ (resp.\ $D_{w}(f\otimes g) = f\partial_{w}(g)$).

\begin{lem}
There exists $F'_{w}\in R\otimes_{R^{W}}R$ of degree $2\ell(w)$ such that $D_{w}(F'_{w'}) \in \mathbb{K}^{\times}\delta_{ww'}$.
\end{lem}
\begin{proof}
By Lemma~\ref{lem:composition of Demazure is zero}, we have $D_{w}(1\otimes \partial_{w'}(p)) = 0$ if $\ell(w) + \ell(w') \ge \ell(w_{S_{1}})$ and $ww' \ne w_{S_{1}}$.
Moreover $D_{w}(1\otimes\partial_{w'}(p)) \in \mathbb{K}^{\times}$ if $ww' = w_{S_{1}}$.
Hence a triangular argument implies the lemma.
\end{proof}
\begin{prop}\label{prop:generic decomposition of R otimes R}
There exists $F_{w}\in R\otimes_{R^{S_{1}}}R$ of degree $2\ell(w_{S_{1}})$ such that $\varphi_{w}(F_{w'}) \in \mathbb{K}^{\times}\delta_{ww'}\prod_{t}\alpha_{t}$ where $t$ runs through reflections in $W_{S_{1}}$.
\end{prop}
\begin{proof}
We prove the existence of $F_{w_{S_{1}}}$.
Then $F_{w} = (1\otimes w_{S_{1}}w^{-1})(F_{w_{S_{1}}})$ gives the proposition.

Set $F = F'_{w_{S_{1}}}$ and we prove that this satisfies the condition for $F_{w_{S_{1}}}$.
Recall that for each $x\in W_{S_{1}}$ we have fixed a reduced expression $x = s_{1}\cdots s_{l}$.
Set $\pi_{x} = \prod_{i = 1}^{l}s_{1}\cdots s_{i - 1}(\alpha_{s_{i}})$.
Then by the definition we have $\partial_{x} \in \mathbb{K}^{\times}\pi_{x}^{-1}x + \sum_{y < x}Qy$ in $\End_{\mathbb{K}}(Q)$.
Hence $x\in \mathbb{K}^{\times}\pi_{x} \partial_{x} + \sum_{y < x}Q\partial_{y}$.
Therefore $\varphi_{x}\in \mathbb{K}^{\times}\pi_{x} D_{x} + \sum_{y < x}QD_{y}$.
Hence $\varphi_{x}(F) = 0$ if $x\ne w_{S_{1}}$ and $\varphi_{w_{S_{1}}}(F)\in \mathbb{K}^{\times}\pi_{w_{S_{1}}}$.
As $\pi_{w_{S_{1}}}\in \mathbb{K}^{\times}\prod_{t}\alpha_{t}$ where $t$ runs through reflections in $W_{S_{1}}$, we get the proposition.
\end{proof}

Set $\partial_{w}^{R} = 1\otimes\partial_{w}\colon R\otimes_{R^{S_{1}}}R\to R\otimes_{R^{S_{1}}}R$.
\begin{lem}
We have $\varphi_{x}(\partial_{s}^{R}(F)) = x(\alpha_{s})^{-1}(\varphi_{x}(F) - \varphi_{xs}(F))$ for any $F\in R\otimes_{R^{S_{1}}}R$.
\end{lem}
\begin{proof}
This follows from the definition.
\end{proof}

\begin{lem}
We have
\[
\varphi_{x}(\partial_{w}^{R}(F_{w_{S_{1}}}))
\begin{cases}
=0 & (x\not\ge w_{S_{1}}w^{-1}),\\
\in \mathbb{K}^{\times}\prod_{tx < x}\alpha_{t} & (x = w_{S_{1}}w^{-1}).
\end{cases}
\]
and $\partial_{w_{S_{1}}}^{R}(F_{w_{S_{1}}})\in \mathbb{K}^{\times}$.
\end{lem}
\begin{proof}
We prove the lemma by backward induction on $\ell(w)$.
If $\ell(w) = 0$, then $w = 1$ and the lemma follows from the property of $F_{w_{S_{1}}}$.
Assume that $\ell(w) > 0$ and take $s\in S$ such that $sw < w$.
We have $\varphi_{y}(\partial_{w}^{R}(F)) \in \mathbb{K}^{\times} \varphi_{y}(\partial_{s}^{R}\partial_{sw}^{R}(F)) = \mathbb{K}^{\times}y(\alpha_{s})^{-1}(\varphi_{y}(\partial_{sw}^{R}(F)) - \varphi_{ys}(\partial_{sw}^{R}(F)))$.
If $y\not\ge w_{S_{1}}w^{-1}$, then $y,ys\not\ge w_{S_{1}}(sw)^{-1}$.
Hence it is zero.
If $y = w_{S_{1}}w^{-1}$, then $ys > y$.
Hence $y\not\ge ys = w_{S_{1}}(sw)^{-1}$.
By inductive hypothesis, $\varphi_{y}(\partial_{sw}^{R}(F)) = 0$.
Hence we have $\varphi_{y}(\partial_{w}^{R}(F)) \in \mathbb{K}^{\times}y(\alpha_{s})^{-1}\prod_{tys < ys}\alpha_{t}$ by inductive hypothesis.
By the strong exchange condition, we have $tys < ys$ if and only if $ty < y$ or $tys = y$.
Hence we get the first claim of the lemma.

Since $\partial_{w_{S_{1}}}^{R}(F_{w_{S_{1}}})$ has degree zero, it is in $\mathbb{K}$.
Therefore $\partial_{w_{S_{1}}}^{R}(F_{w_{S_{1}}}) = \varphi_{1}(\partial_{w_{S_{1}}}^{R}(F_{w_{S_{1}}}))\otimes 1$.
The lemma follows.
\end{proof}

\begin{prop}\label{prop:equivariant Demazure basis}
$\{\partial^{R}_{w}(F_{w_{S_{1}}})\mid w\in W\}$ is a basis of $R\otimes_{R^{S_{1}}}R$ as a left $R$-module.
More generally, for each $x\in W$, $\{\partial^{R}_{w}(F_{w_{S_{1}}})\mid w\le x^{-1}w_{S_{1}}\}$ is a basis of $\{f\in R\otimes_{R^{S_{1}}}R\mid \text{$\varphi_{y}(f) = 0$ for $y\not\ge x$}\}$.
\end{prop}
\begin{proof}
For the first part, we can apply the proof of Proposition~\ref{prop:Demazure basis} using $\partial_{w_{S_{1}}}^{R}(F_{w_{S_{1}}})\in \mathbb{K}^{\times}$.

Let $f_{w}\in R$ for each $w\in W_{S_{1}}$.
Note that $w\le x^{-1}w_{S_{1}}$ if and only if $x\le w_{S_{1}}w^{-1}$.
For the second part, we prove that $\varphi_{y}(\sum_{w}f_{w}\partial_{w}(F_{w_{S_{1}}})) = 0$ for any $y\not\ge x$ if and only if $f_{w} = 0$ for any $x\not\le w_{S_{1}}w^{-1}$.
To prove the if part, it is sufficient to prove $\varphi_{y}(\partial_{w}(F_{w_{S_{1}}})) = 0$ for $y\not\ge x$ and $x\le w_{S_{1}}w^{-1}$.
This follows from the previous lemma.
We prove the only if part.
Assume that $\varphi_{y}(\sum_{w}f_{w}\partial_{w}(F_{w_{S_{1}}})) = 0$ for any $y\not\ge x$ and $f_{w} \ne 0$ for some $w$.
Let $w$ be a maximal element among such elements.
We prove $w\le x^{-1}w_{S_{1}}$.
If $\varphi_{w_{S_{1}}w^{-1}}(f_{w'}\partial_{w'}^{R}(F_{w_{S_{1}}})) \ne 0$, then $f_{w'} \ne 0$ and $w_{S_{1}}w^{-1}\ge w_{S_{1}}(w')^{-1}$ by the previous lemma, therefore $w\le w'$ and the maximality of $w$ implies $w' = w$.
Hence $\varphi_{w_{S_{1}}w^{-1}}(\sum_{w'} f_{w'}\partial_{w'}^{R}(F_{w_{S_{1}}})) = f_{w}\varphi_{w_{S_{1}}w^{-1}}(\partial_{w}^{R}(F_{w_{S_{1}}}))$ and it is not zero by the previous lemma.
Therefore $w_{S_{1}}w^{-1}\ge x$, hence $w\le x^{-1}w_{S_{1}}$.
We get the proposition.
\end{proof}

\begin{lem}\label{lem:Galois descent, partially}
The map $Q\otimes_{Q^{S_{1}}}Q \to \prod_{W_{S_{1}}}Q$ defined by $f\otimes g\mapsto (fw(g))_{w\in W_{S_{1}}}$ is an isomorphism.
\end{lem}
\begin{proof}
Let $F$ be the field of fractions of $Q$.
Then the $W_{S_{1}}$-fixed points in $F$ is the field of fractions of $Q^{S_{1}}$.
Let $F^{S_{1}}$ be this field.
By Galois descent we have $F\otimes_{F^{S_{1}}}F\simeq \prod_{W_{S_{1}}}F$.
The map $Q\otimes_{Q^{S_{1}}}Q\to \prod_{W_{S_{1}}}Q\hookrightarrow \prod_{W_{S_{1}}}F$ is the same as $Q\otimes_{Q^{S_{1}}}Q\to F\otimes_{F^{S_{1}}}F\xrightarrow{\sim} \prod_{W_{S_{1}}}F$.
Hence the map in the proposition is injective.
The surjectivity follows from Proposition~\ref{prop:generic decomposition of R otimes R}.
\end{proof}

We have $(R\otimes_{R^{S_{1}}}R)\otimes_{R}Q\simeq R\otimes_{R^{S_{1}}}Q\simeq R\otimes_{R^{S_{1}}}Q^{S_{1}}\otimes_{Q^{S_{1}}}Q\simeq Q\otimes_{Q^{S_{1}}}Q$.
Hence the previous lemma gives a structure of an object in $\BigCat{}{}$ to $R\otimes_{R^{S_{1}}}R$.

\subsection{Convolution}\label{subsec:Convolution}
Let $S_{1},S_{2},S_{3}\subset S$ be subsets which satisfy Assumption~\ref{assump:assumption}.
In this subsection, for $M\in \BigCat{S_{1}}{S_{2}}$ and $N\in \BigCat{S_{2}}{S_{3}}$, we define $M\otimes_{S_{2}} N\in \BigCat{S_{1}}{S_{3}}$.
As a graded $(R^{S_{1}},R^{S_{3}})$-bimodule, we put $M\otimes_{S_{2}} N = M\otimes_{R^{S_{2}}}N$.
To define $(M\otimes_{S_{2}} N)_{Q}^{x}$, we need some lemmas.

For $x\in W_{S_{1}}\backslash W/W_{S_{2}}$ and $y\in W_{S_{2}}\backslash W/W_{S_{3}}$, we put $x\times^{W_{S_{2}}}y = (x\times y)/{\sim}$ where $(a,b)\sim (a',b')$ if there exists $w\in W_{S_{2}}$ such that $a' = aw^{-1}$ and $b' = wb$.
\begin{lem}
Let $x\in W_{S_{1}}\backslash W/W_{S_{2}}$ and $y\in W_{S_{2}}\backslash W/W_{S_{3}}$.
Then we have a $(Q^{S_{1}},Q^{S_{3}})$-algebra isomorphism
\[
\left(\prod_{a\in x}Q\right)^{W_{S_{1}}\times W_{S_{2}}}\otimes_{Q^{S_{2}}}\left(\prod_{b\in y}Q\right)^{W_{S_{2}}\times W_{S_{3}}}
\simeq \left(\prod_{c\in x\times^{W_{S_{2}}}y}Q\right)^{W_{S_{1}}\times W_{S_{3}}}.
\]
The isomorphism is defined as follows.
For $f = (f_{a})\in (\prod_{a\in x}Q)^{W_{S_{2}}}$, $g = (g_{b})\in (\prod_{b\in y}Q)^{W_{S_{2}}}$, define $\Phi(f\otimes g)\in \prod_{c \in x\times^{W_{S_{2}}}y}Q$ by $\Phi(f\otimes g) = (f_{a}a(g_{b}))_{(a,b)}$.
The isomorphism above is a restriction of $\Phi$.
\end{lem}
\begin{proof}
It is easy to see that $\Phi$ is $W_{S_{1}}\times W_{S_{3}}$-equivariant and therefore it is sufficient to prove that 
\[
\Phi\colon \left(\prod_{a\in x}Q\right)^{W_{S_{2}}}\otimes_{Q^{S_{2}}}\left(\prod_{b\in y}Q\right)^{W_{S_{2}}}
\to
\prod_{c\in x\times^{W_{S_{2}}}y} Q
\]
is an isomorphism.

Assume that $x = y = W_{S_{2}}$.
Then the statement is
\[
\Phi\colon \left(\prod_{a\in W_{S_{2}}}Q\right)^{W_{S_{2}}}\otimes_{Q^{S_{2}}}\left(\prod_{b\in W_{S_{2}}}Q\right)^{W_{S_{2}}}
\simeq
\prod_{c\in W_{S_{2}}}Q,
\]
namely $Q\otimes_{Q^{S_{2}}}Q\simeq \prod_{W_{S_{2}}}Q$.
This is Lemma~\ref{lem:Galois descent, partially}.

We consider the general case.
Fix $\Lambda_{x}\subset x$ (resp.\ $\Lambda_{y}\subset y)$ to be a complete set of representatives of $x/W_{S_{2}}$ (resp.\ $W_{S_{2}}\backslash y$).
Then we have
\begin{align*}
& \left(\prod_{a\in x}Q\right)^{W_{S_{2}}}\otimes_{Q^{S_{2}}}\left(\prod_{b\in y}Q\right)^{W_{S_{2}}}\\
& \simeq
\left(\prod_{a'\in \Lambda_{x}}Q\right)\otimes_{Q}\left(\prod_{a\in W_{S_{2}}}Q\right)^{W_{S_{2}}}\otimes_{Q^{S_{2}}}\left(\prod_{b\in W_{S_{2}}}Q\right)^{W_{S_{2}}}\otimes_{Q}\left(\prod_{b'\in \Lambda_{y}}Q\right)\\
& \simeq
\left(\prod_{a'\in \Lambda_{x}}Q\right)\otimes_{Q}\left(\prod_{w\in W_{S_{2}}}Q\right)\otimes_{Q}\left(\prod_{b'\in \Lambda_{y}}Q\right)\\
& \simeq
\prod_{a'\in \Lambda_{x},w\in W_{S_{2}},c'\in \Lambda_{z}}Q.
\end{align*}
Since $\Lambda_{x}\times W_{S_{2}}\times \Lambda_{y}\to x\times^{W_{S_{2}}}y$ is bijective, we get the lemma.
\end{proof}
\begin{rem}\label{rem:tensor of three Q}
This lemma is generalized as follows.
Let $S_{4}\subset S$ be another subset which satisfies Assumption~\ref{assump:assumption} and $z\in W_{S_{3}}\backslash W/W_{S_{4}}$.
Then we have an isomorphism
\[
\left(\prod_{x}Q\right)^{W_{S_{1}}\times W_{S_{2}}}\otimes_{Q^{S_{2}}}
\left(\prod_{y}Q\right)^{W_{S_{2}}\times W_{S_{3}}}\otimes_{Q^{S_{3}}}
\left(\prod_{z}Q\right)^{W_{S_{3}}\times W_{S_{4}}}
\simeq
\left(\prod_{x\times^{W_{S_{2}}}y\times^{W_{S_{3}}}z}Q\right)^{W_{S_{1}}\times W_{S_{4}}}
\]
defined by $(f_{a})\otimes (g_{b})\otimes (h_{c})\mapsto (f_{a}a(g_{b})ab(h_{c}))_{(a,b,c)}$.
\end{rem}
Consider the multiplication map $\pi\colon x\times^{W_{S_{2}}}y\to xy = \{ab\mid a\in x,b\in y\}$.
Then we have an embedding
\[
\left(\prod_{a\in xy}Q\right)^{W_{S_{1}}\times W_{S_{3}}}
\hookrightarrow
\left(\prod_{c\in x\times^{W_{S_{2}}}y}Q\right)^{W_{S_{1}}\times W_{S_{3}}}
\simeq \presupscript{S_{1}}{Q}^{S_{2}}_{x}\otimes_{Q^{S_{2}}}\presupscript{S_{2}}{Q}^{S_{3}}_{y}
\]
defined by
\[
(f_{a})\mapsto (f_{\pi(c)}).
\]
The subset $xy\subset W$ is invariant under the multiplication by $W_{S_{1}}$ from the left and by $W_{S_{3}}$ from the right.
Hence we have a decomposition $xy = \coprod_{z\in W_{S_{1}}\backslash xy/W_{S_{3}}}z$.
Therefore we have
\[
\left(\prod_{a\in xy}Q\right)^{W_{S_{1}}\times W_{S_{3}}}
= \prod_{z\in W_{S_{1}}\backslash xy / W_{S_{3}}}\left(\prod_{a\in z}Q\right)^{W_{S_{1}}\times W_{S_{3}}}
\simeq \prod_{z\in W_{S_{1}}\backslash xy / W_{S_{3}}}\presupscript{S_{1}}{Q}^{S_{3}}_{z}.
\]
Hence we get
\begin{equation}\label{eq:embedding from Q to Q otimes Q}
\prod_{z\in W_{S_{1}}\backslash xy / W_{S_{3}}}\presupscript{S_{1}}{Q}^{S_{3}}_{z}
\hookrightarrow
\presupscript{S_{1}}{Q}^{S_{2}}_{x}\otimes_{Q^{S_{2}}}\presupscript{S_{2}}{Q}^{S_{3}}_{y}.
\end{equation}
By tracing the construction one can see that this is an embedding of $Q^{S_{1}}\otimes_{\mathbb{K}}Q^{S_{3}}$-algebras.
Let
\[
e_{x,y}^{z}\in \presupscript{S_{1}}{Q}^{S_{2}}_{x}\otimes_{Q^{S_{2}}}\presupscript{S_{2}}{Q}^{S_{3}}_{y}
\]
be the image of $1\in \presupscript{S_{1}}{Q}^{S_{3}}_{z}$.
For $z\in W_{S_{1}}\backslash (W\setminus xy)/W_{S_{3}}$, we put $e_{x,y}^{z} = 0$.

Now we define $M\otimes_{S_{2}} N\in \BigCat{S_{1}}{S_{3}}$.
Since $N_{Q}$ is a left $Q^{S_{2}}$-module, we have
\begin{align*}
(M\otimes_{R^{S_{2}}} N)_{Q} & \simeq M\otimes_{R^{S_{2}}} N_{Q}\\
& \simeq M\otimes_{R^{S_{2}}} Q^{S_{2}}\otimes_{Q^{S_{2}}}N_{Q}\\
& \simeq M_{Q}\otimes_{Q^{S_{2}}}N_{Q}\\
& \simeq \bigoplus_{x\in W_{S_{1}}\backslash W/W_{S_{2}},y\in W_{S_{2}}\backslash W/W_{S_{3}}}M_{Q}^{x}\otimes_{Q^{S_{2}}}N_{Q}^{y}.
\end{align*}
Each $M_{Q}^{x}\otimes_{Q^{S_{2}}}N_{Q}^{y}$ is a $\presupscript{S_{1}}{Q}^{S_{2}}_{x}\otimes_{Q^{S_{2}}}\presupscript{S_{2}}{Q}^{S_{3}}_{y}$-module.
Now we put
\[
(M\otimes_{S_{2}} N)_{Q}^{z} = \bigoplus_{x\in W_{S_{1}}\backslash W/W_{S_{2}},y\in W_{S_{2}}\backslash W/W_{S_{3}}}e_{x,y}^{z}(M_{Q}^{x}\otimes_{Q^{S_{2}}}N_{Q}^{y}).
\]
This defines $M\otimes_{S_{2}}N\in \BigCat{S_{1}}{S_{3}}$.
When $S_{2} = \emptyset$ we simply write $M\otimes N$.

If $\varphi\colon M_{1}\to M_{2}$ and $\psi\colon N_{1}\to N_{2}$ are morphisms, $(\varphi\otimes \psi)_{Q}\colon (M_{1}\otimes_{S_{2}} N_{1})_{Q}\to (M_{2}\otimes_{S_{2}} N_{2})_{Q}$ induces $\presupscript{S_{1}}{Q}^{S_{2}}_{x}\otimes_{Q^{S_{2}}}\presupscript{S_{2}}{Q}^{S_{3}}_{y}$-module homomorphism $(M_{1})_{Q}^{x}\otimes_{Q^{S_{2}}}(N_{1})_{Q}^{y}\to (M_{2})_{Q}^{x}\otimes_{Q^{S_{2}}}(N_{2})_{Q}^{y}$ and therefore it induces $e_{x,y}^{z}((M_{1})_{Q}^{x}\otimes_{Q^{S_{2}}}(N_{1})_{Q}^{y})\to e_{x,y}^{z}((M_{2})_{Q}^{x}\otimes_{Q^{S_{2}}}(N_{2})_{Q}^{y})$.
Hence $\varphi\otimes \psi$ defines a morphism $M_{1}\otimes_{S_{2}} N_{1}\to M_{2}\otimes_{S_{2}} N_{2}$ in $\BigCat{S_{1}}{S_{3}}$.
Namely $\otimes_{S_{2}}$ defines a bifunctor $\BigCat{S_{1}}{S_{2}}\times \BigCat{S_{2}}{S_{3}}\to \BigCat{S_{1}}{S_{3}}$.

The following proposition is easy to prove.
\begin{prop}
Let $M\in \BigCat{S_{1}}{S_{2}}$.
Then $\presupscript{S_{1}}{R}^{S_{1}}_{1}\otimes_{S_{1}}M\simeq M\simeq M\otimes_{S_{2}}\presupscript{S_{2}}{R}^{S_{2}}_{1}$.
\end{prop}

We also prove the following.

\begin{prop}
We have associativity.
Namely if $S_{1},S_{2},S_{3},S_{4}$ are subsets which satisfy Assumption~\ref{assump:assumption}, $M_{1}\in \BigCat{S_{1}}{S_{2}}$, $M_{2}\in \BigCat{S_{2}}{S_{3}}$ and $M_{3}\in \BigCat{S_{3}}{S_{4}}$, then we have $(M_{1}\otimes_{S_{2}}M_{2})\otimes_{S_{3}}M_{3}\simeq M_{1}\otimes_{S_{2}}(M_{2}\otimes_{S_{3}}M_3)$.
\end{prop}
\begin{proof}
We have an isomorphism as $(R^{S_{1}},R^{S_{4}})$-bimodules and we prove that this is an isomorphism in $\BigCat{S_{1}}{S_{4}}$.
Let $\varphi_{x,y}^{z}\colon \presupscript{S_{1}}{Q}^{S_{3}}_{z}\hookrightarrow \presupscript{S_{1}}{Q}^{S_{2}}_{x}\otimes_{Q^{S_{2}}}\presupscript{S_{2}}{Q}^{S_{3}}_{y}$ be the $z$-part of the map in \eqref{eq:embedding from Q to Q otimes Q}.
In the below we have $x_{ij}\in W_{S_{i}}\backslash W/W_{S_{j}}$.
We have
\begin{align*}
&((M_{1}\otimes_{S_{2}}M_{2})\otimes_{S_{3}}M_{3})_{Q}^{x_{14}}\\
& = \bigoplus_{x_{13},x_{34}}\varphi_{x_{13},x_{34}}^{x_{14}}(1)((M_{1}\otimes_{S_{2}} M_{2})_{Q}^{x_{13}}\otimes_{Q^{S_{3}}}(M_{3})_{Q}^{x_{34}})\\
& = \bigoplus_{x_{12},x_{23},x_{13},x_{34}}(\varphi_{x_{12},x_{23}}^{x_{13}}\otimes \id)(\varphi_{x_{13},x_{34}}^{x_{14}}(1))(\varphi_{x_{12},x_{23}}^{x_{13}}(1)\otimes 1)((M_{1})_{Q}^{x_{12}}\otimes_{Q^{S_{2}}} (M_{2})_{Q}^{x_{23}}\otimes_{Q^{S_{3}}}(M_{3})_{Q}^{x_{34}}).
\end{align*}
In the second, note that $\presupscript{S_{1}}{Q}^{S_{3}}_{x_{13}}$ acts on $\varphi_{x_{12},x_{23}}^{x_{13}}(1)((M_{1})_{Q}^{x_{12}}\otimes_{Q^{S_{2}}} (M_{2})_{Q}^{x_{23}})\subset (M_{1}\otimes_{S_{2}}M_{2})_{Q}^{x_{13}}$ through $\varphi_{x_{12},x_{23}}^{x_{13}}$.
We have
\begin{align*}
(\varphi_{x_{12},x_{23}}^{x_{13}}\otimes \id)(\varphi_{x_{13},x_{34}}^{x_{14}}(1))(\varphi_{x_{12},x_{23}}^{x_{13}}(1)\otimes 1)
& =
(\varphi_{x_{12},x_{23}}^{x_{13}}\otimes \id)(\varphi_{x_{13},x_{34}}^{x_{14}}(1))(\varphi_{x_{12},x_{23}}^{x_{13}}\otimes \id)(1)\\
& =
(\varphi_{x_{12},x_{23}}^{x_{13}}\otimes \id)(\varphi_{x_{13},x_{34}}^{x_{14}}(1))
\end{align*}
as $\varphi_{x_{12},x_{23}}^{x_{13}}\otimes \id$ respects the multiplications.
Similarly, we have
\begin{align*}
& (M_{1}\otimes_{S_{1}}(M_{2}\otimes_{S_{2}}M_{3}))_{Q}^{x_{14}}\\
& = \bigoplus_{x_{12},x_{23},x_{34},x_{24}}(\id\otimes\varphi_{x_{23},x_{34}}^{x_{24}})(\varphi_{x_{12},x_{24}}^{x_{14}}(1))((M_{1})_{Q}^{x_{12}}\otimes_{Q^{S_{2}}} (M_{2})_{Q}^{x_{23}}\otimes_{Q^{S_{3}}}(M_{3})_{Q}^{x_{34}}).
\end{align*}
Hence it is sufficient to prove
\[
\sum_{x_{13}}(\varphi_{x_{12},x_{23}}^{x_{13}}\otimes \id)(\varphi_{x_{13},x_{34}}^{x_{14}}(1)) = 
\sum_{x_{24}}(\id\otimes\varphi_{x_{23},x_{34}}^{x_{24}})(\varphi_{x_{12},x_{24}}^{x_{14}}(1)).
\]
in $\presupscript{S_{1}}{Q}^{S_{2}}_{x_{12}}\otimes_{Q^{S_{2}}}\presupscript{S_{2}}{Q}^{S_{3}}_{x_{23}}\otimes_{Q^{S_{3}}}\presupscript{S_{3}}{Q}^{S_{3}}_{x_{34}}$.
Under the isomorphism in Remark~\ref{rem:tensor of three Q}, both sides are equal to $(e_{a,b,c})_{a\in x_{12},b\in x_{23},c\in x_{34}}$ such that 
\[
e_{a,b,c} = 
\begin{cases}
1 & (abc\in x_{14})\\
0 & (abc\notin x_{14}).
\end{cases}
\]
We get the proposition.
\end{proof}
The associativity isomorphisms come from the associativity isomorphisms of usual tensor products and hence they satisfy the coherence conditions.
Hence we get the following theorem.
\begin{thm}
The following data give a weak $2$-category.
\begin{itemize}
\item objects: finitary subsets of $S$ which satisfy Assumption~\ref{assump:assumption}.
\item for objects $S_{1},S_{2}\subset S$, $\Hom(S_{1},S_{2}) = \BigCat{S_{1}}{S_{2}}$.
\item the composition $\Hom(S_{2},S_{3})\times \Hom(S_{1},S_{2})\to \Hom(S_{1},S_{3})$ is given by $(M_{2},M_{1})\mapsto M_{1}\otimes_{S_{2}}M_{2}$.
\end{itemize}
\end{thm}

\subsection{Pull-back and push-forwards}
As a special case of convolutions, we define the following functors.
Let $S_{1},S_{2}\subset S$ be subsets which satisfy Assumption~\ref{assump:assumption}.
We assume that $S'_{1}\subset S_{1}$ and $S'_{2}\subset S_{2}$ are given.
Then we define functors $\presubscript{S_{1},S'_{1}}{\pi}_{S_{2},S'_{2},*}\colon \BigCat{S'_{1}}{S'_{2}}\to \BigCat{S_{1}}{S_{2}}$ and $\presubscript{S_{1},S'_{1}}{\pi}_{S_{2},S'_{2}}^{*}\colon \BigCat{S_{1}}{S_{2}}\to \BigCat{S'_{1}}{S'_{2}}$ by 
\begin{gather*}
\presubscript{S_{1},S'_{1}}{\pi}_{S_{2},S'_{2},*}(M) = \presupscript{S_{1}}{R}^{S'_{1}}_{1}\otimes_{S'_{1}} M\otimes_{S'_{2}} \presupscript{S'_{2}}{R}^{S_{2}}_{1},\\
\presubscript{S_{1},S'_{1}}{\pi}_{S_{2},S'_{2}}^{*}(N) = \presupscript{S'_{1}}{R}^{S_{1}}_{1}\otimes_{S_{1}} M\otimes_{S_{2}} \presupscript{S_{2}}{R}^{S'_{2}}_{1}.
\end{gather*}
Since $\presupscript{S'_{1}}{R}^{S_{1}}_{1} = \presupscript{S_{1}}{R}^{S'_{1}}_{1} = R^{S'_{1}}$ and $\presupscript{S'_{2}}{R}^{S_{2}}_{1} = \presupscript{S_{2}}{R}^{S'_{2}}_{1} = R^{S'_{2}}$, we have
\[
  \presubscript{S_{1},S'_{1}}{\pi}_{S_{2},S'_{2},*}(N) = N|_{(R^{S_{1}},R^{S_{2}})},\quad
  \presubscript{S_{1},S'_{1}}{\pi}_{S_{2},S'_{2}}^{*}(M) = R^{S'_{1}}\otimes_{R^{S_{1}}} M\otimes_{R^{S_{2}}}R^{S'_{1}}
\]
and it follows from the definition that
\[
\presubscript{S_{1},S'_{1}}{\pi}_{S_{2},S'_{2},*}(N)_{Q}^{x}
=
\bigoplus_{x'\in W_{S'_{1}}\backslash x/W_{S'_{2}}}N_{Q}^{x'}.
\]

\begin{lem}
\begin{enumerate}
\item The map $\presupscript{S'_{1}}{R}^{S_{1}}_{1}\otimes_{S_{1}} \presupscript{S_{1}}{R}^{S'_{1}}_{1}\to \presupscript{S'_{1}}{R}^{S'_{1}}_{1}$ defined by $f\otimes g\mapsto fg$ is a morphism in $\BigCat{S'_{1}}{S'_{1}}$.
\item We have a map $\presupscript{S_{1}}{R}^{S_{1}}_{1} = R^{S_{1}}\hookrightarrow R^{S'_{1}}\simeq R^{S'_{1}}\otimes_{R^{S'_{1}}}R^{S'_{1}} = \presupscript{S_{1}}{R}^{S'_{1}}_{1}\otimes_{S_{1}} \presupscript{S'_{1}}{R}^{S_{1}}_{1}$.
This is a morphism in $\BigCat{S_{1}}{S_{1}}$.
\end{enumerate}
\end{lem}
\begin{proof}
(1)
By the definition we have $(\presupscript{S'_{1}}{R}^{S_{1}}_{1}\otimes_{S_{1}} \presupscript{S_{1}}{R}^{S'_{1}}_{1})_{Q}^{x} = 0$ if $x\notin W_{S'_{1}}\backslash W_{S_{1}}/W_{S'_{1}}$, and if $x\in W_{S'_{1}}\backslash W_{S_{1}}/W_{S'_{1}}$ then we have $(\presupscript{S'_{1}}{R}^{S_{1}}_{1}\otimes_{S_{1}} \presupscript{S_{1}}{R}^{S'_{1}}_{1})_{Q}^{x} = e_{1,1}^{x}((\presupscript{S'_{1}}{R}^{S_{1}}_{1})_{Q}^{1}\otimes_{Q^{S_{1}}} (\presupscript{S_{1}}{R}^{S'_{1}}_{1})_{Q}^{1})$.
Define $k\colon \presupscript{S'_{1}}{Q}^{S_{1}}_{1}\otimes_{Q^{S_{1}}}\presupscript{S_{1}}{Q}^{S'_{1}}_{1}\to \presupscript{S'_{1}}{Q}^{S'_{1}}_{1}$ by $k(f\otimes g) = fg$.
Then the map $\Phi\colon (\presupscript{S'_{1}}{R}^{S_{1}}_{1})_{Q}^{1}\otimes_{Q^{S_{1}}} (\presupscript{S_{1}}{R}^{S'_{1}}_{1})_{Q}^{1}\to (\presupscript{S'_{1}}{R}^{S'_{1}}_{1})_{Q}^{1}$ defined as in the lemma satisfies $\Phi(fm) = k(f)\Phi(m)$ for $f\in \presupscript{S'_{1}}{Q}^{S_{1}}_{1}\otimes_{Q^{S_{1}}}\presupscript{S_{1}}{Q}^{S'_{1}}_{1}$ and $m\in (\presupscript{S'_{1}}{R}^{S_{1}}_{1})_{Q}^{1}\otimes_{Q^{S_{1}}} (\presupscript{S_{1}}{R}^{S'_{1}}_{1})_{Q}^{1}$.
On the other hand the definition of $e_{1,1}^{x}$ gives $k(e_{1,1}^{x}) = 0$ if $x\ne 1$.
Hence $\Phi((\presupscript{S'_{1}}{R}^{S_{1}}_{1}\otimes_{S_{1}} \presupscript{S_{1}}{R}^{S'_{1}}_{1})_{Q}^{x}) = 0$ if $x\ne 1$ and therefore we get (1).

(2)
The map is obviously an $(R^{S_{1}},R^{S_{1}})$-bimodule homomorphism and it is a morphism in $\BigCat{S_{1}}{S_{1}}$ since the supports of both sides are $\{1\}$.
\end{proof}

\begin{prop}
The functor $\presubscript{S_{1},S'_{1}}{\pi}_{S_{2},S'_{2},*}$ is the right adjoint functor of $\presubscript{S_{1},S'_{1}}{\pi}_{S_{2},S'_{2}}^{*}$.
\end{prop}
\begin{proof}
Let $M\in \BigCat{S_{1}}{S_{2}}$ and $N\in \BigCat{S'_{1}}{S'_{2}}$.
We have an isomorphism
\[
\Hom_{(R^{S_{1}},R^{S_{2}})}(M,N)\simeq \Hom_{(R^{S'_{1}},R^{S'_{2}})}(R^{S'_{1}}\otimes_{R^{S_{1}}}M\otimes_{R^{S_{2}}}R^{S'_{2}},N),
\]
namely, we have
\[
\Hom_{(R^{S_{1}},R^{S_{2}})}(M,\presubscript{S_{1},S'_{1}}{\pi}_{S_{2},S'_{2},*} N)\simeq \Hom_{(R^{S'_{1}},R^{S'_{2}})}(\presubscript{S_{1},S'_{1}}{\pi}_{S_{2},S'_{2}}^{*}M,N).
\]
We prove that, under this isomorphism, morphisms in $\BigCat{S_{1}}{S_{2}}$ correspond to morphisms in $\BigCat{S'_{1}}{S'_{2}}$.
Let $\varphi\colon M\to \presubscript{S_{1},S'_{1}}{\pi}_{S_{2},S'_{2}}^{*}N$ be a morphism in $\BigCat{S_{1}}{S_{2}}$.
Then the corresponding $R^{S'_{1}}\otimes_{R^{S_{1}}}M\otimes_{R^{S_{2}}}R^{S'_{2}}\to N$ is given by
\begin{multline*}
\presupscript{S'_{1}}{R}^{S_{1}}_{1}\otimes_{S_{1}} M\otimes_{S_{2}} \presupscript{S_{2}}{R}^{S'_{2}}_{1}
\xrightarrow{\id\otimes\varphi\otimes\id}
\presupscript{S'_{1}}{R}^{S_{1}}_{1}\otimes_{S_{1}} \presupscript{S_{1}}{R}^{S'_{1}}_{1}\otimes_{S'_{1}} N\otimes_{S'_{2}} \presupscript{S'_{2}}{R}^{S_{2}}_{1}\otimes_{S_{2}} \presupscript{S_{2}}{R}^{S'_{2}}_{1}\\
\to \presupscript{S'_{1}}{R}^{S'_{1}}_{1}\otimes_{S'_{1}} N\otimes_{S'_{2}} \presupscript{S'_{2}}{R}^{S'_{2}}_{1}\simeq N
\end{multline*}
where the second map is induced by the map in the previous lemma (1).
Hence this is a morphism in $\BigCat{S'_{1}}{S'_{1}}$.

On the other hand, if $\psi\colon R^{S'_{1}}\otimes_{R^{S_{1}}}M\otimes_{R^{S_{2}}}R^{S'_{2}}\to N$ is given, then the corresponding map $M\to \presubscript{S_{1},S'_{1}}{\pi}_{S_{2},S'_{2}}^{*}N$ is given by
\begin{multline*}
M\simeq \presupscript{S_{1}}{R}^{S_{1}}_{1}\otimes_{S_{1}} M\otimes_{S_{2}} \presupscript{S_{2}}{R}^{S_{2}}_{1}\to 
\presupscript{S_{1}}{R}^{S'_{1}}_{1}\otimes_{S'_{1}} \presupscript{S'_{1}}{R}^{S_{1}}_{1}\otimes_{S_{1}} M\otimes_{S_{2}} \presupscript{S_{2}}{R}^{S'_{2}}_{1}\otimes_{S'_{2}} \presupscript{S'_{2}}{R}^{S_{2}}_{1}\\
\xrightarrow{\id\otimes \psi\otimes \id}
\presupscript{S_{1}}{R}^{S'_{1}}_{1}\otimes_{S'_{1}} N\otimes_{S'_{2}} \presupscript{S'_{2}}{R}^{S_{1}}_{1}
\end{multline*}
where the first map is induced by the map in the previous lemma (2).
Hence if $\psi$ is a morphism in $\BigCat{S'_{1}}{S'_{2}}$ then the corresponding map is a morphism in $\BigCat{S_{1}}{S_{2}}$.
\end{proof}

\begin{prop}
Let $S''_{1}\subset S'_{1}$ and $S''_{2}\subset S'_{2}$.
Then $\presubscript{S_{1},S'_{1}}{\pi}_{S_{2},S'_{2},*}\circ \presubscript{S'_{1},S''_{1}}{\pi}_{S'_{2},S''_{2},*}\simeq \presubscript{S_{1},S''_{1}}{\pi}_{S_{2},S''_{2},*}$ and $\presubscript{S'_{1},S''_{1}}{\pi}_{S'_{2},S''_{2}}^{*}\circ \presubscript{S_{1},S'_{1}}{\pi}_{S_{2},S'_{2}}^{*}\simeq \presubscript{S_{1},S''_{1}}{\pi}_{S_{2},S''_{2}}^{*}$.
\end{prop}
\begin{proof}
We have $\presupscript{S_{1}}{R}^{S'_{1}}_{1}\otimes_{S'_{1}}\presupscript{S'_{1}}{R}^{S''_{1}}_{1} = R^{S'_{1}}\otimes_{R^{S'_{1}}}R^{S''_{1}} \simeq R^{S''_{1}} = \presupscript{S_{1}}{R}^{S''_{1}}_{1}$ as an $(R^{S_{1}},R^{S''_{1}})$-bimodules and this is also an isomorphism in $\BigCat{S_{1}}{S''_{1}}$ since the support of both sides are $\{1\}$.
We also have $\presupscript{S''_{1}}{R}^{S'_{1}}_{1}\otimes_{S'_{1}}\presupscript{S'_{1}}{R}^{S_{1}}_{1}\simeq \presupscript{S''_{1}}{R}^{S_{1}}_{1}$ in $\presubscript{S_{1}''}{\mathcal{C}}_{S_{1}}$.
The proposition follows.
\end{proof}

\begin{lem}\label{lem:support of push and pull}
Let $\pi\colon W_{S'_{1}}\backslash W/W_{S'_{2}}\to W_{S_{1}}\backslash W/W_{S_{2}}$ be the natural projection.
We have $\supp_{W}(\presubscript{S_{1},S'_{1}}{\pi}_{S_{2},S'_{2}}^{*}(M)) = \pi^{-1}(\supp_{W}(M))$ and $\supp_{W}(\presubscript{S_{1},S'_{1}}{\pi}_{S_{2},S'_{2},*}(N)) = \pi(\supp_{W}(N))$ for $M\in \Sbimod{S_{1}}{S_{2}}$ and $N\in \Sbimod{S'_{1}}{S'_{2}}$.
\end{lem}
\begin{proof}
It follows from the constructions.
\end{proof}

In particular, if $\supp_{W}(M) \subset \{y\in W_{S_{1}}\backslash W/W_{S_{2}}\mid y\le x\}$, then $\supp_{W}(\presubscript{S_{1},S'_{1}}{\pi}_{S_{2},S'_{2}}^{*}(M)) \subset \{y'\in W_{S'_{1}}\backslash W/W_{S'_{2}}\mid y'\le x'\}$ where $x' = W_{S'_{1}}x_{+}W_{S'_{2}}\in W_{S'_{1}}\backslash W/W_{S'_{2}}$.

\begin{lem}\label{lem:push-pull of support cut}
Let $S'_{1}\subset S_{1}, S'_{2}\subset S_{2}$ be subsets and $\pi\colon W_{S'_{1}}\backslash W/W_{S'_{2}}\to W_{S_{1}}\backslash W/W_{S_{2}}$ the natural projection.
Then we have 
\begin{gather*}
\presubscript{S_{2},S'_{2}}{\pi}_{S_{1},S'_{1},*}(M)_{I} \simeq \presubscript{S_{2},S'_{2}}{\pi}_{S_{1},S'_{1},*}(M_{\pi^{-1}(I)}),\quad \presubscript{S_{2},S'_{2}}{\pi}_{S_{1},S'_{1},*}(M)^{I} \simeq \presubscript{S_{2},S'_{2}}{\pi}_{S_{1},S'_{1},*}(M^{\pi^{-1}(I)}),\\
\presubscript{S_{2},S'_{2}}{\pi}_{S_{1},S'_{1}}^{*}(N_{I}) \simeq \presubscript{S_{2},S'_{2}}{\pi}_{S_{1},S'_{1}}^{*}(N)_{\pi^{-1}(I)},\quad \presubscript{S_{2},S'_{2}}{\pi}_{S_{1},S'_{1}}^{*}(N^{I}) \simeq \presubscript{S_{2},S'_{2}}{\pi}_{S_{1},S'_{1}}^{*}(N)^{\pi^{-1}(I)}
\end{gather*}
for $M\in \BigCat{S'_{1}}{S'_{2}}$ and $N\in \BigCat{S_{1}}{S_{2}}$.
\end{lem}
\begin{proof}
It follows from the constructions.
\end{proof}

\subsection{Singular Soergel bimodules}
Let $\mathcal{S}$ be the category of Soergel bimodules introduced in \cite{MR4321542}.
This is a full subcategory of $\BigCat{}{}$ which is stable under $\oplus$, $\otimes$, the grading shift and taking direct summands.

Assume that $S_{1},S_{2}\subset S$ satisfy Assumption~\ref{assump:assumption}.
Since $\presubscript{S_{1},\emptyset}{\pi}_{S_{2},\emptyset,*}(\presubscript{S_{1},\emptyset}{\pi}_{S_{2},\emptyset}^{*}(M))$ is the restriction of $R\otimes_{R^{S_{1}}} M\otimes_{R^{S_{2}}}M$ to $(R^{S_{1}},R^{S_{2}})$, an obvious consequence of Proposition~\ref{prop:Demazure basis} is the following.
\begin{lem}\label{lem:pull-push is direct sum}
We have $\presubscript{S_{1},\emptyset}{\pi}_{S_{2},\emptyset,*}(\presubscript{S_{1},\emptyset}{\pi}_{S_{2},\emptyset}^{*}(M))\simeq \bigoplus_{w_{1}\in W_{S_{1}},w_{2}\in W_{S_{2}}}M(- \ell(w_{1}) - \ell(w_{2}))$.
\end{lem}

\begin{defn}
We say that $M\in \BigCat{S_{1}}{S_{2}}$ is a Soergel bimodule if there exists $B\in \mathcal{S}$ such that $M$ is isomorphic to a direct summand of $\presubscript{S_{1},\emptyset}{\pi}_{S_{2},\emptyset,*}(B)$.
\end{defn}
The full subcategory of Soergel bimodules is written as $\Sbimod{S_{1}}{S_{2}}$.

\begin{rem}\label{rem:Soergel bimodule is torsion free from one-side}
Any $B\in \Sbimod{}{}$ is graded free as a right $R$-module.
Since $\presubscript{S_{1},\emptyset}{\pi}_{S_{2},\emptyset,*}(B)\simeq B|_{R^{S_{2}}}$ as a right $R^{S_{2}}$-module, by Proposition~\ref{prop:Demazure basis}, any $M\in \Sbimod{S_{1}}{S_{2}}$ is graded free as a right $R^{S_{2}}$-module.
In particular, it is torsion-free, hence $M\to M_{Q}$ is injective.
\end{rem}

\begin{lem}\label{lem:indecomposable longest Soergel bimodule}
We have $\presupscript{\emptyset}{R}^{S_{1}}\otimes_{S_{1}}\presupscript{S_{1}}{R}^{\emptyset}\in \mathcal{S}$.
\end{lem}
\begin{proof}
The left hand side is $R\otimes_{R^{S_{1}}}R$.
Starting with $F_{1}$ in Proposition~\ref{prop:generic decomposition of R otimes R}, we can apply an argument in \cite[2.3]{MR4620135}.
\end{proof}

\begin{prop}
Let $S'_{1}\subset S_{1}$ and $S'_{2}\subset S_{2}$.
Then we have $\presubscript{S_{1},S'_{1}}{\pi}_{S_{2},S'_{2},*}(\Sbimod{S'_{1}}{S'_{2}})\subset \Sbimod{S_{1}}{S_{2}}$ and $\presubscript{S_{1},S'_{1}}{\pi}_{S_{2},S'_{2}}^{*}(\Sbimod{S_{1}}{S_{2}})\subset \Sbimod{S'_{1}}{S'_{2}}$
\end{prop}
\begin{proof}
For $\presubscript{S_{1},S'_{1}}{\pi}_{S_{2},S'_{2},*}$ it is obvious.
We prove $\presubscript{S_{1},S'_{1}}{\pi}_{S_{2},S'_{2}}^{*}(\presubscript{S_{1},\emptyset}{\pi}_{S_{2},\emptyset,*}(M))\in \Sbimod{S'_{1}}{S'_{2}}$ for $M\in \Sbimod{}{}$.
First assume that $S'_{1} = S'_{2} = \emptyset$.
Then 
\[
\presubscript{S_{1},\emptyset}{\pi}_{S_{2},\emptyset}^{*}(\presubscript{S_{1},\emptyset}{\pi}_{S_{2},\emptyset,*}(M))
\simeq
\presupscript{\emptyset}{R}^{S_{1}}_{1}\otimes_{S_{1}} \presupscript{S_{1}}{R}^{\emptyset}_{1}\otimes M\otimes \presupscript{\emptyset}{R}^{S_{2}}_{1}\otimes_{S_{2}} \presupscript{S_{2}}{R}^{\emptyset}_{1}.
\]
Hence we get $\presubscript{S_{1},\emptyset}{\pi}_{S_{2},\emptyset}^{*}(\presubscript{S_{1},\emptyset}{\pi}_{S_{2},\emptyset,*}(M))\in \Sbimod{}{}$ by Lemma~\ref{lem:indecomposable longest Soergel bimodule}.

We prove the proposition for general $S'_{1}$ and $S'_{2}$.
By Lemma~\ref{lem:pull-push is direct sum}, it is sufficient to prove 
\[
\presubscript{S'_{1},\emptyset}{\pi}_{S'_{2},\emptyset,*}(\presubscript{S'_{1},\emptyset}{\pi}_{S'_{2},\emptyset}^{*}(\presubscript{S_{1},S'_{1}}{\pi}_{S_{2},S'_{2}}^{*}(\presubscript{S_{1},\emptyset}{\pi}_{S_{2},\emptyset,*}(M))))\in \Sbimod{S'_{1}}{S'_{2}}.
\]
The left hand side is isomorphic to
\[
\presubscript{S'_{1},\emptyset}{\pi}_{S'_{2},\emptyset,*}(\presubscript{S_{1},\emptyset}{\pi}_{S_{2},\emptyset}^{*}(\presubscript{S_{1},\emptyset}{\pi}_{S_{2},\emptyset,*}(M)))
\]
and this is in $\Sbimod{S'_{2}}{S'_{2}}$ since $\presubscript{S_{1},\emptyset}{\pi}_{S_{2},\emptyset}^{*}(\presubscript{S_{1},\emptyset}{\pi}_{S_{2},\emptyset,*}(M))\in \Sbimod{}{}$.
\end{proof}

\begin{prop}
Let $S_{1},S_{2},S_{3}\subset S$ be subsets which satisfy Assumption~\ref{assump:assumption}.
Then $\Sbimod{S_{1}}{S_{2}}\otimes_{S_{2}} \Sbimod{S_{2}}{S_{3}}\subset \Sbimod{S_{1}}{S_{3}}$.
\end{prop}
\begin{proof}
For $M,N\in \Sbimod{}{}$, we have 
\begin{align*}
\presubscript{S_{1},\emptyset}{\pi}_{S_{2},\emptyset,*}(M)\otimes_{S_{2}} \presubscript{S_{2},\emptyset}{\pi}_{S_{3},\emptyset,*}(N)
& \simeq \presupscript{S_{1}}{R}^{\emptyset}_{1}\otimes M\otimes \presupscript{\emptyset}{R}^{S_{2}}_{1}\otimes_{S_{2}} \presupscript{S_{2}}{R}^{\emptyset}_{1}\otimes N\otimes \presupscript{\emptyset}{R}^{S_{3}}_{1}\\
&\simeq \presubscript{S_{1},\emptyset}{\pi}_{S_{3},\emptyset,*}(M\otimes (\presupscript{\emptyset}{R}^{S_{2}}_{1}\otimes_{S_{2}} \presupscript{S_{2}}{R}^{\emptyset}_{1})\otimes N).
\end{align*}
Since $\presupscript{\emptyset}{R}^{S_{2}}_{1}\otimes_{S_{2}} \presupscript{S_{2}}{R}^{\emptyset}_{1}\in \Sbimod{}{}$, we get the proposition.
\end{proof}

\begin{prop}\label{prop:projectivity}
Let $I'\subset I\subset W_{S_{1}}\backslash W/W_{S_{2}}$ be closed subsets.
Then for $M,N\in \Sbimod{S_{1}}{S_{2}}$, the natural map $\Hom_{\BigCat{S_{1}}{S_{2}}}(M,N_{I})\to\Hom_{\BigCat{S_{1}}{S_{2}}}(M,N_{I}/N_{I'})$ is surjective.
\end{prop}
\begin{proof}
We may assume $N = \presubscript{S_{1},\emptyset}{\pi}_{S_{2},\emptyset,*}(N_{0})$ for some $N_{0}\in \Sbimod{}{}$.
Let $\pi\colon W\to W_{S_{1}}\backslash W/W_{S_{2}}$ be the natural projection.
Then by Lemma~\ref{lem:push-pull of support cut} and the adjointness, we have
\[
\begin{tikzcd}
\Hom_{\BigCat{S_{1}}{S_{2}}}(M,N_{I})\arrow[r]\arrow[d,"\sim" sloped] & \Hom_{\BigCat{S_{1}}{S_{2}}}(M,N_{I}/N_{I'})\arrow[d,"\sim" sloped]\\
\Hom_{\BigCat{}{}}(\presubscript{S_{1},\emptyset}{\pi}_{S_{2},\emptyset}^{*}(M),(N_{0})_{\pi^{-1}(I)})\arrow[r] & \Hom_{\BigCat{}{}}(\presubscript{S_{1},\emptyset}{\pi}_{S_{2},\emptyset}^{*}(M),(N_{0})_{\pi^{-1}(I)}/(N_{0})_{\pi^{-1}(I')}).
\end{tikzcd}
\]
Therefore we may assume $S_{1} = S_{2} = \emptyset$.
This is \cite[Proposition~2.24]{MR4620135}.
\end{proof}

\begin{lem}\label{lem:projectivity, useful case}
Let $M,N\in \Sbimod{S_{1}}{S_{2}}$ and $x\in W_{S_{1}}\backslash W/W_{S_{2}}$ be a maximal element in $\supp_{W}(N)$.
Then $\Hom_{\Sbimod{S_{1}}{S_{2}}}(M,N)\to \Hom_{\BigCat{S_{1}}{S_{2}}}(M^{x},N^{x})$ is surjective.
\end{lem}
\begin{proof}
Let $I$ be the closure of $\supp_{W}(N)$ and set $I' = I\setminus\{x\}$.
Then $I'$ is also closed.
We have $N_{I} = N$ and $N_{I}/N_{I'}\simeq N^{x}$.
Hence $\Hom_{\Sbimod{S_{1}}{S_{2}}}(M,N)\to \Hom_{\BigCat{S_{1}}{S_{2}}}(M,N^{x})$ is surjective.
The right hand side is isomorphic to $\Hom_{\BigCat{S_{1}}{S_{2}}}(M^{x},N^{x})$ by Lemma~\ref{lem:adjointness of truncation}.
\end{proof}

\begin{prop}\label{prop:stalk is graded free}
Let $M\in \Sbimod{S_{1}}{S_{2}}$, $x\in W_{S_{1}}\backslash W/W_{S_{2}}$ and $I_{2}\subset I_{1}\subset W_{S_{1}}\backslash W/W_{S_{2}}$ closed subsets such that $I_{1}\setminus I_{2} = \{x\}$.
Then $M_{I_{1}}/M_{I_{2}}\subset M_{Q}^{x}$ is $\presupscript{S_{1}}{R}^{S_{2}}_{x}$-stable and graded free as an $\presupscript{S_{1}}{R}^{S_{2}}_{x}$-module.
\end{prop}
\begin{proof}
We may assume $M = \presubscript{S_{1},\emptyset}{\pi}_{S_{2},\emptyset,*}(N)$ for some $N\in \Sbimod{}{}$.
Then
\[
M_{I_{1}}/M_{I_{2}}
\simeq 
\presubscript{S_{1},\emptyset}{\pi}_{S_{2},\emptyset,*}(N_{\pi^{-1}(I_{1})}/N_{\pi^{-1}(I_{2})})
\]
where $\pi\colon W\to W_{S_{1}}\backslash W/W_{S_{2}}$ is the natural projection.
The module $N_{\pi^{-1}(I_{1})}/N_{\pi^{-1}(I_{2})}$ is an $R$-bimodule and it is graded free as a left $R$-module~\cite[Proposition~2.30]{MR4620135}.
The module $M_{I_{1}}/M_{I_{2}}$ is isomorphic to $N_{\pi^{-1}(I_{1})}/N_{\pi^{-1}(I_{2})}$ as a $\mathbb{K}$-module and $\presupscript{S_{1}}{R}^{S_{2}}_{x}$ acts through an embedding $\presupscript{S_{1}}{R}^{S_{2}}_{x} \hookrightarrow R$.
Therefore $M_{I_{1}}/M_{I_{2}}$ is $\presupscript{S_{1}}{R}^{S_{2}}_{x}$-stable, and, by Proposition~\ref{prop:Demazure basis}, this is graded free.
\end{proof}

\begin{rem}
Let $M,N\in \BigCat{S_{1}}{S_{2}}$ and $x\in W_{S_{1}}\backslash W/W_{S_{2}}$.
Assume that $\supp_{W}(M),\supp_{W}(N) =  \{x\}$.
Then we have $\Hom_{\BigCat{S_{1}}{S_{2}}}(M,N) = \Hom_{(R^{S_{1}},R^{S_{2}})}(M,N)$.
If moreover $M\to M_{Q}$ and $N\to N_{Q}$ are injective and $M\subset M_{Q}^{x}$, $N\subset N_{Q}^{x}$ are $\presupscript{S_{1}}{R}^{S_{2}}_{x}$-stable, then we have an injection $\Hom_{(R^{S_{1}},R^{S_{2}})}(M,N)\supset \Hom_{\presupscript{S_{1}}{R}^{S_{2}}_{x}}(M,N)$ and in fact it is an isomorphism by Lemma~\ref{lem:morphism is Q_x-equivariant}.
Hence we have $\Hom_{\BigCat{S_{1}}{S_{2}}}(M,N) = \Hom_{\presupscript{S_{1}}{R}^{S_{2}}_{x}}(M,N)$.
\end{rem}

\begin{thm}\label{thm:classification of indecomposables}
\begin{enumerate}
\item For each $x\in W_{S_{1}}\backslash W/W_{S_{2}}$, there exists an indecomposable object $\presubscript{S_{1}}{B_{S_{2}}(x)}\in\Sbimod{S_{1}}{S_{2}}$ such that $\supp_{W}(\presubscript{S_{1}}{B_{S_{2}}(x)})\subset \{y\in W_{S_{1}}\backslash W/W_{S_{2}}\mid y\le x\}$ and $\presubscript{S_{1}}{B_{S_{2}}(x)}^{x}\simeq \presupscript{S_{1}}{R}^{S_{2}}_{x}(\ell(x_{+}) - \ell(w_{S_{1}}))$ as graded $\presupscript{S_{1}}{R}^{S_{2}}_{x}$-modules.
Moreover $\presubscript{S_{1}}{B_{S_{2}}(x)}$ is unique up to isomorphism.
\item For any indecomposable object $B\in \Sbimod{S_{1}}{S_{2}}$ there exists unique $(x,n)\in (W_{S_{1}}\backslash W/W_{S_{2}})\times \Z$ such that $B\simeq \presubscript{S_{1}}{B_{S_{2}}(x)}(n)$.
\end{enumerate}
\end{thm}

When $S_{1} = S_{2} = \emptyset$, this theorem is proved in \cite{MR4321542}.
The object $\presubscript{\emptyset}{B}_{\emptyset}(x)$ is denoted by $B(x)$.

\begin{proof}
We have $\supp_{W}(\presubscript{S_{1},\emptyset}{\pi}_{S_{2},\emptyset,*}(B(x_{-}))) = \pi(\supp_{W}(B(x_{-}))) \subset \{y\in W_{S_{1}}\backslash W/W_{S_{2}}\mid y\le x\}$ where $\pi\colon W\to W_{S_{1}}\backslash W/W_{S_{2}}$ is the natural projection.
We also have $\presubscript{S_{1},\emptyset}{\pi}_{S_{2},\emptyset,*}(B(x_{-}))^{x}_{Q}\ne 0$.
Therefore there exists an indecomposable direct summand $B$ of $\presubscript{S_{1},\emptyset}{\pi}_{S_{2},\emptyset,*}(B(x_{-}))$ such that $B^{x}_{Q}\ne 0$.
We have $\supp_{W}(B)\subset \{y\in W_{S_{1}}\backslash W/W_{S_{2}}\mid y\le x\}$.

By Proposition~\ref{prop:stalk is graded free}, $B^{x}$ is a graded free $\presupscript{S_{1}}{R}^{S_{2}}_{x}$-module.
We first prove that this is indecomposable, namely isomorphic to $\presupscript{S_{1}}{R}^{S_{2}}_{x}$ up to grading shifts.
Assume that it is not the case and let $q\in \End_{\presupscript{S_{1}}{R}^{S_{2}}_{x}}(B^{x})$ be a non-trivial idempotent.
Then by Lemma~\ref{lem:projectivity, useful case} there exists $\widetilde{q}\colon B\to B$ which induces $q\colon B^{x}\to B^{x}$.
This map is not nilpotent nor isomorphism as $q$ is so.
This is a contradiction since $B$ is indecomposable.
Therefore a suitable grading shift of $B$ satisfies the conditions of $\presubscript{S_{1}}{B_{S_{2}}(x)}$.

To prove (1) and (2), it is sufficient to prove that any $M\in \Sbimod{S_{1}}{S_{2}}$ contains $\presubscript{S_{1}}{B_{S_{2}}(x)}(n)$ as a direct summand for some $x\in W_{S_{1}}\backslash W/W_{S_{2}}$ and $n\in\Z$.
Let $x\in \supp_{W}(M)$ be a maximal element.
Then $M^{x}$ contains a direct summand $\presupscript{S_{1}}{R}^{S_{2}}_{x}(\ell(x_{+}) - \ell(w_{S_{1}}) + n)\simeq \presubscript{S_{1}}{B_{S_{2}}(x)}(n)^{x}$ for some $n\in\Z$ by Proposition~\ref{prop:stalk is graded free}.
Let $p\colon M^{x}\to \presubscript{S_{1}}{B_{S_{2}}(x)}(n)^{x}$ (resp.\ $i\colon \presubscript{S_{1}}{B_{S_{2}}(x)}(n)^{x}\to M^{x}$) be the projection to (resp.\ embedding from) the direct summand.
By Lemma~\ref{lem:projectivity, useful case} there exist $\widetilde{p}\colon M\to \presubscript{S_{1}}{B_{S_{2}}(x)}(n)$ and $\widetilde{i}\colon \presubscript{S_{1}}{B_{S_{2}}(x)}(n)\to M$ such that the induced maps $M^{x}\to \presubscript{S_{1}}{B_{S_{2}}(x)}(n)^{x}$ and $\presubscript{S_{1}}{B_{S_{2}}(x)}(n)^{x}\to M^{x}$ are $p$ and $i$, respectively.
Since $\widetilde{p}\circ\widetilde{i}$ induces the identity map on $\presubscript{S_{1}}{B_{S_{2}}(x)}(n)^{x}$, it is not nilpotent.
As $\presubscript{S_{1}}{B_{S_{2}}(x)}(n)$ is indecomposable, $\widetilde{p}\circ\widetilde{i}$ is an isomorphism.
Hence $\presubscript{S_{1}}{B_{S_{2}}(x)}(n)$ is a direct summand of $M$.
\end{proof}

\subsection{Hecke algebras}
Let $\mathcal{H}$ be the Hecke algebra attached to $(W,S)$.
We use the following notation after \cite{MR1444322}.
This is a $\Z[v,v^{-1}]$-algebra generated by $\{H_{s}\mid s\in S\}$ and defined by the relations $(H_{s} - v^{-1})(H_{s} + v) = 0$ and the braid relations.
As $\{H_{s}\mid s\in S\}$ satisfies the braid relation, we can define $H_{w}$ for $w\in W$ using a reduced expression of $w$.
Then $\{H_{w}\mid w\in W\}$ is a $\Z[v,v^{-1}]$-basis of $\mathcal{H}$.

For each finitary subset $S_{1}\subset S$, let $\mathcal{H}_{S_{1}}$ be the subalgebra of $\mathcal{H}$ generated by $\{H_{s}\mid s\in S_{1}\}$.
It is known that this is isomorphic to the Hecke algebra attached to $(W_{S_{1}},S_{1})$.
For finitary subsets $S_{1},S_{2}\subset S$, we put
\[
\underline{H}_{w_{S_{1}}} = \sum_{w\in W_{S_{1}}}v^{\ell(w_{S_{1}}) - \ell(w)}H_{w}\in \mathcal{H},\quad
\presubscript{S_{1}}{\mathcal{H}_{S_{2}}} = \underline{H}_{w_{S_{1}}}\mathcal{H}\cap \mathcal{H}\underline{H}_{w_{S_2}}.
\]
Note that 
\begin{equation}\label{eq:KL basis for longest element}
\underline{H}_{w_{S_{2}}}H_{w} = v^{-\ell(w)}\underline{H}_{w_{S_{2}}}
\end{equation}
for any $w\in W_{S_{2}}$.
For $x\in W_{S_{1}}\backslash W/W_{S_{2}}$, we put
\[
\presupscript{S_{1}}{H}^{S_{2}}_{x} = \sum_{a\in x}v^{\ell(x_{+}) - \ell(a)}H_{a}.
\]
Then $\{\presupscript{S_{1}}{H}^{S_{2}}_{x}\mid x\in W_{S_{1}}\backslash W/W_{S_{2}}\}$ is a $\Z[v,v^{-1}]$-basis of $\presubscript{S_{1}}{\mathcal{H}}_{S_{2}}$.
We have $\presupscript{S_{1}}{H}^{S_{1}}_{1} = \underline{H}_{w_{S_{1}}}$.
Let $S_{3}\subset S$ be another finitary subset.
Then we have a $\Z[v,v^{-1}]$-bilinear map
\[
*_{S_{2}}\colon \presubscript{S_{1}}{\mathcal{H}}_{S_{2}}\times \presubscript{S_{2}}{\mathcal{H}}_{S_{3}}\to \presubscript{S_{1}}{\mathcal{H}}_{S_{3}}
\]
defined by
\[
h_{1}*_{S_{2}}h_{2} = \left(\sum_{w\in W_{S_{2}}}v^{\ell(w_{S_{2}}) - 2\ell(w)}\right)^{-1}h_{1}h_{2}.
\]

Let $[\Sbimod{S_{1}}{S_{2}}]$ be the split Grothendieck group of $\Sbimod{S_{1}}{S_{2}}$.
For $M\in \Sbimod{S_{1}}{S_{2}}$, $[M]$ denotes the image of $M$ in $[\Sbimod{S_{1}}{S_{2}}]$.
We define $\Z[v,v^{-1}]$-module structure on $[\Sbimod{S_{1}}{S_{2}}]$ by $v[M] = [M(1)]$.
For $M\in \Sbimod{S_{1}}{S_{2}}$, we put $M_{\ge x} = M_{\{y\in W_{S_{1}}\backslash W/W_{S_{2}}\mid y\ge x\}}$.
We define $M_{>x}$ in a similar way.
\begin{defn}
We define a map $\presubscript{S_{1}}{\ch}_{S_{2}}\colon [\Sbimod{S_{1}}{S_{2}}]\to \presubscript{S_{1}}{\mathcal{H}}_{S_{2}}$ by
\[
\presubscript{S_{1}}{\ch}_{S_{2}}([M]) = v^{\ell(w_{S_{1}})}\sum_{x\in W_{S_{1}}\backslash W/W_{S_{2}}}v^{2\ell(x_{-}) - \ell(x_{+})}\grk_{\presupscript{S_{1}}{R}^{S_{2}}_{x}}(M_{\ge x}/M_{>x})\presupscript{S_{1}}{H}^{S_{2}}_{x}.
\]
When $S_{1} = S_{2} = \emptyset$, we write $\ch = \presubscript{\emptyset}{\ch}_{\emptyset}$.
\end{defn}

\begin{rem}
Theorem 1.2 in \cite{MR2844932} is not true as stated.
See \cite{MR2844932_Erratum} for the details.
Here, we use the ``quick fix'' in \cite{MR2844932_Erratum}.
\end{rem}

\begin{prop}\label{prop:subquotient independent from subsets}
Let $I_{1},I'_{1},I_{2},I'_{2}\subset W_{S_{1}}\backslash W/W_{S_{2}}$ be open subsets such that $I_{1}\setminus I_{2} = I'_{1}\setminus I'_{2}$.
Then for $M\in \Sbimod{S_{1}}{S_{2}}$ we have a natural isomorphism $M_{I_{1}}/M_{I_{2}}\simeq M_{I'_{1}}/M_{I'_{2}}$ in $\BigCat{S_{1}}{S_{2}}$.
\end{prop}
\begin{proof}
Set $I''_{1} = I_{1}\cup I'_{1}$ and $I''_{2} = I_{2}\cup I'_{2}$.
Then these are open subsets and $I''_{1}\setminus I''_{2} = I_{1}\setminus I_{2}$.
Hence it is sufficient to prove $M_{I_{1}}/M_{I_{2}}\simeq M_{I''_{1}}/M_{I''_{2}}$ and $M_{I'_{1}}/M_{I'_{2}}\simeq M_{I''_{1}}/M_{I''_{2}}$.
We prove the first one and the second one follows from the same argument.
Therefore we may assume $I'_{1} = I''_{1}$, $I'_{2} = I''_{2}$, namely $I_{1}\subset I'_{1}$ and $I_{2}\subset I'_{2}$.
In this case we have an embedding $M_{I_{1}}/M_{I_{2}}\to M_{I'_{1}}/M_{I'_{2}}$ and we prove that this is an isomorphism.
We may assume $M = \presubscript{S_{1},\emptyset}{\pi}_{S_{2},\emptyset,*}(N)$ for some $N\in \Sbimod{}{}$.
Let $\pi\colon W\to W_{S_{1}}\backslash W/W_{S_{2}}$ be the natural projection.
Then we have
\[
M_{I_{1}}/M_{I_{2}}\simeq \presubscript{S_{1},\emptyset}{\pi}_{S_{2},\emptyset,*}(N_{\pi^{-1}(I_{1})}/N_{\pi^{-1}(I_{2})}).
\]
Therefore the proposition follows from the case $S_{1} = S_{2} = \emptyset$~\cite[Lemma~2.34]{MR4620135}.
\end{proof}
\begin{lem}
We have $\presubscript{S_{1}}{\ch}_{S_{2}}([\presupscript{S_{1}}{R}^{S_{2}}_{1}]) = v^{-\ell(w_{S_{2}}) + \ell(w_{S_{1}\cap S_{2}})}\presupscript{S_{1}}{H}^{S_{2}}_{1}$.
\end{lem}
\begin{proof}
It follows from the definitions.
\end{proof}

Recall that $R\otimes_{R^{S_{1}}}R = \presupscript{\emptyset}{R}^{S_{1}}_{1}\otimes_{S_{1}}\presupscript{S_{1}}{R}^{\emptyset}\in \mathcal{S}$.

\begin{lem}\label{lem:graded rank of R otimes R over R^W}
We have $\ch(R\otimes_{S_{1}}R) = \sum_{w\in W_{S_{1}}}v^{-2\ell(w)}$.
\end{lem}
\begin{proof}
Take $(F_{w})$ as in Proposition~\ref{prop:generic decomposition of R otimes R}.
Let $x\in W_{S_{1}}$.
By Proposition~\ref{prop:equivariant Demazure basis}, the space $(R\otimes_{R^{S_{1}}}R)_{\ge x}/(R\otimes_{R^{S_{1}}}R)_{> x}$ has a basis $\{\partial_{x^{-1}w_{S_{1}}}(F_{w_{S_{1}}})\}$ and therefore it is a graded free of the graded rank $-2\ell(x)$.
The lemma follows.
\end{proof}

By \cite{MR4321542} and \cite[Lemma~2.33]{MR4620135}, $\ch$ is a $\Z[v,v^{-1}]$-algebra isomorphism.
The aim of this subsection is to prove the following theorem.
\begin{thm}\label{thm:character and convolution}
Let $S_{1},S_{2},S_{3}\subset S$ be subsets which satisfy Assumption~\ref{assump:assumption}.
\begin{enumerate}
\item The map $\presubscript{S_{1}}{\ch}_{S_{2}}$ is an isomorphism of $\Z[v,v^{-1}]$-modules.
\item The diagram
\[
\begin{tikzcd}[]
[\Sbimod{S_{1}}{S_{2}}]\times [\Sbimod{S_{2}}{S_{3}}] \arrow[r,"\otimes_{S_{2}}"]\arrow[d,"\presubscript{S_{1}}{\ch}_{S_{2}}\times \presubscript{S_{2}}{\ch}_{S_{3}}"'] &[] [\Sbimod{S_{1}}{S_{3}}]\arrow[d,"\presubscript{S_{1}}{\ch}_{S_{3}}"]\\
\presubscript{S_{1}}{\mathcal{H}}_{S_{2}}\times \presubscript{S_{2}}{\mathcal{H}}_{S_{3}} \arrow[r,"*_{S_{2}}"] & \presubscript{S_{1}}{\mathcal{H}}_{S_{3}}
\end{tikzcd}
\]
is commutative.
\end{enumerate}
\end{thm}

(1) follows from Theorem~\ref{thm:classification of indecomposables}.
Indeed $[\Sbimod{S_{1}}{S_{2}}]$ has a basis $[\presubscript{S_{1}}{B_{S_{2}}(x)}]$ with $x\in W_{S_{1}}\backslash W/W_{S_{2}}$ and we have 
\[
\presubscript{S_{1}}{\ch}_{S_{2}}([\presubscript{S_{1}}{B_{S_{2}}}(x)])\in v^{k}\presupscript{S_{1}}{H}^{S_{2}}_{x} + \sum_{y < x}\Z[v,v^{-1}]\presupscript{S_{1}}{H}_{y}^{S_{2}}.
\]
for some $k\in \Z$ (later we will prove $k = 0$).
Since $\{\presupscript{S_{1}}{H}^{S_{2}}_{x}\mid x\in W_{S_{1}}\backslash W/W_{S_{2}}\}$ is a basis of $\presubscript{S_{1}}{\mathcal{H}}_{S_{2}}$, $\{\presubscript{S_{1}}{\ch}_{S_{2}}([\presubscript{S_{1}}{B_{S_{2}}}(x)])\mid x\in W_{S_{1}}\backslash W/W_{S_{2}}\}$ is a basis.
Hence $\presubscript{S_{1}}{\ch}_{S_{2}}$ is an isomorphism.

We prove (2).
First, we prove the special case of (2).
Let $\Q(v)$ be the field of fractions of $\Z[v,v^{-1}]$.
\begin{prop}\label{prop:ch and push}
We have $\presubscript{S_{1}}{\ch}_{S_{2}}([\presubscript{S_{1},\emptyset}{\pi}_{S_{2},\emptyset,*}(M)]) = v^{-\ell(w_{S_{2}})}\underline{H}_{w_{S_{1}}}\ch([M]))\underline{H}_{w_{S_{2}}}$ for $M\in \Sbimod{}{}$.
In particular, the map $[\Sbimod{}{}]\to [\Sbimod{S_{1}}{S_{2}}]$ induced by $\presubscript{S_{1},\emptyset}{\pi}_{S_{2},\emptyset,*}$ is surjective after tensoring with $\Q(v)$.
\end{prop}
\begin{proof}
Let $\pi\colon W\to W_{S_{1}}\backslash W/W_{S_{2}}$ be the natural projection and $x\in W_{S_{1}}\backslash W/W_{S_{2}}$.
Set $I_{\ge x} = \{y\in W_{S_{1}}\backslash W/W_{S_{2}}\mid y\ge x\}$ and $I_{> x} = \{y\in W_{S_{1}}\backslash W/W_{S_{2}}\mid y > x\}$.
We have
\begin{align*}
\presubscript{S_{1},\emptyset}{\pi}_{S_{2},\emptyset,*}(M)_{\ge x}/\presubscript{S_{1},\emptyset}{\pi}_{S_{2},\emptyset,*}(M)_{> x}
\simeq
\presubscript{S_{1},\emptyset}{\pi}_{S_{2},\emptyset,*}(M_{\pi^{-1}(I_{\ge x})}/M_{\pi^{-1}(I_{> x})}).
\end{align*}
The module $M_{\pi^{-1}(I_{\ge x})}/M_{\pi^{-1}(I_{> x})}$ is an $R$-bimodule and $\presupscript{S_{1}}{R}^{S_{2}}_{x}$-module structure on the left hand side is given by the restriction with respect to $\presupscript{S_{1}}{R}^{S_{2}}_{x}\hookrightarrow R$.
The $\presupscript{S_{1}}{R}^{S_{2}}_{x}$-module $R$ is graded free of the graded rank $A_{x} = \sum_{w\in W_{S_{1}}\cap x_{-}W_{S_{2}}x_{-}^{-1}}v^{-2\ell(w)}$.
Hence 
\[
\grk_{\presupscript{S_{1}}{R}^{S_{2}}_{x}}(\presubscript{S_{1},\emptyset}{\pi}_{S_{2},\emptyset,*}(M)_{\ge x}/\presubscript{S_{1},\emptyset}{\pi}_{S_{2},\emptyset,*}(M)_{> x})
= A_{x}\grk_{R}(M_{\pi^{-1}(I_{\ge x})}/M_{\pi^{-1}(I_{> x})}).
\]
Take open subsets $\pi^{-1}(I_{> x}) = J_{0}\subset J_{1}\subset\cdots\subset J_{r} = \pi^{-1}(I_{\ge x})$ such that $J_{i}\setminus J_{i - 1} = \{w_{i}\}$.
Then we have $M_{J_{i}}/M_{J_{i - 1}}\simeq M_{\ge w_{i}}/M_{>w_{i}}$ by Proposition~\ref{prop:subquotient independent from subsets}.
We have $\{w_{1},\ldots,w_{r}\} = \pi^{-1}(x) = x$.
Hence
\[
\grk_{R}(M_{\pi^{-1}(I_{\ge x})}/M_{\pi^{-1}(I_{> x})})
=  
\sum_{i = 1}^{r}\grk_{R}(M_{\ge w_{i}}/M_{> w_{i}})
= \sum_{w\in x}\grk_{R}(M_{\ge w}/M_{> w})
\]
By \cite[Lemma~2.9]{MR2844932}, 
\[
v^{2\ell(x_{-}) - \ell(x_{+})}A_{x}\presupscript{S_{1}}{H}^{S_{2}}_{x}
=
v^{\ell(w)- \ell(w_{S_{1}}) - \ell(w_{S_{2}})}\underline{H}_{w_{S_{1}}}H_{w}\underline{H}_{w_{S_{2}}}.
\]
for any $w\in x$.
Hence
\begin{align*}
& \presubscript{S_{1}}{\ch}_{S_{2}}([\presubscript{S_{1},\emptyset}{\pi}_{S_{2},\emptyset,*}(M)])\\
& =
v^{\ell(w_{S_{1}})}\sum_{x\in W_{S_{1}}\backslash W/W_{S_{2}}}\sum_{w\in x}\grk_{R}(M_{\ge w}/M_{>w})v^{\ell(w)- \ell(w_{S_{1}}) - \ell(w_{S_{2}})}\underline{H}_{w_{S_{1}}}H_{w}\underline{H}_{w_{S_{2}}}\\
& =
v^{-\ell(w_{S_{2}})}\underline{H}_{w_{S_{1}}}\ch([M])\underline{H}_{w_{S_{2}}}.
\end{align*}
Since
\[
\underline{H}_{w_{S_{1}}}h\underline{H}_{w_{S_{2}}} = \left(\sum_{w_{1}\in W_{S_{1}}}v^{\ell(w_{S_{1}}) - 2\ell(w_{1})}\right)\left(\sum_{w_{2}\in W_{S_{2}}}v^{\ell(w_{S_{2}}) - 2\ell(w_{2})}\right)h
\]
for $h\in \presubscript{S_{1}}{\mathcal{H}}_{S_{2}}$, we get the last statement.
\end{proof}

\begin{proof}[Proof of Theorem~\ref{thm:character and convolution}]
We prove $(\presubscript{S_{1}}{\ch}_{S_{2}}([M_{1}]))*_{S_{2}}(\presubscript{S_{2}}{\ch}_{S_{3}}([M_{2}])) = \presubscript{S_{1}}{\ch}_{S_{3}}([M_{1}\otimes_{S_{2}} M_{2}])$ for $M_{1}\in \Sbimod{S_{1}}{S_{2}}$ and $M_{2}\in \Sbimod{S_{2}}{S_{3}}$.
It is sufficient to prove this after tensoring with $\Q(v)$.
Therefore we may assume $M_{1} = \presubscript{S_{1},\emptyset}{\pi}_{S_{2},\emptyset,*}(N_{1})$ and $M_{2} = \presubscript{S_{2},\emptyset}{\pi}_{S_{3},\emptyset,*}(N_{2})$.
We have
\begin{align*}
M_{1}\otimes_{S_{2}} M_{2} = \presubscript{S_{1},\emptyset}{\pi}_{S_{2},\emptyset,*}(N_{1})\otimes_{S_{2}} \presubscript{S_{2},\emptyset}{\pi}_{S_{3},\emptyset,*}(N_{2})
& = \presupscript{S_{1}}{R}^{\emptyset}_{1}\otimes N_{1}\otimes \presupscript{\emptyset}{R}^{S_{2}}_{1}\otimes_{S_{2}} \presupscript{S_{2}}{R}^{\emptyset}_{1}\otimes N_{2}\otimes \presupscript{\emptyset}{R}^{S_{3}}_{1}\\
& = \presubscript{S_{1},\emptyset}{\pi}_{S{3},\emptyset,*}(N_{1}\otimes (R\otimes_{R^{S_{2}}}R)\otimes N_{2}).
\end{align*}
We have $\ch(R\otimes_{R^{S_{2}}}R) = v^{-\ell(w_{S_{2}})}\underline{H}_{w_{S_{2}}}$ by Lemma~\ref{lem:graded rank of R otimes R over R^W}.
Since $\ch$ is an algebra homomorphism, 
\begin{align*}
& \presubscript{S_{1}}{\ch}_{S_{3}}([M_{1}\otimes M_{2}])\\
& =
v^{-\ell(w_{S_{3}})}\underline{H}_{w_{S_{1}}}\ch([N_{1}])\ch([R\otimes_{R^{S_{2}}}R])\ch([N_{2}])\underline{H}_{w_{S_{3}}}\tag*{(Proposition~\ref{prop:ch and push})}\\
& = v^{-\ell(w_{S_{2}})-\ell(w_{S_{3}})}\underline{H}_{w_{S_{1}}}\ch([N_{1}])\underline{H}_{w_{S_{2}}}\ch([N_{2}])\underline{H}_{w_{S_{3}}}\tag*{(Lemma~\ref{lem:graded rank of R otimes R over R^W})}\\
& = v^{-\ell(w_{S_{2}})-\ell(w_{S_{3}})}\left(\underline{H}_{w_{S_{1}}}\ch([N_{1}])\underline{H}_{w_{S_{2}}}\right)*_{S_{2}}\left(\underline{H}_{w_{S_{2}}}\ch([N_{2}])\underline{H}_{w_{S_{3}}}\right)\tag*{(Definition of $*_{S_{2}}$)}\\
& = (\presubscript{S_{1}}{\ch}_{S_{2}}([M_{1}]))*_{S_{2}}(\presubscript{S_{2}}{\ch}_{S_{3}}([M_{2}])).\tag*{(Prposition~\ref{prop:ch and push})}
\end{align*}
We get the theorem.
\end{proof}

\subsection{Duality}
For $M\in \Sbimod{S_{1}}{S_{2}}$, we put $\presubscript{S_{1}}{D}_{S_{2}}(M) = \Hom_{R^{S_{2}}}(M,R^{S_{2}})$.
We define $(R^{S_{1}},R^{S_{2}})$-bimodule structure on $\presubscript{S_{1}}{D}_{S_{2}}(M)$ by $(r_{1}\varphi r_{2})(m) = \varphi(r_{1}mr_{2})$ for $r_{1}\in R^{S_{1}}$, $r_{2}\in R^{S_{2}}$, $\varphi \in \presubscript{S_{1}}{D}_{S_{2}}(M)$ and $m\in M$.
Since $M$ is finitely generated as a right $R^{S_{2}}$-module, we have
\[
\presubscript{S_{1}}{D}_{S_{2}}(M)\otimes_{R^{S_{2}}}Q^{S_{2}}\simeq \Hom_{Q^{S_{2}}}(M\otimes_{R^{S_{2}}}Q^{S_{2}},Q^{S_{2}})
= \bigoplus_{x\in W_{S_{1}}\backslash W/W_{S_{2}}}\Hom_{Q^{S_{2}}}(M_{Q}^{x},Q^{S_{2}})
\]
and we put
\[
\presubscript{S_{1}}{D}_{S_{2}}(M)_{Q}^{x} = \Hom_{Q^{S_{2}}}(M_{Q}^{x},Q^{S_{2}}).
\]
This is an $\presupscript{S_{1}}{Q}^{S_{2}}_{x}$-module via $(r\varphi)(m) = \varphi(rm)$ for $r\in \presupscript{S_{1}}{Q}^{S_{2}}_{x}$, $\varphi\in \presubscript{S_{1}}{D}_{S_{2}}(M)_{Q}^{x}$ and $m\in M_{Q}^{x}$.
Such data satisfies the conditions of an object in $\BigCat{S_{1}}{S_{2}}$ except $Q^{S_{1}}\otimes_{R^{S_{1}}}\presubscript{S_{1}}{D}_{S_{2}}(M)\xrightarrow{\sim}Q^{S_{1}}\otimes_{R^{S_{1}}}\presubscript{S_{1}}{D}_{S_{2}}(M)\otimes_{R^{S_{2}}}Q^{S_{2}}$.
We check this and moreover we prove that $\presubscript{S_{1}}{D}_{S_{2}}(M)\in \Sbimod{S_{1}}{S_{2}}$.

\begin{lem}\label{lem:dual and push}
Let $M$ be an $(R^{S_{1}},R)$-bimodule and regard this as an $(R^{S_{1}},R^{S_{2}})$-bimodule by the restriction.
Then $\Hom_{R^{S_{2}}}(M,R^{S_{2}})\simeq \Hom_{R}(M,R)(2\ell(w_{S_{2}}))$.
Therefore if $N\in \Sbimod{}{}$, we have $\presubscript{S_{1}}{D}_{S_{2}}(\presubscript{S_{1},\emptyset}{\pi}_{S_{2},\emptyset,*}(N))\simeq \presubscript{S_{1},\emptyset}{\pi}_{S_{2},\emptyset,*}(\presubscript{\emptyset}{D}_{\emptyset}(N))(2\ell(w_{S_{2}}))$.
\end{lem}
\begin{proof}
We have
\[
\Hom_{R^{S_{2}}}(M,R^{S_{2}})\simeq \Hom_{R}(M,\Hom_{R^{S_{2}}}(R,R^{S_{2}}))
\]
and $\Hom_{R^{S_{2}}}(R,R^{S_{2}})\simeq R(2\ell(w_{S_{2}}))$ by Lemma~\ref{lem:Frobenius extension}.
\end{proof}
Since $\presubscript{\emptyset}{D}_{\emptyset}$ preserves $\Sbimod{}{}$ and $\presubscript{S_{1},\emptyset}{\pi}_{S_{2},\emptyset,*}$ is a functor from $\Sbimod{}{}$ to $\Sbimod{S_{1}}{S_{2}}$, we get $\presubscript{S_{1}}{D}_{S_{2}}(M)\in \Sbimod{S_{1}}{S_{2}}$ if $M = \pi_{S_{1},\emptyset}{\pi}_{S_{2},\emptyset,*}(N)$ for $N\in \Sbimod{}{}$.
Hence $\presubscript{S_{1}}{D}_{S_{2}}(\Sbimod{S_{1}}{S_{2}})\subset \Sbimod{S_{1}}{S_{2}}$.
Since any $M\in \Sbimod{S_{1}}{S_{2}}$ is graded free as a right $R^{S_{2}}$-module (Remark~\ref{rem:Soergel bimodule is torsion free from one-side}), we have $\presubscript{S_{1}}{D}_{S_{2}}^{2}\simeq \id$.

When $S_{1} = S_{2} = \emptyset$, we write $D$ for $\presubscript{\emptyset}{D}_{\emptyset}$.

For each $a = a(v)\in \Z[v,v^{-1}]$ we define $\overline{a} = a(v^{-1})$.
We say that a map $f\colon M\to N$ between $\Z[v,v^{-1}]$-modules is anti $\Z[v,v^{-1}]$-linear if $f(am) = \overline{a}f(m)$ for $a\in \Z[v,v^{-1}]$ and $m\in M$.
We define $\overline{h} = \sum_{w\in W}\overline{a_{w}}H^{-1}_{w^{-1}}$ for $h = \sum_{w\in W}a_{w}H_{w}\in \mathcal{H}$.
It is obviously anti $\Z[v,v^{-1}]$-linear and it is easy to see that this is a $\Z$-algebra homomorphism.
We have $\overline{\underline{H}_{w_{S_{i}}}} = \underline{H}_{w_{S_{i}}}$ for $i  = 1,2$.
Hence $\overline{\presubscript{S_{1}}{\mathcal{H}}_{S_{2}}} = \presubscript{S_{1}}{\mathcal{H}}_{S_{2}}$.
\begin{lem}\label{lem:bar is triangular}
For $x\in W_{S_{1}}\backslash W/W_{S_{2}}$, $\overline{\presupscript{S_{1}}{H}^{S_{2}}_{x}}\in \presupscript{S_{1}}{H}^{S_{2}}_{x} + \sum_{y < x}\Z[v,v^{-1}]\presupscript{S_{1}}{H}^{S_{2}}_{y}$.
\end{lem}
\begin{proof}
If $S_{1} = S_{2} = \emptyset$ then the lemma follows from induction on $\ell(x)$.
In general, since $\presupscript{S_{1}}{H}^{S_{2}}_{x} \in H_{x_{+}} + \sum_{y' < x_{+}}\Z[v,v^{-1}]H_{y'}$, we have $\overline{\presupscript{S_{1}}{H}^{S_{2}}_{x}} \in H_{x_{+}} + \sum_{y' < x_{+}}\Z[v,v^{-1}]H_{y'}$.
The left hand side is in $\presubscript{S_{1}}{\mathcal{H}}_{S_{2}}$, hence it can be written as $\sum_{y\in W_{S_{1}}\backslash W/W_{S_{2}}}c_{y}\presupscript{S_{1}}{H}^{S_{2}}_{y}$ for some $c_{y}\in \Z[v,v^{-1}]$.
Comparing two descriptions, $c_{x} = 1 $ and $c_{y}\ne 0$ only when there exists $y'\in y$ such that $y' < x_{+}$, namely only when $y < x$.
\end{proof}

\begin{prop}
Let $M\in \Sbimod{S_{1}}{S_{2}}$.
Then we have $\presubscript{S_{1}}{\ch}_{S_{2}}([\presubscript{S_{1}}{D}_{S_{2}}(M)]) = \overline{\presubscript{S_{1}}{\ch}_{S_{2}}([M])}$.
\end{prop}
\begin{proof}
Both sides define anti $\Z[v,v^{-1}]$-linear maps $[\presubscript{S_{1}}{S}_{S_{2}}]\to \presubscript{S_{1}}{\mathcal{H}}_{S_{2}}$.
Hence it is sufficient to prove the equality after tensoring with $\Q(v)$.
By Proposition~\ref{prop:ch and push}, we may assume $M = \presubscript{S_{1},\emptyset}{\pi}_{S_{2},\emptyset,*}(N)$ for some $N\in \Sbimod{}{}$.
We have
\begin{align*}
\presubscript{S_{1}}{\ch}_{S_{2}}([\presubscript{S_{1}}{D}_{S_{2}}(M)])
& = v^{2\ell(w_{S_{2}})}\presubscript{S_{1}}{\ch}_{S_{2}}([\presubscript{S_{1},\emptyset}{\pi}_{S_{2},\emptyset,*}(D(N))]) \tag*{(Lemma~\ref{lem:dual and push})}\\
& = v^{\ell(w_{S_2})}\underline{H}_{w_{S_{1}}}\ch([D(N)])\underline{H}_{w_{S_{2}}} \tag*{(Proposition~\ref{prop:ch and push})}\\
& = v^{\ell(w_{S_2})}\underline{H}_{w_{S_{1}}}\overline{\ch([N])}\underline{H}_{w_{S_{2}}} \tag*{\cite[Proposition~3.29]{MR4620135}}\\
& = v^{\ell(w_{S_2})}\overline{\underline{H}_{w_{S_{1}}}\ch([N])\underline{H}_{w_{S_{2}}}} \\
& = \overline{\presubscript{S_{1}}{\ch}_{S_{2}}([M])}. \tag*{(Proposition~\ref{prop:ch and push})}
\end{align*}
We get the proposition.
\end{proof}

\begin{prop}
Let $x\in W_{S_{1}}\backslash W/W_{S_{2}}$.
\begin{enumerate}
\item We have $\presubscript{S_{1}}{D}_{S_{2}}(\presubscript{S_{1}}{B}_{S_{2}}(x))\simeq \presubscript{S_{1}}{B}_{S_{2}}(x)$.
\item We have $\presubscript{S_{1}}{\ch}_{S_{2}}(\presubscript{S_{1}}{B}_{S_{2}}(x)) \in \presupscript{S_{1}}{H}^{S_{2}}_{x} + \sum_{y < x}\Z[v,v^{-1}]\presupscript{S_{1}}{H}^{S_{2}}_{y}$.
\end{enumerate}
\end{prop}
\begin{proof}
We prove (1).
If $S_{1} = S_{2} = \emptyset$ then this is \cite[Theorem~4.1 (3)]{MR4321542}.

The support of $\presubscript{S_{1}}{D}_{S_{2}}(\presubscript{S_{1}}{B}_{S_{2}}(x))$ is contained in $\{y\in W_{S_{1}}\backslash W/W_{S_{2}}\mid y\le x\}$ and contains $x$.
Moreover $\presubscript{S_{1}}{D}_{S_{2}}(\presubscript{S_{1}}{B}_{S_{2}}(x))$ is indecomposable.
Therefore there exists $m\in \Z$ such that $\presubscript{S_{1}}{D}_{S_{2}}(\presubscript{S_{1}}{B}_{S_{2}}(x))\simeq \presubscript{S_{1}}{B}_{S_{2}}(x)(m)$.
Set $I = S_{1}\cap x_{-}S_{2}x_{-}^{-1}$.
Then we have $W_{I} = W_{S_{1}}\cap x_{-}W_{S_{2}}x_{-}^{-1}$~\cite[Theorem~2.2]{MR2844932}.
By the support there exists $B,M\in \Sbimod{S_{1}}{S_{2}}$ such that $\presubscript{S_{1},\emptyset}{\pi}_{S_{2},\emptyset,*}(B(x_{-}))\simeq B\oplus M$, $B$ is a direct sum of shifts of $\presubscript{S_{1}}{B}_{S_{2}}(x)$ and $\supp(M)\subset \{y\in W_{S_{1}}\backslash W/W_{S_{2}}\mid y < x\}$.
Then $\presubscript{S_{1},\emptyset}{\pi}_{S_{2},\emptyset,*}(B(x_{-}))^{x} \simeq B^{x}$.
We have $\presubscript{S_{1},\emptyset}{\pi}_{S_{2},\emptyset,*}(B(x_{-}))^{x} \simeq \presubscript{S_{1},\emptyset}{\pi}_{S_{2},\emptyset,*}(B(x_{-})^{\pi^{-1}(x)})$.
Since $\supp_{W}(B(x_{-}))\subset \{y\in W\mid y\le x_{-}\}$, we have $\supp_{W}(B(x_{-}))\cap \pi^{-1}(x) = \{x_{-}\}$.
Therefore 
\[
B^{x}\simeq \presubscript{S_{1},\emptyset}{\pi}_{S_{2},\emptyset,*}(B(x_{-}))^{x}\simeq B(x_{-})^{x_{-}}\simeq R(\ell(x_{-}))\simeq \bigoplus_{w\in W_{I}}\presupscript{S_{1}}{R}^{S_{2}}_{x}(\ell(x_{-}) - 2\ell(w))).
\]
Therefore we have 
\[
\presubscript{S_{1},\emptyset}{\pi}_{S_{2},\emptyset,*}(B(x_{-}))\simeq \bigoplus_{w\in W_{I}}\presubscript{S_{1}}{B}_{S_{2}}(x)(\ell(x_{-}) - \ell(x_{+}) + \ell(w_{S_{1}})- 2\ell(w))\oplus M.
\]
Since $\ell(x_{+}) - \ell(x_{-}) = \ell(w_{S_{1}}) + \ell(w_{S_{2}}) - \ell(w_{I})$, we have 
\[
\presubscript{S_{1},\emptyset}{\pi}_{S_{2},\emptyset,*}(B(x_{-}))\simeq \bigoplus_{w\in W_{I}}\presubscript{S_{1}}{B}_{S_{2}}(x)(-\ell(w_{S_{2}}) + \ell(w_{I}) - 2\ell(w))\oplus M.
\]
By taking the dual of both sides, by Lemma~\ref{lem:dual and push} and $D(B(x_{-}))\simeq B(x_{-})$, we have
\[
\presubscript{S_{1},\emptyset}{\pi}_{S_{2},\emptyset,*}(B(x_{-}))(2\ell(w_{S_{2}}))\simeq \bigoplus_{w\in W_{I}}\presubscript{S_{1}}{B}_{S_{2}}(x)(\ell(w_{S_{2}}) - \ell(w_{I}) + 2\ell(w) + m)\oplus \presubscript{S_{1}}{D}_{S_{2}}(M).
\]
Therefore we have
\[
\{\ell(w_{S_{2}}) + \ell(w_{I}) - 2\ell(w)\mid w\in W_{I}\} = \{\ell(w_{S_{2}}) - \ell(w_{I}) + 2\ell(w) + m\mid w\in W_{I}\}.
\]
The maximal numbers of both sides are
\[
\ell(w_{S_{2}}) + \ell(w_{I}) = \ell(w_{S_{2}}) - \ell(w_{I}) + 2\ell(w_{I}) + m.
\]
Hence $m = 0$.

Set $h = \presubscript{S_{1}}{\ch}_{S_{2}}(\presubscript{S_{1}}{B}_{S_{2}}(x))$.
We have $\supp_{W}(\presubscript{S_{1}}{B}_{S_{2}}(x))\subset \{y\in W_{S_{1}}\backslash W/W_{S_{2}}\mid y\le x\}$.
We also have $\presubscript{S_{1}}{B}_{S_{2}}(x)_{Q}^{x}\simeq \presupscript{S_{1}}{Q}^{S_{2}}$ if we ignore the grading.
Hence $\presubscript{S_{1}}{B}_{S_{2}}(x)_{\ge x}/\presubscript{S_{1}}{B}_{S_{2}}(x)_{> x}\simeq \presupscript{S_{1}}{R}^{S_{2}}_{x}$ after ignoring the grading.
Therefore there exists $k\in\Z$ such that $h \in v^{k}\presupscript{S_{1}}{H}^{S_{2}}_{x} + \sum_{y < x}\Z[v,v^{-1}]\presupscript{S_{1}}{H}^{S_{2}}_{y}$.
We have $\overline{h}\in v^{-k}\presupscript{S_{1}}{H}^{S_{2}}_{x} + \sum_{y < x}\Z[v,v^{-1}]\presupscript{S_{1}}{H}^{S_{2}}_{y}$ by Lemma~\ref{lem:bar is triangular}.
On the other hand, by (1) and the previous proposition, we have $\overline{h} = h$.
Hence $k = 0$.
\end{proof}

\begin{prop}
Let $S'_{1}\subset S_{1}$ and $S'_{2}\subset S_{2}$ be subsets and $x\in W_{S_{1}}\backslash W/W_{S_{2}}$.
Put $x' = W_{S'_{1}}x_{+}W_{S'_{2}}\in W_{S'_{1}}\backslash W/W_{S'_{2}}$.
Then we have $\presubscript{S_{1},S'_{1}}{\pi}_{S_{2},S'_{2}}^{*}(\presubscript{S_{1}}{B}_{S_{2}}(x)) \simeq \presubscript{S'_{1}}{B}_{S'_{2}}(x')(-\ell(w_{S_{1}}) + \ell(w_{S'_{1}}))$
\end{prop}
\begin{proof}
We have 
\begin{align*}
\presubscript{S'_{1}}{\ch}_{S'_{2}}(\presubscript{S_{1},S'_{1}}{\pi}_{S_{2},S'_{2}}^{*}(\presubscript{S_{1}}{B}_{S_{2}}(x)))
& = \presubscript{S'_{1}}{\ch}_{S'_{2}}(\presupscript{S'_{1}}{R}^{S_{1}}_{1}\otimes_{S_{1}} \presubscript{S_{1}}{B}_{S_{2}}(x)\otimes_{S_{2}} \presupscript{S_{2}}{R}^{S'_{2}}_{1})\\
& = v^{-\ell(w_{S_{1}}) + \ell(w_{S'_{1}})}\presupscript{S'_{1}}{H}^{S_{1}}_{1}*_{S_{1}}(\presubscript{S_{1}}{\ch}_{S_{2}}(\presubscript{S_{1}}{B}_{S_{2}}(x)))*_{S_{2}}\presupscript{S_{2}}{H}^{S'_{2}}_{1}.
\end{align*}
We have $\presubscript{S_{1}}{\ch}_{S_{2}}(\presubscript{S_{1}}{B}_{S_{2}}(x))\in \presupscript{S_{1}}{H}^{S_{2}}_{x} + \sum_{y < x}\Z[v,v^{-1}]\presupscript{S_{1}}{H}^{S_{2}}_{y}$.
Applying \cite[Proposition~2.8]{MR2844932} twice, we have 
\[
\presupscript{S'_{1}}{H}^{S_{1}}_{1}*_{S_{1}}\presupscript{S_{1}}{H}^{S_{2}}_{y}*_{S_{2}}\presupscript{S_{2}}{H}^{S'_{2}}_{1}
=
\sum_{y'\in W_{S'_{1}}\backslash y/W_{S'_{2}}}v^{\ell(y_{+}) - \ell(y'_{+})}\presupscript{S'_{1}}{H}^{S'_{2}}_{y'}.
\]
Hence 
\begin{multline}\label{eq:char of pull of indecomposable}
\presubscript{S'_{1}}{\ch}_{S'_{2}}(\presubscript{S_{1},S'_{1}}{\pi}_{S_{2},S'_{2}}^{*}(\presubscript{S_{1}}{B}_{S_{2}}(x)))\\
\in v^{-\ell(w_{S_{1}}) + \ell(w_{S'_{1}})}\sum_{z'\in W_{S'_{1}}\backslash x/W_{S'_{2}}}v^{\ell(x_{+}) - \ell(z'_{+})}\presupscript{S'_{1}}{H}^{S'_{2}}_{z'} + \sum_{y < x,y'\in W_{S'_{1}}\backslash y/W_{S'_{2}}}\Z[v,v^{-1}]\presupscript{S'_{1}}{H}^{S'_{2}}_{y'}.
\end{multline}
The element $x'$ is the maximal element in $ W_{S'_{1}}\backslash x/W_{S'_{2}}$ and we have $\ell(x_{+}) = \ell(x'_{+})$.
Moreover if $y < x$ and $y' \in W_{S'_{1}}\backslash y/W_{S'_{2}}$ then $y' < x'$.
Hence 
\[
\presubscript{S'_{1}}{\ch}_{S'_{2}}(\presubscript{S_{1},S'_{1}}{\pi}_{S_{2},S'_{2}}^{*}(\presubscript{S_{1}}{B}_{S_{2}}(x)))
\in v^{-\ell(w_{S_{1}}) + \ell(w_{S'_{1}})}\presupscript{S'_{1}}{H}^{S'_{2}}_{x'} + \sum_{y' < x'}\Z[v,v^{-1}]\presupscript{S'_{1}}{H}^{S'_{2}}_{y'}.
\]
Therefore $\presubscript{S_{1},S'_{1}}{\pi}_{S_{2},S'_{2}}^{*}(\presubscript{S_{1}}{B}_{S_{2}}(x))$ contains $\presubscript{S'_{1}}{B}_{S'_{2}}(x')(-\ell(w_{S_{1}}) + \ell(w_{S'_{1}}))$ as a direct summand.
Hence it is sufficient to prove $\presubscript{S_{1},S'_{1}}{\pi}_{S_{2},S'_{2}}^{*}(\presubscript{S_{1}}{B}_{S_{2}}(x))$ is indecomposable.
It is sufficient to prove that the object is indecomposable after applying $\presubscript{S'_{1},\emptyset}{\pi}_{S'_{2},\emptyset}^{*}$.
Therefore we may assume $S'_{1} = S'_{2} = \emptyset$.

We have $x' = x_{+}$.
We take $M\in \Sbimod{}{}$ such that
\[
\presubscript{S_{1},\emptyset}{\pi}_{S_{2},\emptyset}^{*}(\presubscript{S_{1}}{B}_{S_{2}}(x))
=
B(x_{+})(-\ell(w_{S_{1}}))\oplus M.
\]
Let $c_{w}\in\Z[v.v^{-1}]$ be the coefficient of $H_{w}$ in $\ch(B(x_{+}))$.
If $s\in S_{1}$ then $sx_{+}<x_{+}$.
Hence by \cite[Lemma~4.3]{MR3611719}, $c_{w} = vc_{sw}$ for $w\in W$ such that $sw < w$.
We also have a similar formula with respect to the right action of $W_{S_{2}}$.
Therefore $c_{w} = v^{\ell(x_{+}) - \ell(w)}$ for $w\in x$.
Hence we have 
\[
\ch(B(x_{+})(-\ell(w_{S_{1}})))
=
v^{-\ell(w_{S_{1}})}\sum_{w\in x}v^{\ell(x_{+}) - \ell(w)}H_{w} + \sum_{y < x,y'\in y}\Z[v,v^{-1}]H_{y'}.
\]
Therefore, by \eqref{eq:char of pull of indecomposable}, $\ch(M)\in \sum_{y < x,y'\in y}\Z[v,v^{-1}]H_{y'}$.
In particular, the coefficient of $\presupscript{S_{1}}{H}^{S_{2}}_{x}$ in $\presubscript{S_{1}}{\ch}_{S_{2}}([\presubscript{S_{1},\emptyset}{\pi}_{S_{2},\emptyset,*}(M)])$ is zero.
On the other hand, $\presubscript{S_{1},\emptyset}{\pi}_{S_{2},\emptyset,*}(M)$ is a direct summand of
\[
\presubscript{S_{1},\emptyset}{\pi}_{S_{2},\emptyset,*}(\presubscript{S_{1},\emptyset}{\pi}_{S_{2},\emptyset}^{*}(\presubscript{S_{1}}{B}_{S_{2}}(x)))
\simeq
\bigoplus_{w_{1}\in W_{S_{1}},w_{2}\in W_{S_{2}}}\presubscript{S_{1}}{B}_{S_{2}}(x)(-2(\ell(w_{1}) + \ell(w_{2}))).
\]
and therefore $\presubscript{S_{1},\emptyset}{\pi}_{S_{2},\emptyset,*}(M)$ should be a direct sum of $\presubscript{S_{1}}{B}_{S_{2}}(x)$ up to a grading shift.
Hence $M$ must be zero.
\end{proof}

\subsection{Hom formula}
We define $\omega\colon \mathcal{H}\to \mathcal{H}$ by $\omega(\sum_{w\in W}a_{w}H_{w}) = \sum_{w\in W}\overline{a_{w}}H_{w}^{-1}$.
This is an anti-involution and $\omega(\underline{H}_{w_{S_{1}}}) = \underline{H}_{w_{S_{1}}}$.
Hence $\omega(\presubscript{S_{1}}{\mathcal{H}}_{S_{2}}) = \presubscript{S_{2}}{\mathcal{H}}_{S_{1}}$.
For $h\in \mathcal{H}$, let $\varepsilon(h)$ be the coefficient of $H_{1}$ in $h$ and we put $\overline{\varepsilon}(h) = \overline{\varepsilon(\overline{h})}$.
It is not difficult to see that $\varepsilon(h_{1}h_{2}) = \varepsilon(h_{2}h_{1})$.

\begin{thm}
Let $M_{1},M_{2}\in \Sbimod{S_{1}}{S_{2}}$.
Then $\Hom_{\Sbimod{S_{1}}{S_{2}}}(M_{1},M_{2})$ is a graded free $R^{S_{1}}$-module and we have
\[
\grk_{R^{S_{1}}}\Hom_{\Sbimod{S_{1}}{S_{2}}}(M_{1},M_{2}) = 
v^{\ell(w_{S_{2}})}\overline{\varepsilon}(\presubscript{S_{1}}{\ch}_{S_{2}}([M_{2}])*_{S_{2}}\omega(\presubscript{S_{1}}{\ch}_{S_{2}}([M_{1}])))
\]
\end{thm}

\begin{proof}
We may assume $M_{2} = \presubscript{S_{1},\emptyset}{\pi}_{S_{2}.\emptyset,*}(N_{2})$ for $N_{2}\in \Sbimod{}{}$.
The space
\begin{align*}
\Hom(M_{1},M_{2})\simeq \Hom(\presubscript{S_{1},\emptyset}{\pi}_{S_{2},\emptyset}^{*}(M_{1}),N_{2})
\end{align*}
is a graded free $R$-module since $\presubscript{S_{1},\emptyset}{\pi}_{S_{2},\emptyset}^{*}(M_{1}),N_{2}\in \Sbimod{}{}$~\cite[Theorem~4.6]{MR4321542}.
Moreover, the graded rank as an $R$-module is $\overline{\varepsilon}(\ch([N_{2}])\omega(\ch([\presubscript{S_{1},\emptyset}{\pi}_{S_{2},\emptyset}^{*}(M_{1})]))$~\cite[Theorem~4.6]{MR4321542}.
We have
\begin{align*}
\ch([\presubscript{S_{1},\emptyset}{\pi}_{S_{2},\emptyset}^{*}(M_{1})])
& =
\ch([\presupscript{\emptyset}{R}^{S_{1}}_{1}\otimes_{S_{1}}M_{1}\otimes_{S_{2}}\presupscript{S_{2}}{R}^{\emptyset}_{1}])\\
& = v^{-\ell(w_{S_{1}})}\presupscript{\emptyset}{H}^{S_{1}}_{1} *_{S_{1}}(\presubscript{S_{1}}{\ch}_{S_{2}}([M_{1}]))*_{S_{2}}\presupscript{S_{2}}{H}^{\emptyset}_{1}.
\end{align*}
We have $\presupscript{\emptyset}{H}^{S_{1}} = \underline{H}_{w_{S_{1}}}$ and $\underline{H}_{w_{S_{1}}}*_{S_{1}}h = h$ for any $h\in \presubscript{S_{1}}{\mathcal{H}}_{S_{2}}$.
We have a similar formula for the right multiplication of $\presupscript{S_{2}}{H}^{\emptyset}_{1}$.
Hence $\ch([\presubscript{S_{1},\emptyset}{\pi}_{S_{2},\emptyset}^{*}(M_{1})]) = v^{-\ell(w_{S_{1}})}\presubscript{S_{1}}{\ch}_{S_{2}}([M_{1}])$.
(The left hand side is in $\mathcal{H}$ and the right hand side is in $\presubscript{S_{1}}{\mathcal{H}}_{S_{2}}$ which is a subspace of $\mathcal{H}$ by the definition.)
Since $\omega$ is anti $\Z[v,v^{-1}]$-linear
\[
\grk_{R}(\Hom(N_{1},\presubscript{S_{1},\emptyset}{\pi}_{S_{2},\emptyset}^{*}(M_{2}))) = v^{\ell(w_{S_{1}})}\overline{\varepsilon}(\ch([N_{2}])\omega(\presubscript{S_{1}}{\ch}_{S_{2}}([M_{1}]))).
\]
The graded rank of $R$ as an $R^{S_{1}}$-module is $\sum_{w\in W_{S_{1}}}v^{-2\ell(w)}$.
Hence we have
\begin{align*}
\grk_{R^{S_{1}}}(\Hom(N_{1},\presubscript{S_{1},\emptyset}{\pi}_{S_{2},\emptyset}^{*}(M_{2}))) & = v^{\ell(w_{S_{1}})}\left(\sum_{w\in W_{S_{1}}}v^{-2\ell(w)}\right)\overline{\varepsilon}(\ch([N_{2}])\omega(\presubscript{S_{1}}{\ch}_{S_{2}}([M_{1}])))\\
& = \overline{\varepsilon}(\ch(N_{2})\omega(\presubscript{S_{1}}{\ch}_{S_{2}}(M_{1}))\underline{H}_{w_{S_{1}}})
\end{align*}
since $h\underline{H}_{w_{S_{1}}} = (\sum_{w\in W_{S_{1}}}v^{\ell(w_{S_{1}})-2\ell(w)})h$ for $h\in \presubscript{S_{1}}{\mathcal{H}}_{S_{2}}$.
For $h\in \presubscript{S_{2}}{\mathcal{H}}_{S_{1}}$, we have $h\underline{H}_{w_{S_{2}}} = (\sum_{w\in W_{S_{2}}}v^{\ell(w_{S_{2}})-2\ell(w)})h$.
Hence
\begin{align*}
& \overline{\varepsilon}(\ch([N_{2}])\omega(\presubscript{S_{1}}{\ch}_{S_{2}}([M_{1}]))\underline{H}_{w_{S_{1}}})\\
& = \left(\sum_{w\in W_{S_{2}}}v^{\ell(w_{S_{2}})-2\ell(w)}\right)^{-1}\overline{\varepsilon}(\ch([N_{2}])\underline{H}_{w_{S_{2}}}\omega(\presubscript{S_{1}}{\ch}_{S_{2}}([M_{1}]))\underline{H}_{w_{S_{1}}})\\
& = \left(\sum_{w\in W_{S_{2}}}v^{\ell(w_{S_{2}})-2\ell(w)}\right)^{-1}\overline{\varepsilon}(\underline{H}_{w_{S_{1}}}\ch([N_{2}])\underline{H}_{w_{S_{2}}}\omega(\presubscript{S_{1}}{\ch}_{S_{2}}([M_{1}])))\tag*{($\varepsilon(h_{1}h_{2}) = \varepsilon(h_{2}h_{1})$)}\\
& = \overline{\varepsilon}(\underline{H}_{w_{S_{1}}}\ch([N_{2}])\underline{H}_{w_{S_{2}}}*_{S_{2}}\omega(\presubscript{S_{1}}{\ch}_{S_{2}}([M_{1}]))).\tag*{(the definition of $*_{S_{2}}$)}
\end{align*}
We get the theorem by Proposition~\ref{prop:ch and push}.
\end{proof}

\subsection{Another realization}
We continue to assume that $S_{1},S_{2}\subset S$ are subsets which satisfy Assumption~\ref{assump:assumption}.
Let $N\in \Sbimod{S_{1}}{S_{2}}$ and consider $M = \presubscript{S_{1},\emptyset}{\pi}_{S_{2},\emptyset}^{*}(N) = R\otimes_{R^{S_{1}}}N\otimes_{R^{S_{2}}}R\in \Sbimod{}{}$.
We define a left action of $w_{1}\in W_{S_{1}}$ and a right action of $w_{2}\in W_{S_{2}}$ by $w_{1}(f\otimes n\otimes g)w_{2} = w_{1}(f)\otimes n\otimes w_{2}^{-1}(g)$.
This is not a morphism in $\Sbimod{}{}$ but this action satisfies 
\begin{equation}\label{eq:compatibility of Weyl action}
w_{1}(fmg)w_{2} = w_{1}(f)(w_{1}mw_{2})w_{2}^{-1}(g),\quad w_{1}M_{Q}^{x}w_{2} = M_{Q}^{w_{1}xw_{2}}
\end{equation}
for any $m\in M$, $f,g\in R$ and $x\in W$.
Since $N$ is projective as a left $R^{S_{1}}$-module and a right $R^{S_{2}}$-module, we have $M^{W_{S_{1}}\times W_{S_{2}}} = R^{S_{1}}\otimes_{R^{S_{1}}}N\otimes_{R^{S_{2}}}R^{S_{2}} \simeq N$.
Hence we have an isomorphism
\begin{equation}\label{eq:fixed part recovers all}
R\otimes_{R^{S_{1}}}M^{W_{S_{1}}\times W_{S_{2}}}\otimes_{R^{S_{2}}}R\xrightarrow{\sim} M.
\end{equation}

\begin{defn}
Let $\Sbimod{S_{1}}{S_{2}}'$ be the category of $M\in \Sbimod{}{}$ with an action of $W_{S_{1}}\times W_{S_{2}}$ which satisfies \eqref{eq:compatibility of Weyl action}, \eqref{eq:fixed part recovers all}.
For $M_{1},M_{2}\in \Sbimod{S_{1}}{S_{2}}'$, we define $\Hom_{\Sbimod{S_{1}}{S_{2}}'}(M_{1},M_{2})$ as the space of $M_{1}\to M_{2}$ in $\Sbimod{}{}$ which commutes with actions of $W_{S_{1}}\times W_{S_{2}}$.
\end{defn}
As we have seen, $\presubscript{S_{1},\emptyset}{\pi}_{S_{2},\emptyset}^{*}$ can be regarded as a functor $\Sbimod{S_{1}}{S_{2}}\to \Sbimod{S_{1}}{S_{2}}'$.
We also write $\presubscript{S_{1},\emptyset}{\pi}_{S_{2},\emptyset}^{*}$ for this functor.
\begin{thm}\label{thm:alternative description}
The functor $\Sbimod{S_{1}}{S_{2}}\to \Sbimod{S_{1}}{S_{2}}'$ is an equivalence of categories.
\end{thm}
\begin{proof}
For $N_{1},N_{2}\in \Sbimod{S_{1}}{S_{2}}'$,  we define an action of $(w_{1},w_{2})\in W_{S_{1}}\times W_{S_{2}}$ on $\varphi\in\Hom_{\Sbimod{}{}}(M_{1},M_{2})$ by $(w_{1}\varphi w_{2})(m) = w_{1}\varphi (w_{1}^{-1}mw_{2}^{-1})w_{2}$.
Then, by the definition, we have $\Hom_{\Sbimod{S_{1}}{S_{2}}'}(M_{1},M_{2}) = \Hom_{\Sbimod{}{}}(M_{1},M_{2})^{W_{S_{1}}\times W_{S_{2}}}$.

First, we prove that the functor is fully faithful.
For $M_{1},M_{2}\in \Sbimod{S_{1}}{S_{2}}$, we have
\begin{align*}
\Hom_{\Sbimod{S_{1}}{S_{2}}'}(\presubscript{S_{1},\emptyset}{\pi}_{S_{2},\emptyset}^{*}(M_{1}),\presubscript{S_{1},\emptyset}{\pi}_{S_{2},\emptyset}^{*}(M_{2}))
&= \Hom_{\Sbimod{}{}}(\presubscript{S_{1},\emptyset}{\pi}_{S_{2},\emptyset}^{*}(M_{1}),\presubscript{S_{1},\emptyset}{\pi}_{S_{2},\emptyset}^{*}(M_{2}))^{W_{1}\times W_{2}}\\
& \simeq \Hom_{\Sbimod{S_{1}}{S_{2}}}(M_{1},\presubscript{S_{1},\emptyset}{\pi}_{S_{2},\emptyset,*}(\presubscript{S_{1},\emptyset}{\pi}_{S_{2},\emptyset}^{*}(M_{2})))^{W_{1}\times W_{2}}.
\end{align*}
In the last the action of $W_{1}\times W_{2}$ on the space of morphisms is induced by the action on $\presubscript{S_{1},\emptyset}{\pi}_{S_{2},\emptyset,*}(\presubscript{S_{1},\emptyset}{\pi}_{S_{2},\emptyset}^{*}(M_{2}))$.
Hence we have
\begin{align*}
& \Hom_{\Sbimod{S_{1}}{S_{2}}}(M_{1},\presubscript{S_{1},\emptyset}{\pi}_{S_{2},\emptyset,*}(\presubscript{S_{1},\emptyset}{\pi}_{S_{2},\emptyset}^{*}(M_{2})))^{W_{1}\times W_{2}}\\
& =
\Hom_{\Sbimod{S_{1}}{S_{2}}}(M_{1},\presubscript{S_{1},\emptyset}{\pi}_{S_{2},\emptyset,*}(\presubscript{S_{1},\emptyset}{\pi}_{S_{2},\emptyset}^{*}(M_{2}))^{W_{1}\times W_{2}}).
\end{align*}
By \eqref{eq:fixed part recovers all}, we have
\[
\presubscript{S_{1},\emptyset}{\pi}_{S_{2},\emptyset,*}(\presubscript{S_{1},\emptyset}{\pi}_{S_{2},\emptyset}^{*}(M_{2}))^{W_{1}\times W_{2}}
=
(R\otimes_{R^{S_{1}}}M_{2}\otimes_{R_{S_{2}}}R)^{W_{S_{1}}\times W_{S_{2}}}
\simeq M_{2}
\]
Hence the functor is fully faithful.

We prove that the functor is essentially surjective.
Let $M\in \Sbimod{S_{1}}{S_{2}}'$ and set $N = M^{W_{S_{1}}\times W_{S_{2}}}$.
By the second condition in \eqref{eq:compatibility of Weyl action}, for each $x\in W_{S_{1}}\backslash W/W_{S_{2}}$ the subspace $\bigoplus_{w\in x}M^{w}_{Q}$ is stable under the action of $W_{S_{1}}\times W_{S_{2}}$.
Set $N_{Q}^{x} = \left(\bigoplus_{w\in x}M^{w}_{Q}\right)^{W_{S_{1}}\times W_{S_{2}}}$.
This is a $\left(\prod_{w\in x}Q\right)^{W_{S_{1}}\times W_{S_{2}}} = \presupscript{S_{1}}{Q}^{S_{2}}_{x}$-module and it gives a structure of an object of $\BigCat{S_{1}}{S_{2}}$ to $N$.

By the construction $N\hookrightarrow M = \presubscript{S_{1},\emptyset}{\pi}_{S_{2},\emptyset,*}(M)$ is a morphism in $\presubscript{S_{1}}{\mathcal{C}}_{S_{2}}$ and the map $\Phi\colon\presubscript{S_{1},\emptyset}{\pi}_{S_{2},\emptyset}^{*}(N)\to M$ induced by the adjointness is a bijection by \eqref{eq:fixed part recovers all}.
By Proposition~\ref{prop:Demazure basis}, $N$ is a direct summand of $\presubscript{S_{1},\emptyset}{\pi}_{S_{2},\emptyset,*}(M)$, hence $N\in \Sbimod{S_{1}}{S_{2}}$.
If $(w_{1},w_{2})\in W_{S_{1}}\times W_{S_{2}}$, $f,g\in R$ and $n\in N$, then 
\begin{align*}
\Phi(w_{1}(f)\otimes n\otimes w_{2}^{-1}(g)) & = w_{1}(f)nw_{2}^{-1}(g) = w_{1}(f)(w_{1}nw_{2})w_{2}^{-1}(g)\\
&  = w_{1}(fng)w_{2} = w_{1}\Phi(f\otimes n\otimes g)w_{2}.
\end{align*}
since $n$ is $(W_{S_{1}}\times W_{S_{2}})$-invariant and by \eqref{eq:compatibility of Weyl action}.
Hence $\Phi$ commutes with $(W_{S_{1}}\times W_{S_{2}})$-action and $M\simeq \presubscript{S_{1},\emptyset}{\pi}_{S_{2},\emptyset}^{*}(N)$ in $\Sbimod{S_{1}}{S_{2}}'$.
\end{proof}

Let $S_{3}$ be another subset of $S$ satisfying Assumption~\ref{assump:assumption}.
For $M_{1}\in \Sbimod{S_{1}}{S_{2}}'$ and $M_{2}\in \Sbimod{S_{2}}{S_{3}}'$, we have an action of $W_{S_{2}}$ on $M_{1}\otimes M_{2}$ defined by $w_{2}(m_{1}\otimes m_{2}) = m_{1}w_{2}^{-1}\otimes w_{2}m_{2}$.
Then we put
\[
M_{1}\otimes_{S_{2}}'M_{2} = (M_{1}\otimes M_{2})^{W_{S_{2}}}.
\]
\begin{prop}
Let $N_{1}\in \Sbimod{S_{1}}{S_{2}}$ and $N_{2}\in \Sbimod{S_{2}}{S_{3}}$.
Then
\[
\presubscript{S_{1},\emptyset}{\pi}_{S_{3},\emptyset}^{*}(N_{1}\otimes_{S_{2}} N_{2})
\simeq
(\presubscript{S_{1},\emptyset}{\pi}_{S_{2},\emptyset}^{*}(N_{1}))\otimes_{S_{2}}'(\presubscript{S_{2},\emptyset}{\pi}_{S_{3},\emptyset}^{*}(N_{2})).
\]
In particular $M_{1}\otimes'_{S_{2}}M_{2}\in \Sbimod{S_{1}}{S_{3}}'$ for $M_{1}\in \Sbimod{S_{1}}{S_{2}}'$ and $M_{2}\in \Sbimod{S_{2}}{S_{3}}'$.
\end{prop}
\begin{proof}
We have
\begin{align*}
(\presubscript{S_{1},\emptyset}{\pi}_{S_{2},\emptyset}^{*}(N_{1}))\otimes(\presubscript{S_{2},\emptyset}{\pi}_{S_{3},\emptyset}^{*}(N_{2}))
& =
\presupscript{\emptyset}{R}^{S_{1}}_{1}\otimes_{S_{1}}N_{1}\otimes_{S_{2}}\presupscript{S_{2}}{R}^{\emptyset}_{1}\otimes \presupscript{\emptyset}{R}^{S_{2}}_{1}\otimes_{S_{2}} N_{2}\otimes_{S_{3}}\presupscript{S_{3}}{R}^{\emptyset}_{1}.
\end{align*}
On this description $W_{S_{2}}$ acts on $\presupscript{S_{2}}{R}^{\emptyset}_{1}\otimes \presupscript{\emptyset}{R}^{S_{2}}_{1}$ which is in the middle and we have
\[
(\presupscript{S_{2}}{R}^{\emptyset}_{1}\otimes \presupscript{\emptyset}{R}^{S_{2}}_{1})^{W_{S_{2}}}
=(R\otimes_{R}R)^{W_{S_{2}}}\simeq R^{S_{2}}.
\]
Hence 
\begin{align*}
(\presubscript{S_{1},\emptyset}{\pi}_{S_{2},\emptyset}^{*}(N_{1}))\otimes'_{S_{2}}(\presubscript{S_{2},\emptyset}{\pi}_{S_{3},\emptyset}^{*}(N_{2}))
& =
\presupscript{\emptyset}{R}^{S_{1}}_{1}\otimes_{S_{1}}N_{1}\otimes_{S_{2}} N_{2}\otimes_{S_{3}}\presupscript{S_{3}}{R}^{\emptyset}_{1}\\
& = \presubscript{S_{1},\emptyset}{\pi}_{S_{3},\emptyset}^{*}(N_{1}\otimes_{S_{2}}N_{2}).
\end{align*}
We get the proposition.
\end{proof}

\section{Parity sheaves}
We continue to assume that $\mathbb{K}$ is a complete local Noetherian integral domain.

\subsection{General notation}
For an algebraic variety $X$ with an action of an algebraic group $B$, let $D_B^{\mathrm{b}}(X)$ be the bounded $B$-equivariant derived category of constructible $\mathbb{K}$-coefficient sheaves.
Let $f\colon X\to Y$ be a morphism between algebraic varieties with $B$-actions and assume that $f$ commutes with the $B$-actions.
Then we have functors $f^!,f^*\colon D_B^{\mathrm{b}}(Y)\to D_B^{\mathrm{b}}(X)$ and $f_!,f_*\colon D_B^{\mathrm{b}}(X)\to D_B^{\mathrm{b}}(Y)$.
We also have the Verdier dual functor $D = D_X\colon D_B^{\mathrm{b}}(X)\to D_B^{\mathrm{b}}(X)$.
For $\mathcal{F}\in D_B^{\mathrm{b}}(X)$, the $n$-th $B$-equivariant cohomology of $\mathcal{F}$ is denoted by $H_B^n(X,\mathcal{F})$ and we put $H^{\bullet}_B(X,\mathcal{F}) = \bigoplus_n H_B^n(X,\mathcal{F})$.
We also put $\Hom^{\bullet}_{D_B^{\mathrm{b}}(X)}(\mathcal{F},\mathcal{G}) = \bigoplus_n\Hom_{D_B^{\mathrm{b}}(X)}(\mathcal{F},\mathcal{G}[n])$ for $\mathcal{F},\mathcal{G}\in D_B^{\mathrm{b}}(X)$.
The constant sheaf on $X$ is denoted by $\mathbb{K}$.
The analogous notation applies to ind-varieties.

\subsection{Theorem}
Let $G$ be a Kac-Moody group over $\C$ attached to a generalized Cartan matrix.
We also have the Borel subgroup $B\subset G$, the unipotent radical $U\subset B$ and the Cartan subgroup $T\subset B$ such that $B = TU$.
Let $\Phi$ be the set of roots, $\Pi$ the set of simple roots and $W$ the Weyl group.
For each $\alpha\in\Phi$, we have the reflection $s_{\alpha}\in W$.
The subset $S = \{s_{\alpha}\mid \alpha\in\Pi\}$ gives a structure of a Coxeter system to $W$.
Let $X^*(T)$ be the character group of $T$ and set $V = X^*(T)\otimes_{\Z}\mathbb{K}$.
For each $s = s_{\alpha}\in S$ with $\alpha\in\Pi$, we put $\alpha_s = \alpha$ and $\alpha_s^\vee = \alpha^\vee$.
Then with $(V,\{(\alpha_s,\alpha_s^\vee)\}_{s\in S})$, we have the category of Soergel bimodules $\Sbimod{}{}$.

Let $S_{1}\subset S$ be a finitary subset.
We have a parabolic subgroup $P_{S_{1}} \subset G$.
Let $\presubscript{S_{1}}{X} = P_{S_{1}}\backslash G$ be the generalized flag variety attached to $S_{1}$.
We also put $X = \presubscript{\emptyset}{X}$.
For each $w\in W_{S_{1}}\backslash W$, we have the Schubert cell $\presubscript{S_{1}}{X}_{w}\subset \presubscript{S_{1}}{X}_{\le w}$.
Assume that another finitary subset $S_{2}\subset S$ is given.
Then for $x\in W_{S_{1}}\backslash W/W_{S_{2}}$, we put $\presubscript{S_{1}}{X}_{S_{2},x} = \bigcup_{w\in W_{S_{1}}\backslash x}\presubscript{S_{1}}{X}_{S_{2},w}$ and $\presubscript{S_{1}}{X}_{S_{2},\le x} = \bigcup_{y\le x}\presubscript{S_{1}}{X}_{S_{2},y}$.
Let $\presubscript{S_{1}}{j}_{S_{2},w}\colon \presubscript{S_{1}}{X}_{S_{2},w}\hookrightarrow \presubscript{S_{1}}{X}$ be the inclusion map.
If $S'_{1}$ is a subset of $S_{1}$, we have the projection $\presubscript{S_{1},S'_{1}}{\pi}\colon \presubscript{S'_{1}}{X}\to \presubscript{S_{1}}{X}$.

Let $P_{S_{2}}\subset G$ be the parabolic subgroup corresponding to $S_{2}$.
Let $\Parity_{P_{S_{2}}}(\presubscript{S_{1}}{X})\subset D^{\mathrm{b}}_{P_{S_{2}}}(\presubscript{S_{1}}{X})$ be the category of $P_{S_{2}}$-equivariant parity complexes on $\presubscript{S_{1}}{X}$~\cite{MR3230821}.
For each $w\in W_{S_{1}}\backslash W/W_{S_{2}}$, there exists an indecomposable parity sheaf $\presubscript{S_{1}}{\mathcal{E}_{S_{2}}(w)}\in\Parity_{P_{S_{2}}}(\presubscript{S_{1}}{X})$ such that $\supp(\presubscript{S_{1}}{\mathcal{E}_{S_{2}}(w)})\subset \presubscript{S_{1}}{X}_{S_{2},\le w}$ and $\presubscript{S_{1}}{\mathcal{E}_{S_{2}}(w)}|_{\presubscript{S_{1}}{X}_{S_{2},w}}\simeq \mathbb{K}[\ell(w_{+}) - \ell(w_{S_{1}})]$~\cite[Theorem~4.6]{MR3230821}.
The functors $\presubscript{S_{1},S'_{1}}{\pi}_{*}$ and $\presubscript{S_{1},S'_{1}}{\pi}^*$ preserve the parity complexes for any subset $S'_{1}\subset S_{1}$~\cite[Proposition~4.10]{MR3230821}.

Let $S_{1},S_{2},S_{3}\subset S$ be finitary subsets.
Throughout this section, we assume the following.
Let $L_{S_{i}}\subset P_{S_{i}}$ be the Levi part of $P_{S_{i}}$ for $i = 1,2,3$.
\begin{itemize}
\item The torsion primes of $L_{S_{i}}$ are invertible in $\mathbb{K}$. (See \cite[2.6]{MR3230821}.)
\item The map $W_{S_{i}}\to \Aut(X^{*}(T)\otimes_{\Z}\mathbb{K})$ is injective.
\end{itemize}

Let $\mathcal{F}\in D^{\mathrm{b}}_{P_{S_{2}}}(\presubscript{S_{1}}{X})$ and $\mathcal{G}\in D^{\mathrm{b}}_{P_{S_{3}}}(\presubscript{S_{2}}{X})$.
We define the convolution product $\mathcal{F}*\mathcal{G}\in D^{\mathrm{b}}_{P_{S_{3}}}(\presubscript{S_{1}}{X})$ as follows.
Let $p\colon G\to \presubscript{S_{1}}{X}$ be the natural projection and
\[
m\colon \presubscript{S_{1}}{X}\times^{P_{S_{2}}}G\to \presubscript{S_{1}}{X},\quad q\colon \presubscript{S_{1}}{X}\times G\to \presubscript{S_{1}}{X}\times^{P_{S_{2}}}G
\]
be the action map of $G$ on $\presubscript{S_{1}}{X}$ and the natural projection, respectively.
Then there exists unique $\mathcal{F}\overset{P_{S_{2}}}{\boxtimes}p^*\mathcal{G}\in D_{P_{S_{3}}}^{\mathrm{b}}(\presubscript{S_{1}}{X}\times^{P_{S_{2}}}G)$ such that $q^*(\mathcal{F}\overset{P_{S_{2}}}{\boxtimes}p^*\mathcal{G})\simeq \mathcal{F}\boxtimes p^*\mathcal{G}$.
Now we put $\mathcal{F}*\mathcal{G} = m_*(\mathcal{F}\overset{P_{S_{2}}}{\boxtimes}p^*\mathcal{G})$.
If $\mathcal{F}\in \Parity_{P_{S_{2}}}(\presubscript{S_{1}}{X})$ and $\mathcal{G}\in \Parity_{P_{S_{3}}}(\presubscript{S_{2}}{X})$ then $\mathcal{F}*\mathcal{G}\in \Parity_{P_{S_{3}}}(\presubscript{S_{1}}{X})$~\cite[Theorem~4.8]{MR3230821}.
Let $S'_{2}\subset S_{2}$ be a subset.
Then we have $P_{S'_{2}}\hookrightarrow P_{S_{2}}$ and hence we have the forgetful functor $\For_{S'_{2}}^{S_{2}}\colon D^{\mathrm{b}}_{P_{S_{2}}}(\presubscript{S_{1}}{X})\to D^{\mathrm{b}}_{P_{S'_{2}}}(\presubscript{S_{1}}{X})$.
This functor has the right adjoint functor $\Average_{S'_{2},*}^{S_{2}}\colon D^{\mathrm{b}}_{P_{S'_{2}}}(\presubscript{S_{1}}{X})\to D^{\mathrm{b}}_{P_{S_{2}}}(\presubscript{S_{1}}{X})$.

\begin{lem}
The functors $\For_{S'_{2}}^{S_{2}}$ and $\Average_{S'_{2},*}^{S_{2}}$ preserve parity sheaves.
\end{lem}
\begin{proof}
Since $\For_{S'_{2}}^{S_{2}}$ commutes with $j_{w}^{!}$, $j_{w}^{*}$ and it sends a constant sheaf to a constant sheaf, it preserves parity sheaves.
Define $i\colon \presubscript{S_{1}}{X}\to \presubscript{S_{1}}{X}\times^{P_{S'_{2}}}P_{S_{2}}$ and $\sigma\colon \presubscript{S_{1}}{X}\times^{P_{S'_{2}}}P_{S_{2}}\to \presubscript{S_{1}}{X}$ by $i(x) = (x,1)$ and $\sigma(x,g) = xg$.
Then $i^{*}\circ\For_{P_{S'_{2}}}^{P_{S_{2}}}\colon D_{P_{S_{2}}}^{\mathrm{b}}(\presubscript{S_{1}}{X}\times^{P_{S'_{2}}}P_{S_{2}})\to D^{\mathrm{b}}_{P_{S'_{2}}}(\presubscript{S_{1}}{X})$ is an equivalence of categories~\cite[Theorem~6.5.10]{MR4337423} and by the proof of \cite[Theorem~6.6.1]{MR4337423} the functor $\Average_{S'_{2},*}^{S_{2}}$ is given by $\sigma_{*}\circ(i^{*}\circ\For_{P_{S'_{2}}}^{P_{S_{2}}})^{-1}$.
Since $i^{*}\circ\For_{P_{S'_{2}}}^{P_{S_{2}}}$ commutes with $j_{w}^{*}$ and $j_{w}^{!}$ in the obvious sense and it sends a constant sheaf to a constant sheaf, it preserves parity sheaves.
The fiber of $\sigma$ is isomorphic to $P_{2_{2}}/P_{S'_{2}}$ which is isomorphic to a generalized flag variety of the Levi part of $P_{S_{2}}$.
Hence $\sigma$ is an even resolution~\cite[Definition~2.33]{MR3230821}, $\sigma_{*}$ preserves parity sheaves~\cite[Proposition~2.34]{MR3230821}.
\end{proof}

\begin{thm}\label{thm:equivalence}
There exists an equivalence of categories $\presubscript{S_{1}}{\mathbb{H}}_{S_{2}}\colon \Parity_{P_{S_{2}}}(\presubscript{S_{1}}{X})\to \Sbimod{S_{1}}{S_{2}}$.
The functor satisfies the following.
\begin{enumerate}
\item For $\mathcal{F}\in \Parity_{P_{S_{2}}}(\presubscript{S_{1}}{X})$ and $\mathcal{G}\in \Parity_{P_{S_{3}}}(\presubscript{S_{2}}{X})$, we have $\presubscript{S_{1}}{\mathbb{H}}_{S_{3}}(\mathcal{F}*\mathcal{G})\simeq \presubscript{S_{1}}{\mathbb{H}}_{S_{2}}(\mathcal{F})\otimes_{S_{2}}\presubscript{S_{2}}{\mathbb{H}}_{S_{3}}(\mathcal{G})$.
\item For $S'_{1}\subset S_{1}$, we have $\presubscript{S_{1}}{\mathbb{H}}_{S_{2}}\circ\presubscript{S_{1},S'_{1}}{\pi}_{*}\simeq \presubscript{S_{1},S'_{1}}{\pi}_{S_{2},S_{2},*}\circ\presubscript{S'_{1}}{\mathbb{H}}_{S_{2}}$ and $\presubscript{S'_{1}}{\mathbb{H}}_{S_{2}}\circ \presubscript{S_{1},S'_{1}}{\pi}^{*}\simeq \presubscript{S_{1},S'_{1}}{\pi}_{S_{2},S_{2}}^{*}\circ\presubscript{S_{1}}{\mathbb{H}}_{S_{2}}$.
\item For $S'_{2}\subset S_{2}$, we have $\presubscript{S_{1}}{\mathbb{H}}_{S_{2}}\circ \Average_{S'_{2},*}^{S_{2}}\simeq \presubscript{S_{1},S_{1}}{\pi}_{S_{2},S'_{2},*}\circ\presubscript{S_{1}}{\mathbb{H}}_{S'_{2}}$ and $\presubscript{S_{1}}{\mathbb{H}}_{S'_{2}}\circ \For_{S'_{2}}^{S_{2}}\simeq \presubscript{S_{1},S_{1}}{\pi}_{S_{2},S'_{2},*}\circ\presubscript{S_{1}}{\mathbb{H}}_{S_{2}}$.
\item We have $\presubscript{S_{1}}{D}_{S_{2}}\circ\presubscript{S_{1}}{\mathbb{H}}_{S_{2}}\simeq \presubscript{S_{1}}{\mathbb{H}}_{S_{2}}\circ D$.
\item We have $\presubscript{S_{1}}{\mathbb{H}}_{S_{2}}(\presubscript{S_{1}}{\mathcal{E}}_{S_{2}}(w))\simeq \presubscript{S_{1}}{B}_{S_{2}}(w)$ for each $w\in W_{S_{1}}\backslash W/W_{S_{2}}$.
\end{enumerate}
\end{thm}
It is proved in \cite{MR4620135} when $S_{2} = \emptyset$.
The functor $\presubscript{S_{1}}{\mathbb{H}}_{S_{2}}$ is given by taking the global sections.
We will give the definition in the next subsection.

\subsection{Definition $\presubscript{S_{1}}{\mathbb{H}}_{S_{2}}$}
For $\mathcal{F}\in \Parity_{B}(X)$, set $\mathbb{H}^{k}(\mathcal{F}) = H^{k}_{B}(X,\mathcal{F})$ and $\mathbb{H}(\mathcal{F}) = \bigoplus_{k}\mathbb{H}^{k}(\mathcal{F})$.
For each $x\in W$ we take a representative $\dot{x}\in G$ and set $\mathbb{H}(\mathcal{F})_{Q}^{x} = \bigoplus_{k}H^{k}(\{\dot{x}B\},\mathcal{F}_{\dot{x}B})\otimes_{R}Q$.
Then as in \cite{MR4620135} this has a structure of an object in $\mathcal{S}$ and $\mathbb{H}$ gives an equivalence of categories $\mathbb{H}\colon\Parity_{B}(X)\to \mathcal{S}$.

We construct $\presubscript{S_{1}}{\mathbb{H}}_{S_{2}}\colon \Parity_{P_{S_{2}}}(\presubscript{S_{1}}{X})\to \Sbimod{S_{1}}{S_{2}}'$.
Let $\mathcal{F}\in \Parity_{S_{2}}(\presubscript{S_{1}}{X})$ and consider $\widetilde{\mathcal{F}} = \For_{\emptyset}^{S_{2}}\presubscript{S_{1},\emptyset}{\pi}^{*}(\mathcal{F})$.
We define an action of $W_{S_{1}}\times W_{S_{2}}$ on $\mathbb{H}(\widetilde{\mathcal{F}})$ as follows.
Let $w_{1}\in W_{S_{1}}$.
Then the multiplication of a lift of $w_{1}$ from the left to $G$ induces an action of $w_{1}$ on $X$ and $\presubscript{S_{1},\emptyset}{\pi}\circ w_{1} = \presubscript{S_{1},\emptyset}{\pi}$.
Let $\{*\}$ be the space of one point and $a\colon \presubscript{S_{1}}{X}\to \{*\}$ be the unique map.
Then we have a map
\begin{multline*}
\mathbb{H}^{k}(\widetilde{\mathcal{F}}) = H^{k}(a_{*}(\presubscript{S_{1},\emptyset}{\pi}_{*}\presubscript{S_{1},\emptyset}{\pi}^{*}\For_{\emptyset}^{S_{2}}(\mathcal{F})))\to H^{k}(a_{*}(\presubscript{S_{1},\emptyset}{\pi}_{*}(w_{1})_{*}w_{1}^{*}\presubscript{S_{1},\emptyset}{\pi}^{*}\For_{\emptyset}^{S_{2}}(\mathcal{F})))\\
\simeq H^{k}(a_{*}(\presubscript{S_{1},\emptyset}{\pi}_{*}\presubscript{S_{1},\emptyset}{\pi}^{*}\For_{\emptyset}^{S_{2}}(\mathcal{F}))) = \mathbb{H}^{k}(\widetilde{F}).
\end{multline*}

Let $p'\colon P'\to \{*\}$ be a $P_{S_{2}}$-resolution, $P = P'\times X$, $p\colon P\to X$ the second projection and $a(p')\colon P/B\to P'/B$ be the map induced by the first projection.
Then $p$ is a $P_{S_{2}}$-resolution and hence we have $\mathcal{F}\in D^{\mathrm{b}}(P/P_{S_{2}})$.
Let $\pi(p')\colon P/B\to P/P_{S_{2}}$ be the natural map.
Then by definition, $\mathbb{H}(\widetilde{\mathcal{F}})$ is the cohomology of an object $p\mapsto a(p')_{*}(\pi(p')^{*}\presubscript{S_{1},\emptyset}{\pi}^{*}\mathcal{F})$ in $D_{B}^{\mathrm{b}}(\{*\})$.
Now $w_{2}\in W_{S_{2}}$ acts on $P/B$ from the right and the same construction for $w_{1}\in W_{S_{1}}$ can apply.
Consequently we get an action of $w_{2}\in W_{S_{2}}$ on $\mathbb{H}(\widetilde{\mathcal{F}})$.
From the construction, this action satisfies \eqref{eq:compatibility of Weyl action}.

Note that we have a natural map $\For_{\emptyset}^{S_{2}}\mathcal{F}\to \presubscript{S_{1},\emptyset}{\pi}_{*}\presubscript{S_{1},\emptyset}{\pi}^{*}\For_{\emptyset}^{S_{2}}(\mathcal{F})$.
Consider the induced map $H^{\bullet}_{B}(\presubscript{S_{1}}{X},\For_{\emptyset}^{S_{2}}\mathcal{F})\to \mathbb{H}^{k}(\widetilde{\mathcal{F}})$.
The diagram
\[
\begin{tikzcd}
\For_{\emptyset}^{S_{2}}\mathcal{F}\arrow[d]\arrow[rd] & \\
\presubscript{S_{1},\emptyset}{\pi}_{*}\presubscript{S_{1},\emptyset}{\pi}^{*}\For_{\emptyset}^{S_{2}}(\mathcal{F})\arrow[r] & \presubscript{S_{1},\emptyset}{\pi}_{*}(w_{1})_{*}w_{1}^{*}\presubscript{S_{1},\emptyset}{\pi}^{*}\For_{\emptyset}^{S_{2}}(\mathcal{F})
\end{tikzcd}
\]
is commutative for all $w_{1}\in W_{S_{1}}$.
Hence the image of $H^{\bullet}_{B}(\presubscript{S_{1}}{X},\For_{\emptyset}^{S_{2}}\mathcal{F})\to \mathbb{H}^{k}(\widetilde{\mathcal{F}})$ is $W_{S_{1}}$-invariant.
Similar argument gives a map $H^{\bullet}_{P_{S_{2}}}(\presubscript{S_{1}}{X},\mathcal{F})\to H^{\bullet}_{B}(\presubscript{S_{1}}{X},\For_{\emptyset}^{S_{2}}\mathcal{F})$ whose image is $W_{S_{2}}$-invariant.
Therefore we have a map $H^{\bullet}_{P_{S_{2}}}(\presubscript{S_{1}}{X},\mathcal{F})\to \mathbb{H}^{k}(\widetilde{\mathcal{F}})$ and it is not difficult to see that this is an $H^{\bullet}_{P_{S_{1}}\times P_{S_{2}}}(\{*\})\simeq R^{S_{1}}\otimes_{\mathbb{K}}R^{S_{2}}$-module homomorphism.
Hence it induces an $R$-bimodule homomorphism
\begin{equation}\label{eq:comparison of pull of sheaves and cohomology}
R\otimes_{R^{S_{1}}}H^{\bullet}_{P_{S_{2}}}(\presubscript{S_{1}}{X},\mathcal{F})\otimes_{R^{S_{2}}}R\to \mathbb{H}(\widetilde{\mathcal{F}}).
\end{equation}
This is $W_{S_{1}}\times W_{S_{2}}$-equivariant, here $(w_{1},w_{2})\in W_{S_{1}}\times W_{S_{2}}$ acts on the left hand side by $f\otimes m\otimes g\mapsto w_{1}(f)\otimes m\otimes w_{2}^{-1}(g)$.
(It follows from the fact that the action on $\mathbb{H}(\widetilde{\mathcal{F}})$ satisfies \eqref{eq:compatibility of Weyl action} and all elements in the image of $H^{\bullet}_{P_{S_{2}}}(\presubscript{S_{1}}{X},\mathcal{F})$ are fixed by $W_{S_{1}}\times W_{S_{2}}$.)

We prove that \eqref{eq:comparison of pull of sheaves and cohomology} is bijective.
\begin{lem}\label{lem:parity sheaves comes from B-eq}
Let $\mathcal{F}\in \Parity_{P_{S_{2}}}(\presubscript{S_{1}}{X})$.
Then there exists $\mathcal{F}_{0}\in \Parity_{B}(\presubscript{S_{1}}{X})$ such that $\mathcal{F}$ is a direct summand of $\Average_{\emptyset,*}^{S_{2}}(\mathcal{F}_{0})$.
\end{lem}
\begin{proof}
We may assume $\mathcal{F} = \presubscript{S_{1}}{\mathcal{E}}_{S_{2}}(w)$ for some $w\in W_{S_{1}}\backslash W/W_{S_{2}}$ and set $\mathcal{F}_{0} = \presubscript{S_{1}}{\mathcal{E}}_{\emptyset}(W_{S_{1}}w_{-})$.
Then, by the support argument, $\Average_{\emptyset,*}^{S_{2}}(\mathcal{F}_{0})$ contains $\mathcal{F}$ as a direct summand up to shift.
By taking a suitable shift of $\mathcal{F}_{0}$, we get the lemma.
\end{proof}

\begin{lem}\label{lem:cohomology of average}
Let $S'_{2}\subset S_{2}$.
For $\mathcal{F}\in \Parity_{P_{S'_{2}}}(\presubscript{S_{1}}{X})$, the $R^{S_{2}}$-module $H^{\bullet}_{P_{S_{2}}}(\presubscript{S_{1}}{X},\Average_{S'_{2},*}^{S_{2}}\mathcal{F})$ is a restriction of the $R^{S'_{2}}$-module $H^{\bullet}_{P_{S'_{2}}}(\presubscript{S_{1}}{X},\mathcal{F})$.
\end{lem}
\begin{proof}
Since $(\For_{S_{2}'}^{S_{2}},\Average_{S_{2}',*}^{S_{2}})$ is an adjoint pair and $\For_{S_{2}'}^{S_{2}}\mathbb{K} = \mathbb{K}$, we have
\[
H^{\bullet}_{P_{S_{2}}}(\presubscript{S_{1}}{X},\Average_{S'_{2},*}^{S_{2}}\mathcal{F})
= \Hom^{\bullet}_{D_{P_{S_{2}}}^{\mathrm{b}}(\presubscript{S_{1}}{X})}(\mathbb{K},\Average_{S'_{2},*}^{S_{2}}\mathcal{F})
\simeq \Hom^{\bullet}_{D_{P_{S'_{2}}}^{\mathrm{b}}(\presubscript{S_{1}}{X})}(\mathbb{K},\mathcal{F})
= H^{\bullet}_{P_{S'_{2}}}(\presubscript{S_{1}}{X},\mathcal{F}).
\]
We get the lemma.
\end{proof}

\begin{lem}\label{lem:projectivity of cohomologies}
For $\mathcal{F}\in \Parity_{P_{S_{2}}}(\presubscript{S_{1}}{X})$, $H^{\bullet}_{P_{S_{2}}}(\presubscript{S_{1}}{X},\mathcal{F})$ is graded free as a left $R^{S_{1}}$-module and a right $R^{S_{2}}$-module.
\end{lem}
\begin{proof}
We prove that the module is graded free as a right $R^{S_{2}}$-module.
We may assume that $\mathcal{F} = \Average_{\emptyset,*}^{S_{2}}(\mathcal{F}_{0})$ for some $\mathcal{F}_{0}\in \Parity_{B}(\presubscript{S_{1}}{X})$ by Lemma~\ref{lem:parity sheaves comes from B-eq}.
By the previous lemma and the fact that $R$ is graded free as an $R^{S_{2}}$-module, we may assume that $S_{2} = \emptyset$.
In this case, by \cite[Corollary~3.5]{MR4620135}, $H^{\bullet}_{P_{S_{2}}})\presubscript{S_{1}}{X},\mathcal{F})$ is an object of $\Sbimod{S_{1}}{\emptyset}$.
An object in $\Sbimod{}{}$ is a direct summand of a direct sum of Bott-Samelson bimodules.
Since Bott-Samelson bimodules are graded free $R$-modules, any object in $\Sbimod{}{}$ is also graded free.
Any object in $\Sbimod{S_{1}}{\emptyset}$ is a direct summand of the restriction of an object in $\Sbimod{}{}$.
Hence it is also graded free.
Therefore the module $H^{\bullet}_{P_{S_{2}}})\presubscript{S_{1}}{X},\mathcal{F})$ is graded free as a right $R^{S_{2}}$-module.
The same argument applies to prove that the module is a graded free left $R^{S_{1}}$-module using $\presubscript{S_{1},\emptyset}{\pi}_{*}$ instead of $\Average_{S_{2},*}^{\emptyset}$.
\end{proof}

\begin{lem}\label{lem:comparison of pull of sheaves and cohomology}
The map \eqref{eq:comparison of pull of sheaves and cohomology} is an isomorphism.
\end{lem}
\begin{proof}
By \cite[Theorem~3.1 (2)]{MR4620135}, we have $\mathbb{H}(\widetilde{\mathcal{F}})\simeq R\otimes_{R^{S_{1}}}H^{\bullet}_{B}(\presubscript{S_{1}}{X},\For_{\emptyset}^{S_{2}}\mathcal{F})$.
By the previous lemma and \cite[Lemma~6.7.4]{MR4337423}, we have $H^{\bullet}_{B}(\presubscript{S_{1}}{X},\For_{\emptyset}^{S_{2}}\mathcal{F})\simeq H^{\bullet}_{P_{S_{2}}}(\presubscript{S_{1}}{X},\mathcal{F})\otimes_{R^{S_{2}}}R$.
\end{proof}

By this lemma and Lemma~\ref{lem:projectivity of cohomologies}, we have 
\begin{align*}
\mathbb{H}(\widetilde{\mathcal{F}})^{W_{S_{1}}\times W_{S_{2}}} & \simeq (R\otimes_{R^{W_{S_{1}}}}H^{\bullet}_{P_{S_{2}}}(\presubscript{S_{1}}{X},\mathcal{F})\otimes_{R^{W_{S_{2}}}}R)^{W_{S_{1}}\times W_{S_{2}}}\\
& = R^{S_{1}}\otimes_{R^{W_{S_{1}}}}H^{\bullet}_{P_{S_{2}}}(\presubscript{S_{1}}{X},\mathcal{F})\otimes_{R^{W_{S_{2}}}}R^{S_{2}} \simeq H^{\bullet}_{P_{S_{2}}}(\presubscript{S_{1}}{X},\mathcal{F}).
\end{align*}
Hence $\mathbb{H}(\mathcal{F})$ satisfies \eqref{eq:fixed part recovers all}.
Therefore $\mathcal{F}\mapsto \mathbb{H}(\widetilde{\mathcal{F}})$ defines a functor
\[
\presubscript{S_{1}}{\mathbb{H}}_{S_{2}}'\colon \Parity_{P_{S_{2}}}(\presubscript{S_{1}}{X})\to \Sbimod{S_{1}}{S_{2}}'.
\]
We prove that this is an equivalence of categories.
\begin{rem}\label{rem:the equivalence is given by cohomology in singular case}
Let $\presubscript{S_{1}}{\mathbb{H}}_{S_{2}}\colon \Parity_{S_{2}}(\presubscript{S_{1}}{X})\to \Sbimod{S_{1}}{S_{2}}'$ be the composition of $\presubscript{S_{1}}{\mathbb{H}}_{S_{2}}'$ and the equivalence in Theorem~\ref{thm:alternative description}.
Then $\presubscript{S_{1}}{\mathbb{H}}_{S_{2}}(\mathcal{F}) = (\presubscript{S_{1}}{\mathbb{H}}_{S_{2}}'(\mathcal{F}))^{W_{S_{1}}\times W_{S_{2}}}$ for any $\mathcal{F}\in \Parity_{P_{S_{2}}}(\presubscript{S_{1}}{X})$ and it is isomorphic to $H^{\bullet}_{P_{S_{2}}}(\presubscript{S_{1}}{X})$ by Lemma~\ref{lem:comparison of pull of sheaves and cohomology} as an $(R^{S_{1}},R^{S_{2}})$-bimodule.
\end{rem}

\subsection{Equivalence}
Let $S'_{1}\subset S_{1}$ and $S'_{2}\subset S_{2}$.
We have functors 
\begin{align*}
\presubscript{S_{1},S'_{1}}{\pi}_{S_{2},S'_{2}}^{\prime *}\colon & \Sbimod{S_{1}}{S_{2}}'\simeq \Sbimod{S_{1}}{S_{2}}\xrightarrow{\presubscript{S_{1},S'_{1}}{\pi}_{S_{2},S'_{2}}^{*}}\Sbimod{S'_{1}}{S'_{2}}\simeq \Sbimod{S'_{1}}{S'_{2}}',\\
\presubscript{S_{1},S'_{1}}{\pi}_{S_{2},S'_{2},*}'\colon & \Sbimod{S'_{1}}{S'_{2}}'\simeq \Sbimod{S'_{1}}{S'_{2}}\xrightarrow{\presubscript{S_{1},S'_{1}}{\pi}_{S_{2},S'_{2},*}}\Sbimod{S_{1}}{S_{2}}\simeq \Sbimod{S_{1}}{S_{2}}'.
\end{align*}
Explicitly they are given as follows.
For $M\in \Sbimod{S_{1}}{S_{2}}'$, we have $\presubscript{S_{1},S'_{1}}{\pi}_{S_{2},S'_{2}}^{\prime *}(M) = M$.
The action of the group $W_{S'_{1}}\times W_{S'_{2}}$ is given by the restriction of that of $W_{S_{1}}\times W_{S_{2}}$.
For $N\in \Sbimod{S'_{1}}{S'_{2}}'$, we have $\presubscript{S_{1},S'_{1}}{\pi}_{S_{2},S'_{2},*}'(N) = R\otimes_{R^{S_{1}}}N^{W_{S'_{1}}\times W_{S'_{2}}}\otimes_{R^{S_{2}}}R$.
\begin{prop}
We have the following.
\begin{enumerate}
\item $\presubscript{S_{1}}{\mathbb{H}}_{S'_{2}}'\circ \For_{S'_{2}}^{S_{2}}\simeq \presubscript{S_{1},S_{1}}{\pi}_{S_{2},S'_{2}}^{\prime *}\circ\presubscript{S_{1}}{\mathbb{H}}_{S_{2}}'$.
\item $\presubscript{S_{1}}{\mathbb{H}}_{S_{2}}'\circ \Average_{S'_{2},*}^{S_{2}}\simeq \presubscript{S_{1},S_{1}}{\pi}_{S_{2},S'_{2},*}'\circ\presubscript{S_{1}}{\mathbb{H}}_{S'_{2}}'$.
\item $\presubscript{S'_{1}}{\mathbb{H}}_{S_{2}}'\circ \presubscript{S_{1},S'_{1}}{\pi}^{*}\simeq \presubscript{S_{1},S'_{1}}{\pi}_{S_{2},S_{2}}^{\prime *}\circ\presubscript{S_{1}}{\mathbb{H}}_{S_{2}}'$.
\item $\presubscript{S_{1}}{\mathbb{H}}_{S_{2}}'\circ \presubscript{S_{1},S'_{1}}{\pi}_{*}\simeq \presubscript{S_{1},S'_{1}}{\pi}_{S_{2},S_{2},*}'\circ\presubscript{S'_{1}}{\mathbb{H}}_{S_{2}}'$.
\end{enumerate}
\end{prop}
\begin{proof}
For (1), as an object in $\Sbimod{}{}$, we have 
\[
\presubscript{S_{1}}{\mathbb{H}}_{S'_{2}}'\circ \For_{S'_{2}}^{S_{2}}(\mathcal{F})
=
\mathbb{H}(\pi_{S_{1},\emptyset}^{*}\For_{\emptyset}^{S'_{2}}\For_{S'_{2}}^{S_{2}}(\mathcal{F}))
\simeq \mathbb{H}(\pi_{S_{1},\emptyset}^{*}\For_{\emptyset}^{S_{2}}(\mathcal{F}))
= \presubscript{S_{1}}{\mathbb{H}}_{S_{2}}'(\mathcal{F})
\]
for $\mathcal{F}\in \Parity_{P_{S_{2}}}(\presubscript{S_{1}}{X})$.
By the definition, the action of $W_{S_{1}}\times W_{S'_{2}}$ on both sides coincides with each other.
Hence we get (1) and (3) follows from similar arguments.

We prove (2).
Let $\mathcal{G}\in \Parity{P_{S'_{2}}}(\presubscript{S_{1}}{X})$.
We have a map
\[
\presubscript{S_{1},S_{1}}{\pi}_{S_{2},S'_{2}}^{\prime *}(\presubscript{S_{1}}{\mathbb{H}}_{S_{2}}'(\Average_{S'_{2},*}^{S_{2}}(\mathcal{G}))) \simeq \presubscript{S_{1}}{\mathbb{H}}_{S'_{2}}'(\For_{S'_{2}}^{S_{2}}\Average_{S'_{2},*}^{S_{2}}(\mathcal{G}))\to \presubscript{S_{1}}{\mathbb{H}}_{S'_{2}}'(\mathcal{G}).
\]
and by the adjointness we have
\begin{equation}\label{eq:map for Av}
\presubscript{S_{1}}{\mathbb{H}}_{S_{2}}'(\Average_{S'_{2},*}^{S_{2}}(\mathcal{G}))\to \presubscript{S_{1},S_{1}}{\pi}'_{S_{2},S'_{2},*}(\presubscript{S_{1}}{\mathbb{H}}_{S'_{2}}'(\mathcal{G})).
\end{equation}
We have 
\[
\presubscript{S_{1},S_{1}}{\pi}'_{S_{2},S'_{2},*}(\presubscript{S_{1}}{\mathbb{H}}_{S'_{2}}'(\mathcal{G}))
\simeq R\otimes_{R^{S_{1}}}H_{P_{S'_{2}}}^{\bullet}(\presubscript{S_{1}}{X},\mathcal{G})\otimes_{R^{S_{2}}}R
\]
by the definition and Remark~\ref{rem:the equivalence is given by cohomology in singular case}.
We also have
\[
\presubscript{S_{1}}{\mathbb{H}}_{S_{2}}'(\Average_{S'_{2},*}^{S_{2}}(\mathcal{G}))
\simeq R\otimes_{R^{S_{1}}}H_{P_{S_{2}}}^{\bullet}(\presubscript{S_{1}}{X},\Average_{S'_{2},*}^{S_{2}}\mathcal{G})\otimes_{R^{S_{2}}}R
\]
by \eqref{eq:comparison of pull of sheaves and cohomology}.
Hence \eqref{eq:map for Av} is an isomorphism by Lemma~\ref{lem:cohomology of average} and (2) follows.
We can use a similar argument for the proof of (4).
\end{proof}

\begin{prop}
The functor $\presubscript{S_{1}}{\mathbb{H}}_{S_{2}}'$ is an equivalence of categories.
\end{prop}
\begin{proof}
First we prove that $\Hom^{\bullet}(\mathcal{F},\mathcal{G})\to \Hom^{\bullet}(\presubscript{S_{1}}{\mathbb{H}}_{S_{2}}'(\mathcal{F}),\presubscript{S_{1}}{\mathbb{H}}_{S_{2}}'(\mathcal{G}))$ is an isomorphism for $\mathcal{F},\mathcal{G}\in \Parity_{P_{S_{2}}}(\presubscript{S_{1}}{X})$.
By Lemma~\ref{lem:parity sheaves comes from B-eq} we may assume that $\mathcal{G} = \Average_{\emptyset,*}^{S_{2}}\mathcal{G}'$ for some $\mathcal{G}'\in \Parity_{B}(\presubscript{S_{1}}{X})$.
Similarly we may assume $\mathcal{G}' = \presubscript{S_{1},\emptyset}{\pi}_{*}\mathcal{G}_{0}$ for some $\mathcal{G}_{0}\in \Parity_{B}(X)$.
Since the proposition is true for $S_{1} = S_{2} = \emptyset$, we have
\begin{align*}
\Hom^{\bullet}(\mathcal{F},\Average_{\emptyset,*}^{S_{2}}\presubscript{S_{1},\emptyset}{\pi}_{*}\mathcal{G}_{0}) & \simeq
\Hom^{\bullet}(\presubscript{S_{1},\emptyset}{\pi}^{*}\For_{\emptyset}^{S_{2}}\mathcal{F},\mathcal{G}_{0})\\
& \simeq \Hom^{\bullet}(\mathbb{H}(\presubscript{S_{1},\emptyset}{\pi}^{*}\For_{\emptyset}^{S_{2}}\mathcal{F}),\mathbb{H}(\mathcal{G}_{0}))\\
& \simeq \Hom^{\bullet}(\presubscript{S_{1},\emptyset}{\pi}_{S_{2},\emptyset}^{*}\presubscript{S_{1}}{\mathbb{H}}_{S_{2}}'(\mathcal{F}),\mathbb{H}(\mathcal{G}_{0}))\\
& \simeq \Hom^{\bullet}(\presubscript{S_{1}}{\mathbb{H}}_{S_{2}}'(\mathcal{F}),\presubscript{S_{1},\emptyset}{\pi}_{S_{2},\emptyset,*}\mathbb{H}(\mathcal{G}_{0}))\\
& \simeq \Hom^{\bullet}(\presubscript{S_{1}}{\mathbb{H}}_{S_{2}}'(\mathcal{F}),\presubscript{S_{1}}{\mathbb{H}}_{S_{2}}'(\Average_{\emptyset,*}^{S_{2}}\presubscript{S_{1},\emptyset}{\pi}_{*}\mathcal{G}_{0})).
\end{align*}
Therefore the functor is fully faithful.

We prove that the functor is essentially surjective.
Let $M\in \Sbimod{S_{1}}{S_{2}}'$ and take $M_{0}\in \Sbimod{}{}$ such that $M$ is a direct summand of $\presubscript{S_{1},\emptyset}{\pi}_{S_{2},\emptyset,*}(M_{0})$.
Since $\mathbb{H}$ is an equivalence, there exists $\mathcal{F}_{0}\in\Parity_{B}(X)$ such that $M_{0}\simeq \mathbb{H}(\mathcal{F}_{0})$.
Then $M$ is a direct summand of $\presubscript{S_{1},\emptyset}{\pi}_{S_{2},\emptyset,*}(\mathbb{H}(\mathcal{F}_{0}))\simeq \presubscript{S_{1}}{\mathbb{H}}_{S_{2}}(\Average_{\emptyset,*}^{S_{2}}\presubscript{S_{1},\emptyset}{\pi}_{*}\mathcal{F}_{0})$.
We can take an idempotent $e\in \End(\presubscript{S_{1},\emptyset}{\pi}_{S_{2},\emptyset,*}(M_{0}))$ such that $M = e\presubscript{S_{1},\emptyset}{\pi}_{S_{2},\emptyset,*}(M_{0})$.
Since $\presubscript{S_{1}}{\mathbb{H}}_{S_{2}}$ is fully faithful, there exists the corresponding idempotent $e'\in \End(\Average_{\emptyset,*}^{S_{2}}\presubscript{S_{1},\emptyset}{\pi}_{*}\mathcal{F}_{0})$.
Then we have $M\simeq \presubscript{S_{1}}{\mathbb{H}}_{S_{2}}(e'\Average_{\emptyset,*}^{S_{2}}\presubscript{S_{1},\emptyset}{\pi}_{*}\mathcal{F}_{0})$.
\end{proof}

\subsection{Compatibility with convolutions}
In this subsection, we prove Theorem~\ref{thm:equivalence} (1).
Fix $\mathcal{F}_{0}\in \Parity_{P_{S_{2}}}(\presubscript{S_{1}}{X})$, $\mathcal{G}_{0}\in \Parity_{P_{S_{3}}}(\presubscript{S_{2}}{X})$.
Put $\mathcal{F} = \presubscript{S_{1},\emptyset}{\pi}^{*}(\mathcal{F}_{0})$ and $\mathcal{G} = \For_{P_{S_{3}}}^{B}(\mathcal{G}_{0})$.
It is easy to check that $\mathcal{F}*\mathcal{G} = \presubscript{S_{1},\emptyset}{\pi}^{*}\For_{B}^{P_{S_{3}}}(\mathcal{F}_{0}*\mathcal{G}_{0})$.
Let $N_{S_{2}}\subset P_{S_{2}}$ be the pro-unipotent radical of $P_{S_{2}}$ and $L_{S_{2}}$ the Levi subgroup.
Set $\presubscript{S_{2}}{\widetilde{X}} = N_{S_{2}}\backslash G$.
We use the natural projection $r\colon \presubscript{S_{2}}{\widetilde{X}}\to \presubscript{S_{2}}{X}$.
The group $L_{S_{2}}\times L_{S_{2}}$ acts on $X\times \presubscript{S_{2}}{\widetilde{X}}$ by $(g_{1},g_{2})(x,y) = (xg_{1}^{-1},g_{2}y)$.
The diagonal subgroup $\diag(L_{S_{2}})\subset L_{S_{2}}\times L_{S_{2}}$ acts freely on this space and the quotient is $X\times^{L_{S_{2}}} \presubscript{S_{2}}{\widetilde{X}}$.
The complex $r^{*}\mathcal{G}$ can be regarded as an object in $D_{L_{S_{2}}\times B}^{\mathrm{b}}(\presubscript{S_{2}}{\widetilde{X}})$.
More precisely here we use the quotient equivalence~\cite[Theorem~6.5.9]{MR4337423}.
Therefore $\mathcal{F}\boxtimes r^{*}\mathcal{G}$ is $L_{S_{2}}\times L_{S_{2}}$-equivariant.
Let $q'\colon X\times \presubscript{S_{2}}{\widetilde{X}}\to X\times^{L_{S_{2}}} \presubscript{S_{2}}{\widetilde{X}}$ be the natural projection.
Then there exists $\mathcal{F}\overset{L_{S_{2}}}{\boxtimes}r^{*}\mathcal{G}\in D^{\mathrm{b}}_{B}(X\times^{L_{S_{2}}} \presubscript{S_{2}}{\widetilde{X}})$ such that $(q')^{*}(\mathcal{F}\overset{L_{S_{2}}}{\boxtimes}r^{*}\mathcal{G})\simeq \For_{\diag L_{S_{2}}}^{L_{S_{2}}\times L_{S_{2}}}\mathcal{F}\boxtimes r^{*}\mathcal{G}$.
We have
\[
H^{\bullet}_{\diag(L_{S_{2}})\times B}(X\times \presubscript{S_{2}}{\widetilde{X}},\For_{\diag L_{S_{2}}}^{L_{S_{2}}\times L_{S_{2}}}(\mathcal{F}\boxtimes r^{*}\mathcal{G}))
\simeq
H^{\bullet}_{B}(X\times^{L_{S_{2}}} \presubscript{S_{2}}{\widetilde{X}},\mathcal{F}\overset{L_{S_{2}}}{\boxtimes} r^{*}\mathcal{G}).
\]
Let $p\colon G\to \presubscript{S_{2}}{X}$ be the natural projection.
Since projections $X\times^{L_{S_{2}}}G\to X\times^{L_{S_{2}}}\presubscript{S_{2}}{\widetilde{X}}$ and $X\times^{L_{S_{2}}}G\to X\times^{P_{S_{2}}}G$ are fibrations such that fibers are pro-affine spaces, we have 
\[
H^{\bullet}_{B}(X\times^{L_{S_{2}}} \presubscript{S_{2}}{\widetilde{X}},\mathcal{F}\overset{L_{S_{2}}}{\boxtimes} r^{*}\mathcal{G}).
\simeq
H^{\bullet}_{B}(X\times^{P_{S_{2}}} G,\mathcal{F}\overset{P_{S_{2}}}{\boxtimes} p^{*}\mathcal{G})
\simeq H_{B}^{\bullet}(\mathcal{F}*\mathcal{G}).
\]
The last isomorphism comes from the definition of the convolution.
Therefore we get
\[
H_{B}^{\bullet}(\mathcal{F}*\mathcal{G})
\simeq
H^{\bullet}_{\diag(L_{S_{2}})\times B}(X\times \presubscript{S_{2}}{\widetilde{X}},\For_{\diag L_{S_{2}}}^{L_{S_{2}}\times L_{S_{2}}}(\mathcal{F}\boxtimes p^{*}\mathcal{G})).
\]
In particular this is an $H_{\diag(L_{S_{2}})}(\{*\})\simeq R^{S_{2}}$-module.
The cohomologies $H^{\bullet}_{P_{S_{2}}}(X,\mathcal{F})\simeq H^{\bullet}_{L_{S_{2}}}(X,\mathcal{F})$ and $H^{\bullet}_{L_{S_{2}}\times B}(\presubscript{S_{2}}{\widetilde{X}},r^{*}\mathcal{G})\simeq H^{\bullet}_{B}(\presubscript{S_{2}}{X},\mathcal{G})$ are free as $H^{\bullet}_{L_{S_{2}}}(\{*\})\simeq R^{S_{2}}$-modules by Lemma~\ref{lem:projectivity of cohomologies}.
Hence we have~\cite[Lemma~6.7.4, Proposition~6.7.5]{MR4337423}
\begin{align*}
& H^{\bullet}_{\diag(L_{S_{2}})\times B}(X\times \presubscript{S_{2}}{\widetilde{X}},\For_{\diag L_{S_{2}}}^{L_{S_{2}}\times L_{S_{2}}}(\mathcal{F}\boxtimes r^{*}\mathcal{G}))\\
& \simeq
H^{\bullet}_{L_{S_{2}}\times L_{S_{2}}\times B}(X\times \presubscript{S_{2}}{\widetilde{X}},\mathcal{F}\boxtimes r^{*}\mathcal{G})\otimes_{H^{\bullet}_{L_{S_{2}}\times L_{S_{2}}}(\{*\})}H^{\bullet}_{\diag L_{S_{2}}}(\{*\})\\
& \simeq H^{\bullet}_{L_{S_{2}}}(X,\mathcal{F})\otimes_{R^{S_{2}}}H^{\bullet}_{L_{S_{2}}\times B}(\presubscript{S_{2}}{\widetilde{X}},r^{*}\mathcal{G})\\
& \simeq H^{\bullet}_{L_{S_{2}}}(X,\mathcal{F})\otimes_{R^{S_{2}}}H^{\bullet}_{B}(\presubscript{S_{2}}{X},\mathcal{G}).
\end{align*}
In particular this is a free $H_{\diag(L_{S_{2}})}(\{*\})$-module by Lemma~\ref{lem:projectivity of cohomologies}.
Therefore
\begin{align*}
& H^{\bullet}_{\diag T\times B}(X\times \presubscript{S_{2}}{\widetilde{X}},\For_{\diag T}^{L_{S_{2}}\times L_{S_{2}}}(\mathcal{F}\boxtimes r^{*}\mathcal{G}))\\
& \simeq H^{\bullet}_{\diag(L_{S_{2}})\times B}(X\times \presubscript{S_{2}}{\widetilde{X}},\For_{\diag L_{S_{2}}}^{L_{S_{2}}\times L_{S_{2}}}\mathcal{F}\boxtimes p^{*}\mathcal{G})\otimes_{H_{L_{S_{2}}}(\{*\})}H_{T}(\{*\}).
\end{align*}
We have
\begin{align*}
& H^{\bullet}_{\diag T\times B}(X\times \presubscript{S_{2}}{\widetilde{X}},\For_{\diag T}^{L_{S_{2}}\times L_{S_{2}}}(\mathcal{F}\boxtimes r^{*}\mathcal{G}))\\
& \simeq 
H^{\bullet}_{\diag T\times B}(X\times \presubscript{S_{2}}{\widetilde{X}},\For_{\diag T}^{T\times T}(\For_{T}^{L_{S_{2}}}(\mathcal{F})\boxtimes \For_{T}^{L_{S_{2}}}(r^{*}\mathcal{G}))).
\end{align*}
Set $\widetilde{X} = U\backslash G$ and let $\widetilde{\pi}\colon \presubscript{S_{2}}{\widetilde{X}}\to \widetilde{X}$ be the natural projection.
Let $r_{0}\colon \widetilde{X}\to X$ be the natural projection.
\begin{lem}
The diagram
\[
\begin{tikzcd}
D^{\mathrm{b}}(\presubscript{S_{2}}{X})\arrow[rr,"\presubscript{S_{2},\emptyset}{\pi}^{*}"]\arrow[d,"r^{*}"] && D^{\mathrm{b}}(X)\arrow[d,"r_{0}^{*}"]\\
D^{\mathrm{b}}_{L_{S_{2}}}(\presubscript{S_{2}}{\widetilde{X}})\arrow[r,"\For_{T}^{L_{S_{2}}}"'] & D^{\mathrm{b}}_{T}(\presubscript{S_{2}}{\widetilde{X}}) & D^{\mathrm{b}}_{T}(\widetilde{X})\arrow[l,"\widetilde{\pi}^{*}"].
\end{tikzcd}
\]
is commutative up to natural equivalences.
Here the actions of $L_{S_{2}}$ and $T$ are induced by the multiplications from the left.
\end{lem}
\begin{proof}
In general, for an algebraic group $H$, let $\Infl_{H/H}^{H}$ be the inflation functor.
Recall that the $r^{*}$ is the abbreviation of the compositions $D^{\mathrm{b}}(\presubscript{S_{2}}{X})\xrightarrow{\Infl_{L_{S_{2}}/L_{S_{2}}}^{L_{S_{2}}}}D^{\mathrm{b}}_{L_{S_{2}}}(\presubscript{S_{2}}{X})\xrightarrow{r^{*}}D^{\mathrm{b}}_{L_{S_{2}}}(\presubscript{S_{2}}{\widetilde{X}})$~\cite[Theorem~6.5.9]{MR4337423}.
By \cite[Lemma~6.5.8]{MR4337423}, we have a natural isomorphism $\Infl_{T/T}^{T}\simeq \For_{T}^{L_{S_{2}}}\circ\Infl_{L_{S_{2}}/L_{S_{2}}}^{L_{S_{2}}}$.
The lemma follows from the following commutative (up to natural equivalences) diagram:
\[
\begin{tikzcd}
D^{\mathrm{b}}(\presubscript{S_{2}}{X})\arrow[rr,"\presubscript{S_{2},\emptyset}{\pi}^{*}"]\arrow[d,"\Infl_{L_{S_{2}}/L_{S_{2}}}^{L_{S_{2}}}"']\arrow[rd,"\Infl_{T/T}^{T}"] && D^{\mathrm{b}}(X)\arrow[d,"\Infl_{T/T}^{T}"]\\
D^{\mathrm{b}}_{L_{S_{2}}}(\presubscript{S_{2}}{X})\arrow[r,"\For_{T}^{L_{S_{2}}}"']\arrow[d,"r^{*}"] & D^{\mathrm{b}}_{T}(\presubscript{S_{2}}{X})\arrow[r,"\presubscript{S_{2},\emptyset}{\pi}^{*}"]\arrow[d,"r^{*}"] & D^{\mathrm{b}}_{T}(X)\arrow[d,"r_{0}^{*}"]\\
D^{\mathrm{b}}_{L_{S_{2}}}(\presubscript{S_{2}}{\widetilde{X}})\arrow[r,"\For_{T}^{L_{S_{2}}}"'] & D^{\mathrm{b}}_{T}(\presubscript{S_{2}}{\widetilde{X}}) & D^{\mathrm{b}}_{T}(\widetilde{X})\arrow[l,"\widetilde{\pi}^{*}"].
\end{tikzcd}
\]
Here we use the fact that forgetful functors and inflation functors commute with pull-backs \cite[Proposition~6.5.4, 6.5.7]{MR4337423} and $r = \presubscript{S_{2},\emptyset}{\pi}\circ r_{0}\circ \widetilde{\pi}$.
\end{proof}
Hence
\begin{align*}
& H^{\bullet}_{\diag T\times B}(X\times \presubscript{S_{2}}{\widetilde{X}},\For_{\diag T}^{T\times T}(\For_{T}^{L_{S_{2}}}(\mathcal{F})\boxtimes \For_{T}^{L_{S_{2}}}(r^{*}\mathcal{G})))\\
& \simeq H^{\bullet}_{\diag T\times B}(X\times \presubscript{S_{2}}{\widetilde{X}},\For_{\diag T}^{T\times T}(\For_{T}^{L_{S_{2}}}(\mathcal{F})\boxtimes \widetilde{\pi}^{*}r_{0}^{*}\presubscript{S_{2},\emptyset}{\pi}^{*}\mathcal{G})).
\end{align*}
The projection $\widetilde{\pi}$ is a fibration whose fibers are affine spaces, hence
\begin{align*}
& H^{\bullet}_{\diag T\times B}(X\times \presubscript{S_{2}}{\widetilde{X}},\For_{\diag T}^{T\times T}(\For_{T}^{L_{S_{2}}}(\mathcal{F})\boxtimes \widetilde{\pi}^{*}r_{0}^{*}\presubscript{S_{2},\emptyset}{\pi}^{*}\mathcal{G}))\\
& \simeq H^{\bullet}_{\diag T\times B}(X\times \widetilde{X},\For_{\diag T}^{T\times T}(\For_{T}^{L_{S_{2}}}(\mathcal{F})\boxtimes r_{0}^{*}\presubscript{S_{2},\emptyset}{\pi}^{*}\mathcal{G})).
\end{align*}
As we have seen, this is isomorphic to $H^{\bullet}_{B}(X,\For_{T}^{L_{S_{2}}}(\mathcal{F})*\presubscript{S_{2},\emptyset}{\pi}^{*}\mathcal{G})$.
Hence we get
\[
H^{\bullet}_{B}(X,\mathcal{F}*\mathcal{G})\otimes_{R^{S_{2}}}R \simeq H^{\bullet}_{B}(X,\For_{T}^{L_{S_{2}}}(\mathcal{F})*\presubscript{S_{2},\emptyset}{\pi}^{*}\mathcal{G}).
\]
Using the same argument to define $W_{S_{1}}\times W_{S_{2}}$ action on $\presubscript{S_{1}}{\mathbb{H}}_{S_{2}}(\mathcal{F}_{0}) = H^{\bullet}_{B}(X,\mathcal{F})$, we can define an action of $W_{S_{2}}$ on the right hand side.
In terms of the left hand side, by the construction, the action of $w\in W_{S_{2}}$ is given by $m\otimes f\mapsto m\otimes w(f)$.
Hence we get
\[
H^{\bullet}_{B}(X,\mathcal{F}*\mathcal{G}) \simeq H^{\bullet}_{B}(X,\For_{T}^{L_{S_{2}}}(\mathcal{F})*\presubscript{S_{2},\emptyset}{\pi}^{*}\mathcal{G})^{W_{S_{2}}}.
\]
The functor $\mathbb{H}$ is compatible with the convolution, hence we have
\[
H^{\bullet}_{B}(X,\For_{T}^{L_{S_{2}}}(\mathcal{F})*\presubscript{S_{2},\emptyset}{\pi}^{*}\mathcal{G})
\simeq 
H^{\bullet}_{B}(X,\For_{T}^{L_{S_{2}}}(\mathcal{F}))\otimes_{R}H^{\bullet}_{B}(X,\presubscript{S_{2},\emptyset}{\pi}^{*}\mathcal{G})
=
\presubscript{S_{1}}{\mathbb{H}}_{S_{2}}(\mathcal{F}_{0})\otimes \presubscript{S_{2}}{\mathbb{H}}_{S_{3}}(\mathcal{G}_{0}).
\]
From the construction, the $W_{S_{2}}$-action on the left hand side is the same as the right $W_{S_{2}}$-action used in the definition of $\otimes_{S_{2}}'$.
On the other hand, we have $\mathcal{F}*\mathcal{G} = \For_{T}^{L_{S_{3}}}\presubscript{S_{1},\emptyset}{\pi}^{*}(\mathcal{F}_{0}*\mathcal{G}_{0})$ by the definitions.
Therefore
\[
(\presubscript{S_{1}}{\mathbb{H}}_{S_{2}}(\mathcal{F}_{0})\otimes \presubscript{S_{2}}{\mathbb{H}}_{S_{3}}(\mathcal{G}_{0}))^{W_{S_{2}}}
\simeq 
H^{\bullet}_{B}(X,\For_{T}^{L_{S_{2}}}(\mathcal{F})*\presubscript{S_{2},\emptyset}{\pi}^{*}\mathcal{G})^{W_{S_{2}}}
= \presubscript{S_{1}}{\mathbb{H}}_{S_{2}}(\mathcal{F}_{0}*\mathcal{G}_{0}).
\]

\def\cprime{$'$}

\end{document}